\documentclass[a4paper, 11pt, twoside]{article}
\usepackage{hyperref}

\usepackage{amsmath,amssymb}
\usepackage{amsthm}
\usepackage{mathrsfs}
\usepackage[dvipdfmx]{graphicx}

\usepackage[normalem]{ulem}
\usepackage{url}
\usepackage{comment}
\usepackage{bm}
\usepackage{accents}

\usepackage{enumerate}
\usepackage[all]{xy}
\usepackage{mathtools}
\usepackage{graphicx}
\usepackage{overpic}
\usepackage{pdfpages}

\usepackage[margin=1in]{geometry}

\usepackage{fancyhdr}

\usepackage{tikz}
\usepackage{tikz-cd}
\usepackage{pgfplots}

\usetikzlibrary{arrows}
\usetikzlibrary{patterns}

\theoremstyle{definition}
\newtheorem{defi}{Definition}[section]
\newtheorem{rem}[defi]{Remark}
\newtheorem{ex}[defi]{Example}

\newtheorem*{acknow}{Acknowledgements}

\theoremstyle{plain}
\newtheorem{thm}[defi]{Theorem}

\newtheorem{prop}[defi]{Proposition}
\newtheorem{lem}[defi]{Lemma}
\newtheorem{cor}[defi]{Corollary}

\newtheorem{question}{Question}[section]

\newcommand{\Hom}{\operatorname{Hom}}

\renewcommand{\ker}{\operatorname{Ker}}
\newcommand{\coker}{\operatorname{Coker}}

\renewcommand{\Im}{\operatorname{Im}}

\newcommand{\C}{\mathbb{C}}
\newcommand{\R}{\mathbb{R}}

\newcommand{\Z}{\mathbb{Z}}

\newcommand{\SO}{\mathrm{SO}}
\newcommand{\Spin}{\mathrm{Spin}}

\newcommand{\Dvar}{\overline{\partial}}

\newcommand{\id}{\operatorname{id}}

\renewcommand{\tilde}{\widetilde}

\renewcommand{\epsilon}{\varepsilon}

\newcommand{\la}{\langle}
\newcommand{\ra}{\rangle}

\newcommand{\sign}{\operatorname{sign}}

\newcommand{\aug}{\operatorname{aug}}
\newcommand{\Dbar}{\operatorname{\bar{\partial}}}
\newcommand{\Ham}{\operatorname{Ham}}

\newcommand{\topwedge}{{\textstyle{\bigwedge}^{\!\mathrm{max}}}}

\newcommand{\rest}[2]{\left.#1\right|_{#2}}

\newcommand{\Graph}{\operatorname{Graph}}
\newcommand{\gr}{\operatorname{gr}}
\newcommand{\ori}{\mathrm{ori}}

\newcommand\shorttitle{On knot types of clean Lagrangian intersections in $T^*\mathbb{R}^3$}
\newcommand\authors{Yukihiro Okamoto}

\usepackage{authblk}

\AtEndDocument{%
  \par
  \medskip
  \begin{tabular}{@{}l@{}}%
    Department of Mathematical Sciences, Tokyo Metropolitan University, \\
Minami-osawa, Hachioji, Tokyo, 192-0397, Japan \\
    \textit{E-mail address}: \texttt{yukihiro@tmu.ac.jp}
  \end{tabular}}

\begin{document}

\title{On knot types of clean Lagrangian intersections in $T^*\mathbb{R}^3$}
\author{Yukihiro Okamoto }
\date{}
\maketitle

\begin{abstract}
\noindent
Let $K_0$ and $K$ be knots in $\mathbb{R}^3$. Suppose that by a compactly supported Hamiltonian isotopy on $T^*\R^3$, the conormal bundle of $K_0$ is isotopic to a Lagrangian submanifold which intersects the zero section cleanly along $K$. In this paper, we prove some constraints on the pair of knot types of $K_0$ and $K$.
One example is that if $K_0$ is the unknot, then $K$ is also the unknot. We also consider some cases where $K_0$ and $K$ have specific knot types, such as torus knots and connected sums of trefoil knots. The key step is finding a DGA map between the Chekanov-Eliashberg DGAs of the unit conormal bundles of knots. The main results are deduced from a relation between the augmentation varieties of $K_0$ and $K$ determined by these DGAs.
\end{abstract}

\section{Introduction}

\textbf{Background and motivation.}
Let us first consider a closed symplectic manifold $M$ and compact Lagrangian submanifolds $L_0$ and $L_1$ of $M$. We assume that $L_0$ and $L_1$ have a clean intersection $K \coloneqq L_0\cap L_1$ and $K$ is connected.
Then, under several conditions on $M,L_0$ and $L_1$, Po\'{z}niak proved in \cite[Corollary 3.4.13]{P} the following result about the $\Z/2$-coefficient Floer homology:
\[ HF_*(L_0,L_1)\cong  H_*(K; \Z/2).\]
An analogous result in $\Z$-coefficients follows from \cite[Theorem 9.18]{S} by Schm\"{a}schke when the pair $(L_0, L_1)$ has a relative spin structure.

In this paper, we consider exact symplectic manifolds, mainly the cotangent bundle of $\R^n$.
The image of the zero section of $T^*\R^n$ is identified with $\R^n$. 
Let $L_1$ be an exact Lagrangian submanifold of $T^*\R^n$ which bounds a compact Legendrian submanifold of the unit cotangent bundle of $\R^n$. 
We make two assumptions: The first one is that $\R^n$ and $L_1$ have a clean intersection $K=\R^n\cap L_1$, and the second one is that $K$ is connected.
Under some conditions on $L_1$, we can prove in the same way as in \cite{P} a result about the $\Z/2$-coefficient Floer cohomology
\[  HF^*(\R^n, L_1) \cong H^*(K;\Z/2).\]
For the proof, see Proposition \ref{prop-Floer}.
This gives a topological constraint on the manifold $K$.
To go beyond this result, we ask \textit{how} $\mathit{K}$ \textit{is embedded as a submanifold of} $\mathit{\R^n}$.

More specifically, suppose that $n=3$ and $L_1$ is isotopic to the conormal bundle $L_{K_0}$ of a knot $K_0$ in $\R^3$ by a compactly supported Hamiltonian isotopy.
Then, the invariance of the Floer cohomology by such an isotopy shows that $H^*(K;\Z/2)\cong H^*(K_0;\Z/2)$, which implies that $K$ is also a knot in $\R^3$.
Then, we ask a question;
\begin{question}\label{question-1}
Can  the knot type of $K$ be changed from the knot type of $K_0$?
\end{question}
At present, the author does not know an example that $K$ and $K_0$ are not isotopic.
Here, it is essential that the isotopy is a Hamiltonian isotopy. In fact, for an arbitrary pair of knots $(K_0, K)$, there exists a compactly supported $C^{\infty}$ isotopy $(\varphi^t)_{t\in [0,1]}$ on $T^*\R^3$ such that $\varphi^1 (L_{K_0})$ intersects $\R^3$ cleanly along $K$. See Proposition \ref{prop-smooth-isotopy}.

In fact, our setting in $T^*\R^3$ can be embedded into a setting where I. Smith posed a question. We discuss his question and the preceding results about it at the end of this section.

\

\noindent
\textbf{Notation.}
Throughout this paper, we abbreviate singular homology group over the integers $H_*(X ;\Z)$ for any topological space $X$ by $H_*(X)$.

\

\noindent
\textbf{Main results.}
In \cite{N}, Ng defined the \textit{augmentation variety} $V_K$ for any knot $K$ in $\R^3$. It is a knot invariant given as a subset of $(\C^*)^2$. See Definition \ref{defi-aug} using the conventions of \cite{N-intro}.  In the situation of Question \ref{question-1} where $\varphi(L_{K_0})$ and $\R^3$ have a clean intersection along $K$, we prove the following result about the knots $K_0$ and $K$.
\begin{thm}[Theorem \ref{cor-aug}]\label{thm-intro-aug} There exist $a,b\in \Z$ with $a \equiv 1 \mod 2$ such that the set
\[\{( x' ,y') \in (\C^*)^2\mid \text{ there exists } (x,y)\in V_{K} \text{ such that } (x')^a = x y^{-b} \text{ and }y'= y^{ a} \} \]
is a subset of $V_{K_0}$.
\end{thm}

Let us see how this theorem is deduced from \textit{Symplectic Field Theory} (\textit{SFT}) introduced by Eliashberg, Givental and Hofer \cite{EGH}.
As a notation, let $U^*\R^3$ denote the unit cotangent bundle of $\R^3$ with respect to the standard metric.
For any knot $K$ of $\R^3$, let $\Lambda_{K}$ denote its unit conormal bundle. Then, $U^*\R^3$ has a canonical contact structure and $\Lambda_K$ is a Legendrian submanifold of $U^*\R^3$.

We focus on the \textit{Chekanov-Eliashberg differential graded algebra} (\textit{DGA}) of a Legendrian submanifold of a contact manifold.
This is an algebraic invariant of Legendrian submanifolds introduced by Chekanov \cite{Chek} for Legendrian links and by Eliashberg \cite{Eli} in the higher dimensional case. For the construction under some conditions, we will refer to \cite{DR, EES-R,EES-ori,EES}.
We use a result by Ekholm, Etnyre, Ng and Sullivan \cite{EENS} that the Chekanov-Eliashberg DGA of $\Lambda_K$ for a knot $K$ in $\R^3$ is stable tame isomorphic to the \textit{framed knot DGA}, which is combinatorially defined from a braid representation of $K$. By definition, the augmentation variety $V_K$ is determined by the framed knot DGA of $K$.

The relation between $V_{K_0}$ and $V_K$ follows from a DGA map between the Chekanov-Eliashberg DGAs associated to an exact Lagrangian cobordism which has $\Lambda_{K_0}$ as the positive boundary and $\Lambda_{K}$ as the negative boundary.
We state a more general result in Theorem \ref{thm-clean-intersection} for an exact Lagrangian submanifold in $T^*\R^n$ which has a clean intersection with the zero section.

The DGA maps in SFT have been studied by, for instance, \cite{E, EHK} in $\Z/2$-coefficient and \cite{K} in $\Z$-coefficient.
For the sake of our applications, we want to deal with DGAs with coefficients in the integral group rings, and therefore we need to care about the orientations of the moduli spaces of pseudo-holomorphic curves.

\begin{rem}
For any knot $K$ in $\R^3$, it is known that the Chekanov-Eliashberg DGA of $\Lambda_K$ contains much information about the knot type of $K$. As a matter of fact, Ekholm-Ng-Shende \cite[Theorem 1.1]{ENS} shows that an enhanced version of the homology of the Chekanov-Eliashberg DGA of $\Lambda_K$ is a complete knot invariant for $K$ up to taking the mirror of $K$.

\end{rem}

Combining Theorem \ref{thm-intro-aug} with the computations of the augmentation varieties of knots,
we can find constraints on the pair of knot types of $K_0$ and $K$.
The first result is about the unknot. In fact, this follows from a result by Ganatra-Pomerleano \cite{GP} as we will see below. However, our method using Theorem \ref{thm-intro-aug} is different from theirs.

\begin{thm}[Proposition \ref{cor-unknot}]\label{thm-intro-unknot}
If $K_0$ is the unknot, then $K$ is also the unknot.
\end{thm}
We prove this theorem by using the arguments in the proof of \cite[Proposition 5.10]{N}, which proves that the unknot is detected by the cord algebra. 

The second result is about the right-handed trefoil knot $3_1$ and its mirror $\overline{3_1}$, and more generally, their connected sums. For $m,m'\in \Z_{\geq 0}$, let $K_{m,m'}\coloneqq (3_1)^{\# m}\# (\overline{3_1})^{\# m'}$ denote the connected sum of $m$ $3_1$ knots and $m'$ $\overline{3_1}$ knots.
\begin{thm}[Proposition \ref{prop-trefoil}]\label{thm-intro-trefoil}
Suppose that $K_0=K_{m_0,m'_0}$ and $K=K_{m,m'}$. Then, $m_0+m'_0\geq m+m'$. Moreover, if $m_0+m'_0=m+m'$, then $m_0\equiv m  \mod 2$.
In particular, if $K_0=3_1$, then, $K\neq \overline{3_1}$.
\end{thm}

The third result is about the torus knots. For $p,q\in \Z$ which are coprime and satisfy $1\leq p < |q|$, let $T_{(p,q)}$ denote the $(p,q)$-torus knot.
\begin{thm}[Proposition \ref{prop-torus}]\label{thm-intro-torus}
Suppose that $K_0=T_{(p_0,q_0)}$ and $K=T_{(p,q)}$. Then, $p_0\geq p$. Moreover, if $p_0=p\geq 2$, then $q_0= q$.
\end{thm}

\noindent
\textbf{Relation to a question by I.  Smith and preceding results}

In \cite{GP} and \cite{SW}, the following question by Smith, concerning the ``persistence (or rigidity) of unknottedness of Lagrangian intersections'', was considered. We refer to \cite[Question 1.5]{GP}.
\begin{question}[Smith]
In a symplectic manifold $X$ of real dimension $6$, if a pair of Lagrangian $3$-spheres $S_1$ and $S_2$ meet cleanly along a circle, can the isotopy class of this knot (in $S_1$ and $S_2$) change under (nearby) Hamiltonian isotopy?
\end{question}
In \cite{GP}, Ganatra-Pomerleano answered this question negatively when $X$ is a plumbing $W_1$ of $T^*S_1$ and $T^*S_2$ along the unknots in both $S_1$ and $S_2$. For the definition of the plumbings $W_n$ for $n\in \Z_{\geq 0}$, see \cite[Subsection 6.4]{GP}.
\begin{thm}[Proposition 6.29 of \cite{GP}]\label{thm-Smith}
For $\tilde{S}_i$ ($i=1,2$) which is isotopic to $S_i$ by a Hamiltonian isotopy on $W_1$, suppose that $\tilde{S}_1$ and $\tilde{S}_2$ intersect cleanly along a knot. Then, this knot must be the unknot in both $\tilde{S}_1$ and $\tilde{S}_2$.
\end{thm}
We should note that there was a study preceding \cite{GP} by Evans-Smith-Wemyss showing that $\tilde{S}_1\cap \tilde{S}_2$ is the unknot in one component and is the unknot or the trefoil knot in the other component. See \cite[Proposition 6.28]{GP}.
In addition, Theorem \ref{thm-Smith} still holds if we replace $W_1$ by $W_0$. See \cite[Proposition 4.10]{SW} by Smith-Wemyss.

A brief explanation of the proof in \cite{GP} is as follows: A crucial step is \cite[Proposition 6.32]{GP} which shows that the plumbing $W_1$ admits a \textit{dilation} in the symplectic cohomology (with coefficients in a field of characteristic $3$).
Then, the existence of a dilation leads to a strong topological constraint on closed exact Lagrangian submanifolds in $W_1$, especially on the Lagrangian submanifold by a Polterovich surgery of $\tilde{S}_1\cup \tilde{S}_2$.

We return to our setting in $T^*\R^3$ and consider $L_1= \varphi^1(L_{K_0})$, where $K_0$ is the unknot and $(\varphi^t)_{t\in [0,1]}$ is a compactly supported Hamiltonian isotopy.
As a notation, for any Riemannian manifold $M$, let $D^*_aM$ denote the disk bundle of $T^*M$ of radius $a>0$. By an open embedding of $\R^3$ into $S_1$, $T^*\R^3$ is identified with an open subset of $T^*S_1$. 
We take a Riemannian metric on $S_1$.
Suppose that the support of $\varphi^t$ for every $t\in[0,1]$ is contained in $D^*_r S_1$ for some $r>0$, and $K_0$ has a tubular neighborhood $N_1$ in $S_1$ of radius $\epsilon>0$.
For $S_2$, we also consider the unknot $K'_0$ and take a Riemannian metric so that $K'_0$ has a tubular neighborhood $N_2$ of radius $r$.
Using the exponential maps on $S_1$ and $S_2$ with respect to the Riemannian metrics, $N_1$ is identified with the disk bundle of the normal bundle $\nu_{K_0}$ of $K_0\subset S_1$ of radius $\epsilon$, and $N_2$ is identified with the disk bundle of the normal bundle $\nu_{K'_0}$ of $K'_0\subset S_2$ of radius $r$.

We may suppose that there is a diffeomorphism $\bar{\eta}\colon K_0\to K'_0$ such that the induced isomorphism $T^*K'_0\to T^*K_0$ maps $D^*_{\epsilon}K'_0$ to $D^*_rK_0$.
Given an isomorphism $\eta\colon \nu_{K_0}\to \nu_{K'_0}$ which is a lift of $\bar{\eta}$ and preserve the fiber metrics, we define a plumbing $W_{\eta}$ by gluing $D^*_r S_1$ and $D^*_{\epsilon} S_2$ in the way as in \cite[Subsection 6.4]{GP} by replacing the unit disk bundles of $S_1$ and $S_2$ by $D^*_r S_1$ and $D^*_{\epsilon} S_2$.
Note that there exists a natural embedding $D^*_r S_1 \to W_{\eta}$ by which $\R^3$ is embedded into $S_1$ as we did and $L_{K_0} \cap D^*_r S_1$ is embedded into $S_2$. In this case, $K_0$ is diffeomorphically mapped onto $S_1\cap S_2$. 
The completion of $W_{\eta}$ as a Liouville domain is isotopic to the completion of one of the plumbings $(W_n)_{n=0,1,\dots}$, and we may take $\eta$ so that it is isotopic to the completion of $W_1$. For such $\eta$, the symplectic cohomology of $W_{\eta}$ admits a dilation, and thus the same statement as Theorem \ref{thm-Smith} still holds if we replace $W_1$ by $W_{\eta}$. Then, Theorem \ref{thm-intro-unknot} is a direct consequence of this result.

\

\noindent
\textbf{Organization of paper.}
In Section \ref{sec-clean}, we will see how to find a Lagrangian cobordism from a Lagrangian submanifold in the cotangent bundle which intersects the zero section cleanly.
We will also study in Subsection \ref{subsec-Floer} what the Floer cohomology tells us about clean Lagrangian intersections.
In Section \ref{sec-DGA}, we review the construction of the Chekanov-Eliashberg DGA with coefficients in the integral group ring, and define a DGA map associated to an exact Lagrangian cobordism.
Some technical arguments, especially on the orientations of moduli spaces are left to Appendix \ref{sec-proof}.
In Section \ref{sec-knot}, we introduce the knot DGA and the augmentation variety of knots and prove Theorem \ref{thm-intro-aug}.
By using known results about the augmentation variety, we deduce Theorem \ref{thm-intro-unknot}, \ref{thm-intro-trefoil} and \ref{thm-intro-torus}.
In Appendix \ref{sec-proof}, we fix orientations of moduli spaces and prove Theorem \ref{thm-DGA-map} about the DGA map over the integral group rings.

\begin{acknow}
The author would like to thank his supervisor Kei Irie for reading the draft and making valuable comments, and Tomohiro Asano for helpful discussion on Proposition \ref{prop-Floer}.
He also appreciates the referees for informing him of the results of \cite{GP, SW} and for many comments which correct and improve the arguments in the article, especially in Appendix \ref{sec-proof}.
This work was supported by JST, the establishment of university fellowships towards the creation of science technology innovation, Grant Number JPMJFS2123.
The author is also supported by JSPS KAKENHI Grant Number JP23KJ1238.
\end{acknow}

\section{Clean intersection in the cotangent bundle}\label{sec-clean}
In this section, we consider a Lagrangian submanifold in the cotangent bundle $T^*Q$ which have a clean intersection with the zero section. We associate to it a Lagrangian cobordism in the symplectization of the unit cotangent bundle. We also study a relation between the Floer cohomology group of  Lagrangian submanifolds in $T^*Q$ and  the singular cohomology group of the clean intersection with coefficients in $\Z/2$.

\subsection{Preliminaries}\label{subsec-pre}
Let $W$ be a symplectic manifold with a symplectic form $\omega$.
Given a Hamiltonian $H\colon W\times [0,1]\to \R$, a time-dependent Hamiltonian vector field $(X_t)_{t\in [0,1]}$ is defined by
\[\omega ( \cdot, X_t) = d (H(\cdot ,t)). \]
If $(X_t)_{t\in [0,1]}$ generates a flow $(\varphi^t_H)_{t\in [0,1]}$, we denote the Hamiltonian diffeomorphism $\varphi^1_H$ by $\varphi_H$.
Let $\Ham_c(W)$ denote the group of compactly supported Hamiltonian diffeomorphisms on $W$.

We define the notion of clean intersections of two submanifolds.
\begin{defi}
Let $L_0$ and $L_1$ be two submanifolds of $W$.
We say that $L_0$ and $L_1$ have a \textit{clean intersection}, if $L_0\cap L_1$ is a submanifold of both $L_0$ and $L_1$, and
\[T_x (L_0\cap L_1) = T_xL_0\cap T_xL_1 \]
holds for every $x\in L_0\cap L_1$.
\end{defi}
We are mainly interested in the case where both $L_0$ and $L_1$ are Lagrangian submanifolds of $W$.
Let us consider a linear model in $\C^n$ with a symplectic form $\omega_0= \sum_{i=1}^n dx_i\wedge dy_i$ and a metric $\sum_{i=1}^n (dx_i\otimes dx_i +dy_i \otimes dy_i)$. Here, $(x_i+\sqrt{-1}y_i)_{i=1,\dots ,n}$ is the coordinate of $\C^n$.
For any $\R$-subspace $V$ of $\R^n$ and its orthogonal complement $V^{\perp}$ in $\R^n$, the two Lagrangian subspaces
\begin{align*}
L_0= V\oplus \sqrt{-1}V^{\perp},\ L_1 =\R^n
\end{align*}
have a clean intersection $V=L_0\cap L_1$.
For these Lagrangian subspaces of $\C^n$, the following is standard. 
\begin{lem}\label{lem-linear}
For any Lagrangian subspace $L$ in $\C^n$ with $L_0 \cap L =V$, there exists a symmetric operator $S\colon V^{\perp} \to V^{\perp}$ such that $L= \{(v+w)+ \sqrt{-1} S w \mid v\in V,\ w\in V^{\perp} \}$. In particular, the projection $L\to L_1\colon x+ \sqrt{-1} y \mapsto x$ is an isomorphism.
\end{lem}
\begin{proof}
Let $X\subset L$ be the orthogonal complement of $V$ in $L$. Then $X$ is a Lagrangian subspace of $V^{\perp}\oplus \sqrt{-1}V^{\perp}$, since any vector $x\in X$ satisfies $ \omega_0(x, v) = 0= \omega_0(x,\sqrt{-1} v)$ for all $v\in V$. The condition that $L\cap L_1=V$ implies that $X\cap \sqrt{-1} V^{\perp}=0$, and thus there exists a symmetric operator $S$ on $V^{\perp}$ such that $X= \{w+\sqrt{-1}S w \mid w\in V^{\perp}\}$. This proves the lemma.
\end{proof}

\subsection{Lagrangian cobordism arising from Lagrangian submanifolds with clean intersection}\label{subsec-clean}

Let $Q$ be an arbitrary manifold. In the cotangent bundle $T^*Q$, we identify the image of the zero section with $Q$.
For any submanifold $K$ of $Q$, let $L_K$ denote its conormal bundle. We fix a Riemannian metric on $Q$ and
 define $D_r^*Q\coloneqq \{(q,p)\in T^*Q \mid |p|<r \}$ and $\overline{D^*_rQ} \coloneqq \{(q,p)\in T^*Q \mid |p|\leq r \}$ for any $r>0$.

Let $\lambda_Q$ be the canonical Liouville  $1$-form on $T^*Q$.
More precisely, if $(q_1,\dots ,q_n)$ is a local coordinate of $Q$ and $(p_1,\dots ,p_n)$ is the associated coordinate of the fiber in $T^*Q$, then $\lambda_Q=\sum_{i=1}^n p_idq_i$.

\begin{lem}\label{lem-neighborhood}
Let $L_1$ be a Lagrangian submanifold of $T^*Q$ such that $L_1 \cap \overline{ D^*_r Q}$ is compact for some $r>0$. Suppose that $Q$ and $L_1$ have a clean intersection. Let us write $K=Q\cap L_1$.
Then, there exists $\psi \in \Ham_c(D^*_rQ)$ and $\epsilon\in (0,r]$ such that:
\begin{itemize}
\item $\psi(x)=x$ for every $x\in Q$.
\item $\psi (L_1) \cap D^*_{\epsilon}Q = L_K \cap D^*_{\epsilon}Q$.
\end{itemize}
\end{lem}
\begin{proof}We denote $N\coloneqq D^*_{\epsilon_1}Q \cap L_K$ for some $\epsilon_1\in (0,r]$. Let us fix a Riemannian metric on $L_K$.
Since $N$ and $Q$ intersects cleanly along $K\subset N$, \cite[Proposition 3.4.1]{P} shows the following: If $\epsilon_1$ is sufficiently small, then there exist $\epsilon_2>0$ and an open symplectic embedding
\[\Psi \colon D_{\epsilon_2}^*N \to T^*Q\]
such that $\Psi(x)=x$ for every $x\in N$ and $\Psi^{-1}(Q)=L'_K \cap D^*_{\epsilon_2}N$,
where $L'_K$ is the conormal bundle of $K$ in $T^*N$.
Then, $\Psi^{-1}(L_1)$ is a Lagrangian submanifold of $T^*N$ which has a clean intersection with $L'_K$ along $K$.

Let $\pi\colon \Psi^{-1}(L_1) \to N$ be the restriction of the bundle projection $\pi_N\colon T^*N \to N$. We apply Lemma \ref{lem-linear} to the projection $(d\pi)_x$ for any $x\in K$, then it follows that $\pi$ is a submersion to $N$ on a neighborhood of $K$.
Therefore, there exists $\epsilon'_1 \in (0,\epsilon_1]$ and a closed $1$-form $\eta$ on $N'\coloneqq D_{\epsilon'_1}^*Q\cap L_K\subset N$ such that the graph of $\eta$ coincides with $\pi^{-1}(N')$.
Moreover, since $\eta_x=0$ for every $x\in K$, there exists a $0$-form $h$ on $N'$ such that $\eta=d h$ and $h(x)=0$ for every $x\in K$.

Now we take a time-independent Hamiltonian on $D^*_{\epsilon_2}N'$
\[H\coloneqq \rho \cdot (h\circ \pi_N) \colon D_{\epsilon_2}^*N'\to \R\colon (q,p) \mapsto  \rho(q,p)h(q),\]
where $\rho \colon  D_{\epsilon_2}^*N'\to \R$ is a $C^{\infty}$ function with a compact support such that $\rho \equiv 1$ on a neighborhood of $K$.
Note that the Hamiltonian isotopy $(\varphi^t_{H})_{t\in [0,1]}$ for $H$ fixes every point in $L'_K$ since
\[(dH)_x=(d\rho)_xh(q) + \rho(x)  (\pi_N^*(dh)_q) = 0\]
for every $x=(q,p)\in L'_K$.
In addition, $\varphi^1_H$ maps a neighborhood of $K$ in $\Psi^{-1}(L_1)$, which is also a neighborhood of $K$ in the graph of $\eta =dh$, to a neighborhood of $K$ in $N$. 
Finally, we extend on $D^*_rQ$ the Hamiltonian $H\circ \Psi^{-1}\colon \Psi(D_{\epsilon_2}^*N')\to \R$ to be $0$ outside $\Psi(D^*_{\epsilon_2}N')$.
We define $\psi$ as the time-$1$ map of its Hamiltonian isotopy. Then the above observations on $\varphi^1_H$ show that $\psi$ satisfies the two conditions if $\epsilon \in (0,\epsilon'_1]$ is sufficiently small.
\end{proof}

We denote the unit cotangent bundle of $Q$ by $U^*Q \coloneqq \{(q,p)\in T^*Q \mid |p|=1 \}$.
This manifold has a canonical contact form $\alpha_Q\coloneqq \rest{\lambda_Q}{U^*Q}$.

Let $L_1$ be an exact Lagrangian submanifold of $T^*Q$ such that $L_1 \cap \overline{D^*_rQ}$ is compact and $L_1 \setminus D^*_rQ = \{(q,r' p) \mid (q,p)\in \Lambda,\ r' \geq r \}$ for some $r>0$ and a compact Legendrian submanifold $\Lambda$ of $U^*Q$.
Suppose that $Q$ and $L_1$ have a clean intersection along $K= Q \cap L_1$. If
we take $\psi \in \Ham_c(D^*_rQ)$ and $\epsilon\in (0,r]$ of Lemma \ref{lem-neighborhood}, $\psi (L_1)$ is an exact Lagrangian submanifold of $T^*Q$ such that $\psi(L_1) \cap D^*_{\epsilon}Q = L_K \cap D^*_{\epsilon}Q$.
Therefore, we have the embedding
\begin{align}\label{inclusion-clean}
L_K \cap D^*_{\epsilon}Q \to L_1 \colon x \mapsto \psi^{-1}(x).
\end{align}

Let $\Lambda_K\coloneqq L_K \cap U^*Q$ denote the unit conormal bundle of $K$. This is a compact Legendrian submanifold of $U^*Q$.
We identify $T^*Q\setminus Q$ with the symplectization of $U^*Q$ by the diffeomorphism
\begin{align}\label{diffeo-symplectization}
F\colon \R\times U^*Q  \to T^*Q \setminus Q \colon (a, (q,p)) \mapsto (q, e^a p),
\end{align}
for which $F^*\lambda_Q = e^a \alpha_Q$ holds.
Then, we obtain an exact Lagrangian submanifold of $(\R\times U^*Q, e^a \alpha_Q)$
\begin{align}\label{induced-cobordism}
L \coloneqq F^{-1} (\psi (L_1) \setminus K )
\end{align}
satisfying the following conditions for $a_+=\log r$ and $a_-=\log \epsilon$:
\begin{itemize}
\item $L \cap ( [a_-,a_+]\times U^*Q)$ is compact. 
\item $L \cap ( \R_{\geq a_+} \times U^*Q) = \R_{\geq a_+} \times \Lambda$ and
$L \cap ( \R_{\leq a_-} \times U^*Q) = \R_{\leq a_-} \times \Lambda_{K}$.
\end{itemize}
The second condition implies that any primitive function $f\colon L \to \R$ with $df = \rest{(e^a \alpha_Q)}{L}$ satisfies $df=0$ on $( \R_{\geq a_+} \times \Lambda) \sqcup (\R_{\leq a_-} \times \Lambda_K)$.
We should note that on $\R_{\leq a_-} \times \Lambda_K$, $f$ is locally constant, but not necessarily constant.
This is related to the third condition of Definition \ref{def-cobordism}.


\subsection{Floer cohomology of Lagrangian submanifolds with clean intersection}\label{subsec-Floer}

We continue to consider a Lagrangian submanifold $L_1$ in $T^*Q$.
To simplify the discussion in this subsection,
we only deal with the case where $Q$ is an oriented closed manifold with $H_1(Q)=0$ or $Q=\R^n$ with a fixed orientation.
Moreover, we assume that the Maslov class of $L_1$ vanishes. For the definition of the Maslov class, see Remark \ref{rem-Maslov} below.
\begin{rem}\label{rem-Maslov}
In general, for a Lagrangian submanifold $L$ of a symplectic manifold $W$, the Maslov index is defined for every homology class in $H_2(W,L)$. See \cite[Definition 13.2.8]{Oh}.
When $2c_1(TW)=0$ and $H_1(W)=0$, for instance $W=T^*Q$, this is reduced to a cohomology class $\mu_L\in H^1(L)\cong \Hom (H_1(L),\Z)$, called the \textit{Maslov class} of $L$.
\end{rem}

\subsubsection{Gradings and orientations of Lagrangian submanifolds in $T^*Q$}

We shall define the notion of \textit{gradings} of Lagrangian submanifolds of $T^*Q$. One can refer to \cite[Section (11j), (12a)]{Seidel}.
Fix a Riemannian metric on $Q$ in order to identify $T_qQ$ with $T^*_qQ$ for every $q \in Q$.
For every $(q,0)\in T^*Q$ in the image of the zero section $Q$, we have an isomorphism between $\R$-vector spaces
\begin{align}\label{isom-zero-sec}
T_{(q,0)}(T^*Q)\cong T_qQ \otimes_{\R}\C
\end{align}
such that the differential $T_qQ \to T_{(q,0)}(T^*Q)$ of the zero section $Q \to T^*Q$ (resp. of the inclusion map $T_q Q \cong T^*_qQ \to T^*Q$ into the fiber of $q$) coincides via (\ref{isom-zero-sec}) with  the $\R$-linear map $T_qQ \to T_qQ \otimes_{\R}\C\colon v \mapsto v$ (resp. $v \mapsto \sqrt{-1}v$).
We choose an almost complex structure $J_0$ on $T^*Q$ such that $d\lambda_Q(\cdot,J_0\cdot)$ is a Riemannian metric and  the isomorphism (\ref{isom-zero-sec}) is a linear map over $\C$.

Since $Q$ is oriented,
we can take a trivialization of the $\C$-line bundle
\[ \beta \colon \textstyle{\bigwedge}^n_{\C}T(T^*Q) \to \C\]
satisfying the following condition: By (\ref{isom-zero-sec}), $\rest{\textstyle{\bigwedge}^n_{\C}T(T^*Q)}{Q} $ is isomorphic to $(\bigwedge^n_{\R}TQ)\otimes_{\R} \C$ when restricted on $Q\subset T^*Q$. We require that $\beta$ coincides on $Q$ with the trivialization
\[ (\textstyle{\bigwedge}^n_{\R}TQ) \otimes_{\R} \C\to \C \]
which maps $(v_1\wedge \dots \wedge v_n)\otimes 1$ to $1\in \C$ for any $q\in Q$ and any positive orthonormal basis $v_1,\dots ,v_n$ of $T_qQ$ with respect to the orientation of $Q$.

Let $\mathcal{L}^{\ori}(T^*Q)$ denote the fiber bundle over $T^*Q$ whose fiber at each $x\in T^*Q$ consists of all oriented Lagrangian subspaces of $T_x(T^*Q)$. 
Then, we define
\[\det \colon \mathcal{L}^{\ori} (T^*Q) \to S^1=\{z\in \C \mid |z|=1\} \]
as follows: For any oriented Lagrangian subspace $V\subset T_{x} (T^*Q)$, we choose its positive basis $v_1,\dots ,v_n$. Then, 
\[ \det (V) \coloneqq \frac{\beta(v_1\wedge \dots \wedge v_n)}{ |\beta(v_1\wedge \dots \wedge v_n)| } \in S^1.\]
It is easy to check that $\det$ is a well-defined continuous map. 
\begin{rem}
We note that for any $V\in \mathcal{L}^{\ori}(T^*Q)$, $(\det(V))^2\in S^1$ is independent of the orientation of $V$. This defines a map  $\det^2\colon \mathcal{L}(T^*Q) \to S^1$, where  $\mathcal{L}(T^*Q)$ denote the fiber bundle over $T^*Q$ whose fiber at each $x\in T^*Q$ consists of all non-oriented Lagrangian subspaces of $T_x(T^*Q)$. This agrees with the \textit{associated squared phase} map in \cite[Section (12a)]{Seidel}.
\end{rem}
\begin{defi} Let $L$ be an orientable Lagrangian submanifold of $T^*Q$.
A $C^{\infty}$ function $\gr \colon L \to \R$ is called a \textit{grading} of $L$ if 
\[  (\det(T_x L))^2 = e^{2\pi \sqrt{-1} \gr(x)} \]
for every $x\in L$. A pair $(L, \gr)$ is called a \textit{graded Lagrangian submanifold} of $T^*Q$.
For any graded Lagrangian submanifold $\tilde{L}=(L,\gr)$ of $T^*Q$, we define its orientation $\mathfrak{o}_{\tilde{L}}$ so that
\[\det (T_xL) = e^{\pi \sqrt{-1} \gr(x)}\]
for every $x\in L$ when $T_xL$ is oriented by $\mathfrak{o}_{\tilde{L}}$.
\end{defi}
The obstruction to the existence of a grading is the Maslov class of $L$. See \cite[Lemma 2.6]{Seidel-gr}.

An isotopy of a graded Lagrangian submanifold is defined as follows:
For any Hamiltonian flow $(\varphi^t_H)_{t\in [0,1]}$ of a Hamiltonian $H$ on $T^*Q$ and any grading $\gr \colon L \to \R$ of $L$, there exists a unique $1$-parameter family of gradings $(\gr_t \colon \varphi^t_H(L) \to \R)_{t\in [0,1]}$ of $\varphi^t_H(L)$ for every $t\in [0,1]$ such that $\gr_0 =\gr$ and the function $L\times [0,1]\to \R\colon (x,t)\mapsto \gr_t( \varphi^t_H(x))$ is continuous.
For the graded Lagrangian submanifold $\tilde{L}= (L,\gr)$, let us use a notation
\[\varphi_H(\tilde{L})  \coloneqq  (\varphi^1_H(L), \gr_1) . \]
We note that the diffeomorphism $\varphi_H \colon L \to \varphi_H(L)$ preserves the orientations $\mathfrak{o}_{\tilde{L}}$ and $\mathfrak{o}_{\varphi_H(\tilde{L})}$.

\begin{ex}\label{ex-gr}
From the definition of $\beta$ on $Q\subset T^*Q$, $\det(T_q Q) = 1$ for every $q\in Q$.
We define the grading of $Q$ by the constant map
\[\gr_Q \colon Q \to \R\colon q \mapsto 0 \]
and denote $(Q,\gr_Q)$ by $\tilde{Q}$.
For the conormal bundle $L_K$ of a submanifold $K$ in $Q$ of codimension $d$,
$\det(T_xL_K)= (\sqrt{-1})^{d}$ for every $x\in K= Q \cap L_K$. Here, we fix the orientation of $L_K$ by orientating a neighborhood of $K$ in $L_K$ as follow: It is oriented via the identification with a tubular neighborhood of $K$ in $Q$ given by the exponential map $T^*Q\cong TQ \to Q\colon (x,v) \mapsto \exp_xv$.
Let 
\[\gr_{L_K}\colon L_{K} \to \R \]
be the grading of $L_K$ such that $\gr_{L_K}(x) = \frac{d}{2}$ for every $x\in Q \cap L_K$ and denote $(L_K,\gr_{L_K})$ by $\tilde{L}_K$.
Then, $\mathfrak{o}_{\tilde{Q}}$ on $Q$ and $\mathfrak{o}_{\tilde{L}_K}$ on $L_K$ agree with the fixed orientations of $Q$ and $L_K$ respectively.
\end{ex}

\subsubsection{Definition of Floer cohomology of graded Lagrangian submanifolds}

As a notation, we set $D\coloneqq \{z\in \C \mid |z|\leq 1\}$ and
\[D_2\coloneqq D\setminus \{1,-1\},\ D_3\coloneqq D\setminus \{1,e^{2\pi\sqrt{-1}/3}, e^{-2\pi\sqrt{-1}/3}\} .\]
Then, $\partial D_m=\bigcup_{i=0}^{m-1}\partial_i D_m$ ($m=2,3$) for
$\partial_i D_m\coloneqq \{e^{2\pi \sqrt{-1}\theta} \mid (i-1)/m<\theta < i/m\}$.

Let us review briefly the construction of the $\Z$-graded Floer cohomology groups
\[HF^*(\tilde{L_1},\tilde{L_1}),\ HF^*(\tilde{Q},\tilde{L_1})\]
for any graded Lagrangian submanifold $\tilde{L_1}=(L_1,\gr)$ with coefficients in $\Z/2$. For the detail, the readers may refer to \cite{F, Seidel-gr, Seidel}.

First, let $J$ be an almost complex structure such that $d\lambda_Q(\cdot , J\cdot)$ is a Riemannian metric  on $T^*Q$ and $\frac{1}{2}d |p|^2\circ J = \lambda_Q$ outside $D^*_{r+1}Q$. When $Q=\R^n$, we additionally assume that $\frac{1}{2}d|q|^2\circ  J = -\sum_{i=1}^n q_i dp_i$ for $q=(q_1,\dots ,q_n)\in \R^n$ when $|q|\in \R_{\geq 0}$ is sufficiently large. These conditions are necessary to prove $C^0$-boundedness for $J$-holomorphic curves in $T^*Q$. See \cite[\S 2.10.3]{GPS}.

We take a Hamiltonian $H$ on $T^*Q$ such that $L_1$ and $\varphi_H(L_1)$ intersect transversely and that it has the form $w|p|$ outside $D_r^*Q$, where $w\in \R_{>0}$ is smaller than the action of any Reeb chord of $\Lambda$. In other words,
\[ 0< w< \inf \{T \in \R_{>0} \mid \text{There exists }x \in \Lambda \text{ such that }\Phi^T(x)\in \Lambda\},\]
where $(\Phi^t)_{t\in \R}$ is the geodesic flow on $U^*Q$. 
Similarly,  we take a Hamiltonian $G$ on $T^*Q$ such that $Q$ and $\varphi_G(L_1)$ intersect transversely and that it has the form $w|p|$ outside $D^*_rQ$ for some $w\in \R_{\geq 0}$. 

In both cases, the degree in $\Z$ of every transversal intersection point
\[\begin{array}{cc} L_1\cap \varphi_H(L_1) \to \Z, & Q \cap \varphi_G(L_1)\to \Z \end{array}\]
is defined as in \cite[Section (12b)]{Seidel} depending on the gradings of $\tilde{L_1}$, $\varphi_H(\tilde{L_1})$, $\tilde{Q}$ and $\varphi_G(\tilde{L_1})$.

Then, $HF^*(\tilde{L_1},\tilde{L_1})$ is the cohomology of a $\Z$-graded cochain complex $CF^*(\tilde{L_1},\varphi_{H}(\tilde{L_1}))$ with coefficients in $\Z/2$ spanned by the set $L_1\cap \varphi^1_H(L_1)$.
The differential on the cochain complex is defined so that the image of any $x\in L_1\cap \varphi_H(L_1)$ is determined by the modulo $2$ count of the number of rigid $J$-holomorphic curves $u\colon D_2\to T^*Q$ such that $u(\partial_0 D_2)\subset L_1$, $u(\partial_1 D_2)\subset \varphi_H(L_1)$ and $\lim_{z\to -1} u(z) =x$ . Here, we need to take generic $J$ so that $(H,J)$ is regular, that is, the moduli spaces of those $J$-holomorphic curves are cut out transversely.

Likewise, $HF^*(\tilde{Q},\tilde{L_1})$ is the cohomology of a $\Z$-graded cochain complex $CF^*(\tilde{Q},\varphi_G(\tilde{L_1}))$ with coefficients in $\Z/2$ spanned by the set $Q\cap \varphi_G(L_1) $. 
When $(G,J)$ is regular, the differential on the cochain complex is defined so that the image of any $x\in Q\cap \varphi_G(L_1)$ is determined by the modulo $2$ count of the number of rigid $J$-holomorphic curves $u\colon D_2\to T^*Q$ such that $u(\partial_0 D_2)\subset Q$, $u(\partial_1 D_2)\subset \varphi_G(L_1)$ and $\lim_{z\to -1} u(z) =x$.

If we shift the grading of $L_1$ by $\sigma\in \Z$, then for $\tilde{L_1}[\sigma]\coloneqq (L_1,\gr +\sigma)$, 
\[HF^*(\tilde{Q}, \tilde{L_1}[\sigma])= HF^{*+\sigma}(\tilde{Q}, \tilde{L_1}).\]
See \cite[Section (12e)]{Seidel}.


Let us also consider the triangle product.
For  Hamiltonians $H'$ and $H''$ on $T^*Q$, we abbreviate $\varphi_{H'}(\tilde{L_1})$ and $\varphi_{H''}(\tilde{L_1})$ by $\tilde{L}'$ and $\tilde{L}''$ respectively, whose underlying Lagrangian submanifolds are denoted by $L'$ and $L''$. If we take $H'$ and $H''$ properly as above, the cochain complexes $CF^*(\tilde{L}',\tilde{L}'')$, $CF^*(\tilde{Q},\tilde{L}')$ and $CF^*(\tilde{Q},\tilde{L}'')$ are defined.
As defined in \cite[Section (12f)]{Seidel}, we have a chain map
\begin{align}\label{Fukaya-composition}
m_2\colon CF^*(\tilde{L}',\tilde{L}'')\otimes CF^*(\tilde{Q},\tilde{L}')\to CF^*(\tilde{Q}, \tilde{L}'').
\end{align}
It is defined by the modulo $2$ count of the number of rigid $J$-holomorphic curves $v\colon D_3 \to T^*Q$ such that $v(\partial_0 D_3)\subset Q$, $v(\partial_1 D_3)\subset L''$ and $v(\partial_2 D_3)\subset L'$.
Let $\mathcal{M}_{Q,L',L''}(x,x_1,x_2)$ denote the moduli space of such $J$-holomorphic curves with $v(e^{2\pi\sqrt{-1}/3})=x$, $v(e^{-2\pi\sqrt{-1}/3})=x_1$ and $v(1)=x_2$.
Then, $m_2$ induces a map on the cohomology
\begin{align*}
(m_2)^*\colon HF^*(\tilde{L_1},\tilde{L_1}) \otimes HF^*(\tilde{Q} , \tilde{L_1})\to HF^*(\tilde{Q}, \tilde{L_1}).
\end{align*}

For any Hamiltonian $H$ on $T^*Q$ with a compact support, 
there are canonical isomorphisms
\[ \begin{array}{cc} HF^*(\tilde{L_1},\tilde{L_1}) \to HF^*(\varphi_H(\tilde{L_1}),  \varphi_H(\tilde{L_1})), & HF^*(\tilde{Q}, \tilde{L_1}) \to HF^*(\tilde{Q}, \varphi_H(\tilde{L_1})) \end{array}\]
such that the following diagram commutes:
\begin{align}\label{diagram-m2}
\begin{split}
\xymatrix{
HF^*(\tilde{L_1},\tilde{L_1}) \otimes HF^*(\tilde{Q}, \tilde{L_1}) \ar[r]^-{(m_2)^*} \ar[d] &  HF^*(\tilde{Q}, \tilde{L_1}) \ar[d] \\
HF^*(\varphi_H(\tilde{L_1}), \varphi_H(\tilde{L_1})) \otimes HF^*(\tilde{Q}, \varphi_H(\tilde{L_1})) \ar[r]^-{(m_2)^*} &  HF^*(\tilde{Q}, \varphi_H(\tilde{L_1})) .
}
\end{split}
\end{align}
See \cite[Section (12f)]{Seidel} for the isotopy invariance of $(m_2)^*$.


\subsubsection{Floer cohomology of clean Lagrangian intersections and the triangle product}

The next proposition shows that 
if the clean intersection $K$ is connected, then $(m_2)^*$ is determined by the cup product on $K$.
Let $i_K\colon K\to L_1$ denote the inclusion map.
\begin{prop}\label{prop-Floer}
Suppose that $K$ is connected. Then, there exist $\sigma\in \Z$ and isomorphisms of graded $\Z/2$-vector spaces
\begin{align*}
\begin{array}{cc}
I_0\colon HF^*(\tilde{L_1},\tilde{L_1}) \to H^*(L_1;\Z/2),&  I_1, I_2 \colon HF^*(\tilde{Q},\tilde{L_1}) \to H^{*+\sigma}(K;\Z/2)
\end{array}
\end{align*}
such that $I_2\circ (m_2)^*\circ ( I_0\otimes I_1)^{-1} $ coincides with the map
\begin{align*}
H^*(L_1;\Z/2)\otimes H^{*+\sigma}(K;\Z/2) \to H^{*+\sigma}(K;\Z/2)\colon \xi \otimes \eta \mapsto (i_K)^*(\xi) \cup \eta.
\end{align*}
\end{prop}
\begin{proof}
Outside $D_r^*Q$,
we use the following identifications through the diffeomorphisms (\ref{diffeo-symplectization})
\[
\begin{array}{cc} \R_{\geq a_+} \times U^*Q \cong T^*Q\setminus D_r^*Q ,&  \R_{\geq a_+}\times \Lambda \cong L_1\setminus D_r^*Q , \end{array}\]
for $ e^{a_+}=r$.
By Lemma \ref{lem-neighborhood}, we may assume that $L_1 \cap D_{\epsilon}^*Q = L_K \cap D_{\epsilon}^*Q$ for $\epsilon\in (0, r]$. 

The existence of the isomorphisms is proved similarly to the case of closed symplectic manifolds \cite{F, P}.
We choose two Morse functions $f_i$ ($i=1,2$) on $K$ such that
\[2  \sup |f_2| + \sup |f_1| < \min \{r, \epsilon^2/4 \} ,\]
and $2f_2-f_1$ is also a Morse function on $K$.
For $i=1,2$, let $h_i \colon L_1\to \R$ be a Morse function on $L_1$ such that
\begin{align}\label{ineq-Morse}
\begin{split}
&h_i(q,p) = \begin{cases} |p| & \text{ if }|p|\geq r , \\ f_i(q) +|p|^2/2 & \text{ if } 0\leq |p|<\epsilon, \end{cases}\\
&(2h_2-h_1)(q,p)\geq 2 \sup |f_2| +\sup |f_1|\  \text{ if }|p|\geq \epsilon .
\end{split}
\end{align}
We also choose a generic Riemannian metric $g_K$ on $K$ such that the gradient vector fields for $f_1,f_2$ and $2f_2-f_1$ satisfy the Morse-Smale condition. Let $g$ be a Riemannian metric on $L_1$ such that $\rest{g}{K}=g_K$, every fiber of $L_K$ is orthogonal to $K$, and $K$ is a totally geodesic submanifold of $L_1$ (see \cite[Subsection 3.4.2]{P}).
We also require that $g= da \otimes da + g_{\Lambda}$ on $\R_{\geq a_+}\times \Lambda$ for some Riemannian metric $g_{\Lambda}$ on $\Lambda$ and that the gradient vector field for $2h_2-h_1$ satisfies the Morse-Smale condition.

We take a $1$-jet neighborhood of $\Lambda$ in $U^*Q$ via a map $\Psi_{\Lambda}$ from a neighborhood of the zero section in $J^1\Lambda= T^*\Lambda\times\R$ to $U^*Q$.
We then take a Weinstein neighborhood of $L_1$ in $T^*Q$ via a map $\Psi \colon D^*_{\epsilon'}L_1\to T^*Q$ for some $\epsilon'>0$ so that outside $D^*_rQ$,
\[\Psi ((a,z), (x,\xi))= (a, \Psi_{\Lambda}((x,\xi), z)) \in \R_{\geq a_+}\times U^*Q\]
for every $((a,z),(x,\xi)) \in T^*\R_{\geq a_+} \times T^*\Lambda$.
By Lemma \ref{lem-neighborhood}, we may also assume that $\Psi^{-1}(Q)$ coincides with the conormal bundle of $K$ in $D^*_{\epsilon'} L_1$.

Let $\pi_{L_1}\colon D^*_{\epsilon'}L_1\to L_1$ be the bundle projection.
For $i=1,2$, we take a time-independent Hamiltonian $H_i\colon T^*Q\to \R$ so that $H_i(q,p)=|p|$ outside $D_r^*Q$ and $H_i\circ \Psi= h_i \circ \pi_{L_1}$.
We take an almost complex structure $J$ on $T^*Q$ to define Floer cohomology. We additionally require that on $\Psi^{-1}(D^*_rQ)$, $\Psi^*J$ is equal to an almost complex structure on $T^*L_1$ specified from the Riemannian metric $g$ as described in \cite[Subsection 3.4.2]{P}.

Let $C_i$ ($i=1,2$) and $C^l_0$ be the set of critical points of $f_i$ and $2 f_2-f_1$ respectively. 
In addition, let $C^h_0$ be the set of critical points of $2h_2-h_1$ outside $D^*_{\epsilon}Q\cap L_K$.
If $w$ is sufficiently small,
$Q\cap \varphi_{wH_1}(L_1) = C_1$, $Q\cap \varphi_{2wH_2}(L_1) = C_2$ and $L_1\cap \varphi_{w (2H_2-  H_1)}(L_1)= C^l_0 \sqcup C^h_0$.
Therefore,
\begin{align*}
&CF^*(\tilde{Q},\varphi_{w H_1} (\tilde{L_1})) = \bigoplus_{x\in C_1} (\Z/2)x, \ CF^*(\tilde{Q},\varphi_{2w H_2} (\tilde{L_1})) = \bigoplus_{x\in C_2} (\Z/2)x, \\
&CF^*(\varphi_{w H_1}(\tilde{L_1}), \varphi_{2w H_2}(\tilde{L_1})) \cong CF^*( \tilde{L_1},\varphi_{w (2H_2-H_1)}(\tilde{L_1}))=  \bigoplus_{x\in C^l_0\sqcup C^h_0} (\Z/2)x. 
\end{align*}
From the assumption that $K$ is connected, \cite[Theorem 3.4.11]{P} shows the following: If $w>0$ is sufficiently small, then $(wH_1,J)$ and $(2wH_2,J)$ are regular and their Floer cochain complexes coincide with the Morse cochain complexes for $(f_1,g_K)$ and $(f_2,g_K)$ on $K$ respectively up to shift of degrees. 

Likewise,  if $w>0$ is sufficiently small, $(w(2H_2-H_1),J)$ is regular and its Floer cochain complex coincides with the Morse cochain complex for $(2h_2-h_1,g)$ on $L_1$ up to shift of degrees. 

%
Since these Morse cohomology groups are isomorphic to the singular cohomology of $L_1$ and $K$, we obtain isomorphisms of graded $\Z/2$-vector spaces
\begin{align*}
I_i\colon HF^*(\tilde{Q},\tilde{L_1}) \to H^{*+\sigma_i}(K;\Z/2), \ 
I_0\colon HF^*(\tilde{L_1},\tilde{L_1}) \to H^{*+\sigma_0}(L_1;\Z/2).
\end{align*}
from $(wH_i,J)$ ($i=1,2$) and $(w(2H_2-H_1),J)$. Here, $\sigma_0,\sigma_1,\sigma_2\in \Z$ are the differences of degrees.
Since $I_1\circ I_2^{-1} \colon H^{*+\sigma_1}(K) \to H^{*+\sigma_2}(K)$ is an isomorphism between graded $\Z/2$-vector spaces and $K$ is a closed manifold, we have $\sigma_1=\sigma_2$.
We set $\sigma\coloneqq \sigma_1=\sigma_2$.
In addition, $\sigma_0=0$ since $I_0$ agrees with the \textit{PSS isomorphism} over $\Z/2$ in \cite[Section (12e)]{Seidel}.

We also note that the projection
\[ \operatorname{pr} \colon CF^*(\varphi_{wH_1}(\tilde{L_1}),\varphi_{2wH_2} (\tilde{L_1})) \to \bigoplus_{x\in C^l_0} (\Z/2)x \]
is a chain map to the Morse cochain complex for $(2f_2-f_1,g_K)$. This is because the inequality
\[(2h_2-h_1)(x)\geq (2h_2-h_1)(x')= (2f_2-f_1)(x')\]
holds by (\ref{ineq-Morse}) for every $x\in C^h_0$ and $x'\in C^l_0$. Moreover, it is a standard fact that $\operatorname{pr}$ induces $(i_K)^*\colon H^*(L_1;\Z/2)\to H^*(K;\Z/2)$.

Next, we compute the triangle product $(m_2)^*$. By its definition, it comes from (\ref{Fukaya-composition})
for $L'\coloneqq \varphi_{wH_1}(L_1)$ and $L''\coloneqq \varphi_{2wH_2}(L_2)$.
 %
We choose arbitrarily $x\in L'\cap L''$, $x_1\in Q\cap L'=C_1$ and $x_2\in Q\cap L''=C_2$. We denote $x_0 \coloneqq \varphi_{wH_1}^{-1}(x) \in C^l_0\cup C^h_0$. For every $u\in \mathcal{M}_{(Q,L',L'')} (x,x_1,x_2)$, 
\begin{align*}
0\leq \int_{D_3}u^*(d\lambda_Q) 
= \int_{\partial D_3} (\Psi^{-1}\circ \rest{u}{\partial D_3})^*\lambda_{L_1} 
= (wh_1(x_0)- wh_1(x_1)) + ( 2wh_2(x_2)- 2wh_2(x_0)).
\end{align*}
If $x_0\in C^h_0$, the conditions (\ref{ineq-Morse}) on $f_i$ and $h_i$ for $i=1,2$ imply that
\[(h_1(x_0)- h_1(x_1)) + ( 2h_2(x_2)- 2h_2(x_0)) = (2f_2(x_2)-f_1(x_1)) - (2h_2-h_1)(x_0) <0, \]
hence $\mathcal{M}_{Q,L',L''}(x,x_1,x_2)$ is the empty set.

Therefore, we only need to observe the case $x_0\in C^l_0$. Let us set $N\coloneqq L_1\cap D^*_{\epsilon}Q$.
As in \cite[Lemma 3.4.12]{P}, we can prove the following: There exists $w_0>0$ such that
if $0<w<w_0$, then for every $u\in \mathcal{M}_{Q,L',L''}(x,x_1,x_2)$, $\Im u$ is contained in $\Psi( D^*_{\epsilon'} N )$. 
It should be remarked that \cite[Lemma 3.4.12]{P} is a result about pseudo-holomorphic curves on $D_2$, but we can easily modify its proof to those on $D_3$.

The Riemannian metrics $g_K$ and $g$ give isomorphisms $T^*K\cong TK$ and $T^*N \cong TN$, and thus
$T^*K$ is symplectically embedded in $T^*N$.
We shall prove the following claim: For every $u\in \mathcal{M}_{Q,L',L''}(x,x_1,x_2)$, $\Psi^{-1}(\Im u)$ is contained in $T^*K$.
To prove this, we may assume that the conormal bundle $L_K$ is trivial. If not, we choose another vector bundle $V$ over $K$ so that $L_K\oplus V\cong K\times \R^d$, and consider $T^*N$ as a symplectic submanifold of $T^*K\times \C^d$. Then $\Psi^{-1}\circ u$ can be seen as a pseudo-holomorphic curve in $T^*K\times \C^d$ such that
\begin{align*}
&u(\partial_0 D_3)\subset K\times \sqrt{-1}\R^d,\\
&u(\partial_1 D_3)\subset \Graph d (-2wf_2)\times (1+2w\sqrt{-1})\R^d, \\
&u(\partial_2 D_3)\subset   \Graph d (-wf_1)\times (1+w\sqrt{-1})\R^d.
\end{align*}
Here, we let $\Graph \eta$ denote the graph in $T^*K$ for any $1$-form $\eta$ on $K$. 
The $\C^d$-component of $\Psi^{-1}\circ u$ is a holomorphic curve in $\C^d$ such that $\partial_i D_3$ for $i=0$, $1$ and $2$ are mapped to $\sqrt{-1}\R^d$, $(1+w\sqrt{-1})\R^d$ and $ (1+2w\sqrt{-1})\R^d$ respectively. Clearly, this is a constant map to $0\in \C^d$.
Therefore, $\Psi^{-1}\circ u$ lies in $T^*K\times \{0\}$.

Let us abbreviate $\Graph d(-wf_1)$ by $K'$ and $\Graph d(-2wf_2)$ by $K''$.
We have $\tilde{K}=(K,\gr_K)$ from Example \ref{ex-gr}. Fix the gradings of $K'$ and $K''$ by $\tilde{K}' \coloneqq \varphi_{wf_1\circ \pi_K}(\tilde{K})$ and $\tilde{K}'' \coloneqq \varphi_{2w f_2\circ \pi_K} (\tilde{K})$ respectively, where $\pi_K\colon T^*K\to K$ is the bundle projection.
For an almost complex structure on $T^*K$ specified from $g_K$ as in  \cite[Subsection 3.2.4]{P},
we define Floer cochain complexes $CF^*(\tilde{K},\tilde{K}')$, $CF^*(\tilde{K},\tilde{K}'')$ and $CF^*(\tilde{K}',\tilde{K}'')$. They coincide with Morse cochain complexes for $(f_1,g_K)$, $(f_2,g_K)$ and $(2f_2-f_1,g_K)$ respectively. Similar to $\mathcal{M}_{Q,L',L''}(x,x_1,x_2)$, we define $\mathcal{M}_{(K,K',K'')}(x,x_1,x_2)$ to be the moduli space of pseudo-holomorphic curves on $D_3$ in $T^*K$ with the boundary conditions given by $K$, $K'$ and $K''$.
By the claim we have shown above, if $w$ is sufficiently small, then
\[\mathcal{M}_{K,K',K''}(x,x_1,x_2)\to \mathcal{M}_{Q,L',L''}(x,x_1,x_2) \colon v'\mapsto \Psi\circ v'\]
is a bijection.

We now conclude that the following diagram commutes:
\[\xymatrix{
CF^*(\tilde{L}', \tilde{L}'')  \otimes  CF^*(\tilde{Q}, \tilde{L}')  \ar[d]_-{\operatorname{pr} \otimes \id }   \ar[r]^-{m_2} & CF^*(\tilde{Q},\tilde{L}'') \ar@{=}[d] \\
CF^*(\tilde{K}',\tilde{K}'')  \otimes  CF^{*+\sigma}(\tilde{K}, \tilde{K}') \ar[r]^-{m'_2} & CF^{*+\sigma}(\tilde{K},\tilde{K}'') .
}\]
The lower map $m'_2$ is the chain map such that
the coefficient of $x_2$ in $m'_2(x\otimes x_1)$ is the number of points of $\mathcal{M}_{K,K',K''}(x,x_1,x_2)$ modulo $2$. It is well-known that $m'_2$ induces the cup product on $HF^*(\tilde{K},\tilde{K})\cong H^*(K;\Z/2)$. For the proof, see \cite[Part I]{FO}. Since $\operatorname{pr}$ induces $(i_K)^*$ on the cohomology, the proposition is proved.
\end{proof}

\begin{rem}\label{rem-shift-degree}
In the case where $L_1\cap D^*_{\epsilon}Q=L_K$ and $L_1$ is connected, we have a unique grading $\gr$ of $L_1$ such that $\gr (x) = \gr_{L_K}(x)$ for every $x\in K=Q \cap L_K$.
For the grading $\gr_{L_K}$, see Example \ref{ex-gr}. 
For $\tilde{L_1}=(L_1,\gr)$ with this grading, let $\sigma_K\in \Z$ be the shift of degree such that the isomorphisms
\[I_1 , I_2 \colon HF^*(\tilde{Q}, \tilde{L_1})\to H^{*+\sigma_K}(K)\]
preserve the $\Z$-gradings. From the proof of Proposition \ref{prop-Floer}, $\sigma_K$ is determined by local data around $K$. Namely, if $N$ is a tubular neighborhood of $K$ in $Q$, then $\sigma_K$ is computed by replacing the triple $(T^*Q,\tilde{Q},\tilde{L_1})$ by $(T^*N,\tilde{N},\tilde{L}_K)$. Though, in fact,  we can compute that $\sigma_K=0$, this property of $\sigma_K$ is sufficient for our purpose.
\end{rem}

\begin{rem}\label{rem-Z-coefficient}
It is standard to prove that $I_1=I_2$ by taking a homotopy between $f_1$ and $f_2$. In addition,
the author expects that the above results can be extended to $\Z$-coefficients in the following case:  $Q$ and $L_1$ admit brane structures (see \cite[Section (12a)]{Seidel}) and the normal bundle of $K$ in $Q$ is trivial. 
However, the Proposition \ref{prop-Floer} for $\Z/2$-coefficient cohomology is sufficient to prove the main results, so we do not pursue such extensions in this paper.
\end{rem}

\section{SFT for Lagrangian cobordisms}\label{sec-DGA}

In this section, assuming several conditions, we introduce the Chekanov-Eliashberg DGA of a Legendrian submanifold $\Lambda$ in a contact manifold with coefficients in the group ring $\Z[H_1(\Lambda)]$. We refer to \cite{DR, EES-R, EES-ori, EES}.
We then define a DGA map which is associated to an exact Lagrangian cobordism between two Legendrain submanifolds.
Some technical details, especially on the orientations of moduli spaces, are left to Appendix \ref{sec-proof}.

\subsection{Setup}\label{subsec-setup}

Let $(P,\lambda)$ be a $2n$-dimensional Liouville manifold with a Liouville form $\lambda$.
We equip $P\times \R$ with a contact structure by a contact form  $dz - \lambda$, where $z$ is the coordinate on $\R$.
To simplify the discussion on the grading of the DGA, we assume that $2c_1(TP)=0$ and $H_1(P)=0$. These conditions are satisfied in the case of $P=T^*S^2$ in Section \ref{sec-knot}.
Let $\pi_P\colon P\times \R\to P$ denote the projection.
%
Suppose that $\Lambda$ is a compact Legendrian submanifold of $P\times \R$ satisfying the following topological conditions:
\begin{itemize}
\item $\Lambda$ is a connected and oriented manifold. In addition, it admits a spin structure.
\item  
The Maslov class of the immersed Lagrangian submanifold $\pi_P(\Lambda)$ of $P$ vanishes. More concretely, it is described as follows:
For every loop $\gamma \colon S^1 \to \Lambda$,
we choose a map $\tilde{\gamma}\colon \Sigma \to P$ on an oriented surface $\Sigma$ whose restriction on the boundary gives $\pi_P\circ \gamma$, and take a symplectic trivialization $\Psi \colon \tilde{\gamma}^* TP \to \C^n$. Then, the Maslov index of the loop $(\Psi ( (\pi_P)_*(T_{\gamma(t)} \Lambda)))_{t\in S^1}$ of Lagrangian subspaces in $\C^n$ is $0$.
\end{itemize}
Here, the Maslov index of the loop does not depend on the choice of $\tilde{\gamma}$ and $\Psi$ since $2c_1(TP)=0$. See \cite[Subsection 2.2]{EES}.
We also assume that
$\rest{\pi_P}{\Lambda} \colon \Lambda \to P$ is an immersion such that the self intersection of $\pi_P(\Lambda)$ consists only of transverse double points. This condition is satisfied for generic $\Lambda$.

Let us give several definitions related to symplectic field theory in $\R\times (P\times \R)$. We refer to \cite{DR}, but see Remark \ref{rem-difference-of-convention} below about some differences of conventions.
We denote the set of Reeb chords of $\Lambda$ by
\[\mathcal{R}(\Lambda)\coloneqq \{a \colon [0,T] \to P\times \R\mid T>0,\ \textstyle{\frac{d a}{dt}} (t)= (\partial_z)_{a(t)} \text{ for every }t\in [0,T]\text{ and }a(0),a(T)\in \Lambda\}.\]
For every $(a\colon [0,T]\to P\times \R) \in \mathcal{R}(\Lambda)$, we fix a path $\gamma_a \colon [0,1] \to \Lambda$ from $a(0)$ to $a(T)$, which we call a \textit{capping path} of $a$.
Then, we can define for every $a\in \mathcal{R}(\Lambda)$ a loop of Lagrangian subspaces in $\C^{n}$ as follows: 
We choose a symplectic trivialization of $TP$ along the loop $\pi_P\circ \gamma$ which extends to an oriented surface bounding $\pi_P\circ \gamma$. Let $\Psi\colon (\pi_P\circ \gamma)^* TP \to \C^n$ denote this trivialization. Then, we have a path of Lagrangian subspaces $(\Psi((\pi_P)_*(T_{\gamma(t)}\Lambda)))_{t\in [0,1]}$ in $\C^{n}$. We close it up by the positive rotation as described in \cite[Subsection 2.2]{EES-R}. 
When the Maslov index of this loop is $\nu(a)\in \Z$, we define $|a|\coloneqq \nu(a) -1$.
The independence of $|a|$ of the choice of $\Psi$ and $\gamma_a$ follows from the assumption that $2c_1(TP)=0$ and that the Maslov class of $\pi_P(\Lambda)$ vanishes.

Next,
we choose an almost complex structure $J$ on $P$ which is compatible with $d\lambda$. Namely, $d\lambda(\cdot , J\cdot)$ is a Riemannian metric on $P$. Moreover, in the sense of the definitions in \cite[Subsection 2.3]{EES}, we assume that $J$ is \textit{adapted} to $\Lambda$ and $\Lambda$ is \textit{admissible}.
We associate an almost complex structure $\tilde{J}$ on the symplectization $\R\times (P\times \R)$ as follows:
On the contact distribution $\xi \coloneqq \ker (dz- \lambda)$, $\rest{d\pi}{\xi} \colon \xi \to TP$ is a bundle isomorphism, so we define $\tilde{J}$ on $T_{(a,(p,z))}(\R\times (P\times \R))= \R\partial_a \oplus \xi_{(p,z)} \oplus \R\partial_z$ by
\begin{align}\label{alm-cpx-str}
\tilde{J} = \begin{pmatrix} 0 & 0 & 1 \\ 0 & (\rest{d\pi}{\xi})^*J & 0 \\ -1 & 0 & 0 \end{pmatrix}.
\end{align}
Then, $\tilde{J}$ is compatible with $d(e^a(-dz+\lambda))$ and maps $\partial_z$ to $\partial_a$. Moreover,
it is invariant by the following translation for every $s\in \R$:
\[\tau_s\colon \R\times (P\times \R)\to  \R\times (P\times \R)\colon (a,(p,z)) \mapsto (a+s, (p,z)).\]

\begin{rem}\label{rem-difference-of-convention}
In \cite{DR}, $P\times \R$ is equipped with the contact form $dz+\lambda$, so there are differences in some definitions from the present paper. More concretely, the capping path in \cite[Subsection 4.1]{DR} for each Reeb chord $a\colon [0,T] \to P\times \R$ of $\Lambda$ is a path in $\Lambda$ from $a(T)$ to $a(0)$ and the almost complex structure $\tilde{J}$ on $\R\times (P\times \R)$ in \cite[Subsection 3.4]{DR} maps $\partial_a $ to $\partial_z$.
\end{rem}

For the unit disk $D$ and $m\in \Z_{\geq 0}$, fix $p_0,p_1,\dots ,p_m\in \partial D$ in counterclockwise order and let $D_{m+1}\coloneqq D\setminus \{p_0,p_1,\dots ,p_m\}$.
Let $\mathcal{C}_{m+1}$ be the space of conformal structures on $D_{m+1}$ which can be smoothly extended to the disk $D$.
Given $\kappa\in \mathcal{C}_{m+1}$, we define $j_{\kappa}$ to be the associated complex structure on $D_{m+1}$. There are holomorphic strip coordinates
\[\psi_0\colon [0,\infty) \times [0,1]\to D_{m+1},\ \psi_k\colon (-\infty,0]\times [0,1]\to D_{m+1} \]
near $p_0$ and $p_k$ for $k=1,\dots ,m$ respectively. Here, the domains of $\psi_0$ and $\psi_k$ are identified with subspaces of $\C$ by a map $\R\times [0,1] \to \C \colon (s,t)\mapsto s+ \sqrt{-1} t$.
For $m=0,1$, $\mathcal{C}_{m+1}$ consists of one point, and we define $\mathcal{A}_{m+1}$ to be the group of automorphisms on $D_{m+1}$.
We also define $\overline{\partial D_{m+1}}$ to be a compactification of $\partial D_{m+1}$ obtained by adding $\psi_0(\{+\infty\}\times\{0,1\})$ and $\psi_k(\{-\infty\}\times\{0,1\})$ for $k=1,\dots ,m$. The orientation of $\overline{\partial D_{m+1}}$ is induced by the orientation of $\partial D$ as the boundary of $D$.

\subsection{Pseudo-holomorphic curves with boundary in an exact Lagrangian cobordism}\label{subsec-J-hol}

We start with the definition of an exact Lagrangian cobordism in the symplectization of $P\times \R$.
\begin{defi}\label{def-cobordism}
Let $L$ be an exact Lagrangian submanifold of $(\R\times (P\times \R) , e^a (-dz+\lambda))$.
For two compact Legendrian submanifolds $\Lambda_+$ and $\Lambda_-$ of $P\times \R$, $L$ is called an \textit{exact Lagrangian cobordism} from $\Lambda_-$ to $\Lambda_+$, if there exists $a_-,a_+\in \R$ with $a_-<a_+$ such that the following hold:
\begin{enumerate}
\item $L\cap ( [a_-,a_+]\times (P\times \R))$ is compact.
\item $L\cap (\R_{\geq a_+}\times (P\times \R)) =\R_{\geq a_+}\times \Lambda_+ $ and $L\cap (\R_{\leq a_-} \times (P\times \R)) = \R_{\leq a_-} \times \Lambda_- $.
\item There exists a $C^{\infty}$ function $f\colon L\to \R$ such that $\rest{e^a(-dz+\lambda)}{L}=df $ and $f$ is constant on  $\R_{\leq a_-} \times \Lambda_-$.
\end{enumerate}

\end{defi}


Let us introduce some notations.
For a compact subset $\bar{L}\coloneqq L \cap ([a_-,a_+]\times (P\times \R))$ of $L$, we define a strong deformation retract  $r_{\bar{L}}\colon L\to \bar{L}$ so that $r_{\bar{L}}(a,x)=(a_+, x)$ if $a\geq a_+$ and $r_{\bar{L}}(a,x)=(a_-, x)$ if $a\leq a_-$.
We also define embeddings
\begin{align*}
i_{+}\colon \Lambda_{+}\to L \colon x\to (a_{+},x), \ 
i_{-}\colon  \Lambda_{-}\to L \colon x\to (a_{-},x) .
\end{align*}

With these preparations,
we define a moduli space of $\tilde{J}$-holomorphic curves in $\R\times (P\times \R)$
for every $B\in H_1(L)$ and for Reeb chords $(a\colon [0,T]\to P\times \R)\in \mathcal{R}(\Lambda_+)$ and $(b_k\colon [0,T_k]\to P\times \R) \in \mathcal{R}(\Lambda_-)$ for $k=1,\dots ,m$.
Referring to  \cite[Section 4.2.2]{DR}, we first define a space
\begin{align}\label{moduli-hat}
\hat{\mathcal{M}}_{L,B}(a;b_1,\dots ,b_m)
\end{align}
which consists of pairs $(u,\kappa)$ of $\kappa\in \mathcal{C}_{m+1}$ and a pseudo-holomorphic curve
\[u\colon D_{m+1} \to \R\times (P\times \R) \]
with respect to $\tilde{J}$ and $j_{\kappa}$
satisfying the following conditions:
\begin{itemize}
\item $u(\partial D_{m+1} ) \subset L$.
\item There exist $s_0, s_1,\dots ,s_m\in \R$ such that
$u\circ \psi_0(s,t)$ is asymptotic to $(Ts +s_0 , a (T(1- t)))$ if $s\to \infty$ and
$u\circ \psi_k(s,t)$ is asymptotic to $(T_ks +s_k , b_k(T_k(1- t)))$ for $k=1,\dots ,m$  if $s\to -\infty$. 
\item We extend 
$r_{\bar{L}}\circ \rest{u}{\partial D_{m+1}}$ continuously to $\overline{\partial D_{m+1}}$. Let $\partial u$ denote this $1$-chain in $L$. Then,
\[ \left[ \partial u +(i_-\circ  \gamma_{b_1} + \dots +i_-\circ \gamma_{b_m}) - i_+\circ \gamma_{a} ) \right] = B \in H_1(L).\]
\end{itemize}
When $m=0,1$, the group $\mathcal{A}_{m+1}$ acts on this space by $(u,\kappa)\cdot \varphi\coloneqq (u\circ \varphi, \kappa)$ for all $\varphi\in \mathcal{A}_{m+1}$. The moduli space of $\tilde{J}$-holomorphic curves up to conformal equivalence is defined by
\begin{align}\label{moduli-cobordism}
\mathcal{M}_{L,B}(a;b_1,\dots ,b_m) \coloneqq \begin{cases}
 \hat{\mathcal{M}}_{L,B}(a;b_1,\dots ,b_m) & \text{ if }m\geq 2, \\
 \hat{\mathcal{M}}_{L,B}(a;b_1,\dots ,b_m)/\mathcal{A}_{m+1} & \text{ if }m=0,1.
\end{cases}
\end{align}

\subsection{Chekanov-Eliashberg DGA over $\Z[H_1(\Lambda)]$}\label{subsec-CE-DGA}

We consider the moduli space (\ref{moduli-cobordism}) when $\Lambda_+=\Lambda_-=\Lambda$ and $L=\R\times \Lambda$. For each $a\in \mathcal{R}(\Lambda)$, we take the same capping path $\gamma_a$ in $\Lambda_+$ and $\Lambda_-$.
Since $\tau_s^*\tilde{J}=\tilde{J}$ and $\tau_s(\R\times \Lambda)=\R\times \Lambda$ for all $s\in \R$, there is an action of $\R$ on the moduli space defined by $s\cdot (u,\kappa)\coloneqq (\tau_s\circ u,\kappa)$.
For every $A\in H_1(\Lambda)$ and $a,b_1,\dots ,b_m\in \mathcal{R}(\Lambda)$, we define the quotient space
\begin{align}\label{moduli-Legendrian}
\bar{\mathcal{M}}_{\Lambda,A} (a; b_1,\dots ,b_m) \coloneqq \R\backslash \mathcal{M}_{\R\times \Lambda, (i_+)_*A}(a;b_1,\dots ,b_m). 
\end{align}
This moduli space satisfies the following:
$\bar{\mathcal{M}}_{\Lambda,A} (a; b_1,\dots ,b_m) $ is transversely cut out for generic $J$, and its dimension as a smooth manifold is
\[|a|-(|b_1|+\dots +|b_m|)-1.\]
By the Gromov compactness, the union  $\coprod_{A\in H_1(\Lambda)} \bar{\mathcal{M}}_{\Lambda, A}(a;b_1,\dots ,b_m)$ is a compact $0$-dimensional manifold if $|a|-1=|b_1|+\dots +|b_m|$.
For more details, see Appendix \ref{subsec-moduli-P}.

Moreover, if we choose a spin structure $\mathfrak{s}$ on $\Lambda$
and an orientation of the determinant line of the capping operator for each Reeb chord (see Appendix \ref{subsec-cap}), we can specify an orientation of $\coprod_{A\in H_1(\Lambda)} \bar{\mathcal{M}}_{\Lambda, A}(a; b_1,\dots ,b_m)$ for every $a,b_1,\dots ,b_m\in \mathcal{R}(\Lambda)$.
See Definition \ref{def-ori-barM}.

Now, 
we define $\mathcal{A}_*(\Lambda)$ as the $\Z$-graded free associative non-commutative algebra over $\Z[H_1(\Lambda)]$ generated by $\mathcal{R}(\Lambda)$ so that the degree of every $a\in \mathcal{R}(\Lambda)$ is $|a|$.
In addition, we define the graded derivation
$\partial_{\mathfrak{s}} \colon \mathcal{A}_*(\Lambda) \mapsto \mathcal{A}_{*-1}(\Lambda)$
over $\Z[H_1(\Lambda)]$
so that for every $a\in \mathcal{R}(\Lambda)$
\begin{align}\label{diff-DGA}
 \partial_{\mathfrak{s}} (a) \coloneqq \sum_{|a|-1=|b_1|+\dots +|b_m|} \sum_{A\in H_1(\Lambda)} (-1)^{(n-1)(|a|+1)}\left( \#_{\sign}  \bar{\mathcal{M}}_{\Lambda,A}(a;b_1,\dots ,b_m) \right) \cdot e^A \cdot b_1\cdots b_m, 
\end{align}
and extend it by the Leibniz rule. Here, $\#_{\sign}$ is the number of the points counted with the signs induced by the orientation, and $\{e^A\mid A\in \Z[H_1(\Lambda)]\}$ generates the group ring.
From \cite[Theorem 4.1]{EES-ori}
and our choice of orientations of the moduli spaces in Definition \ref{def-ori-barM}, we get the following result. 
\begin{prop}\label{prop-DGA}
The pair $(\mathcal{A}_*(\Lambda),\partial_{\mathfrak{s}})$ is a differential graded algebra over $\Z[H_1(\Lambda)]$, that is, $\partial_{\mathfrak{s}}\circ \partial_{\mathfrak{s}}=0$.
\end{prop}

$(\mathcal{A}_*(\Lambda), \partial_{\mathfrak{s}})$ is called the \textit{Chekanov-Eliashberg DGA} of $\Lambda$.
To construct this DGA, we have chosen auxiliary data: an almost complex structure $J$ on $P$, a spin structure $\mathfrak{s}$ on $\Lambda$ and for every $a\in \mathcal{R}(\Lambda)$, a capping path $\gamma_a$ and an orientation of the determinant line of the capping operator.
We consider the stable tame isomorphism class of this DGA. This notion is defined by an equivalence relation generated by stabilizations and tame automorphisms over $\Z[H_1(\Lambda)]$. See, for instance, \cite[Subsection 2.6]{EES-non}.
\begin{thm}[{\cite[Theorem 4.12]{EES-ori}}]
The stable tame isomorphism class of $(\mathcal{A}_*(\Lambda),\partial_{\mathfrak{s}})$ as a DGA over $\Z[H_1(\Lambda)]$ is independent of the choice of auxiliary data except the spin structures, and invariant by Legendrian isotopies of $\Lambda$ preserving the spin structures.
\end{thm}

It remains to discuss the change of spin structures. Let $\Spin (\Lambda)$ denote the set of spin structures on $\Lambda$ up to isomorphisms.
We fix $\mathfrak{s}\in \Spin (\Lambda)$ to define the DGA $(\mathcal{A}_*(\Lambda), \partial_{\mathfrak{s}})$.
It is well-known that there is a bijection
\[ \mathrm{Spin}(\Lambda) \to H^1(\Lambda;\Z/2)\colon \mathfrak{s}' \mapsto d(\mathfrak{s}, \mathfrak{s}').\]
See \cite[Subsection 4.4.1]{EES-ori}. We define a ring isomorphism
\begin{align}\label{spin-change}
\psi (\mathfrak{s},\mathfrak{s}') \colon \Z[H_1(\Lambda)] \to \Z[ H_1(\Lambda)] \colon e^A\mapsto (-1)^{\la d(\mathfrak{s},\mathfrak{s}'),A \ra} e^A.
\end{align}
\begin{prop}[{\cite[Theorem 4.29]{EES-ori}}]\label{prop-change-spin}
Let $\partial_{\mathfrak{s}'}$ be the differential on $\mathcal{A}_*(\Lambda)$ defined by choosing $\mathfrak{s}'\in \mathrm{Spin}(\Lambda)$. If we define an isomorphism of graded rings
\[\Psi \colon \mathcal{A}_*(\Lambda) \to \mathcal{A}_*(\Lambda) \colon \begin{cases} x \mapsto \psi(\mathfrak{s},\mathfrak{s}') (x) & \text{ if } x\in \Z[H_1(\Lambda)], \\
a \mapsto a & \text{ if }a\in \mathcal{R}(\Lambda), \end{cases}\]
then $\partial_{\mathfrak{s}'} \circ \Psi = \Psi \circ \partial_{\mathfrak{s}}$ holds.
\end{prop}


\subsection{DGA map for Lagrangian cobordism}\label{subsec-DGA-map}

We return to a general exact Lagrangian cobordism $L$ from $\Lambda_-$ to $\Lambda_+$.
Hereafter, we assume that $L$ is oriented and admits a spin structure and that its Maslov class vanishes. As in the former section, both $\Lambda_+$ and $\Lambda_-$ are assumed to be connected. Then, the third condition of the Definition \ref{def-cobordism} is satisfied automatically.

We fix a spin structure $\mathfrak{s}_0$ on $L$. Then, $\Lambda_+$ and $\Lambda_-$ are equipped with spin structures $\mathfrak{s}_{+}$ and $\mathfrak{s}_-$ respectively as the boundary of $\bar{L}$. See Remark \ref{rem-spin-boundary} for a precise description.
By fixing the capping path and the orientation of the determinant line of the capping operator for every Reeb chord, we specify the orientations of moduli spaces and define the Chekanov-Eliashberg DGAs $(\mathcal{A}_*(\Lambda_+),\partial_{\mathfrak{s}_{+}})$ and $(\mathcal{A}_*(\Lambda_-),\partial_{\mathfrak{s}_{-}})$.
\begin{rem}\label{rem-spin-boundary}
From the orientation on $L$,
$\Lambda_{\pm}$ is oriented so that $ T_{(a_{\pm},x)}L = T_{a_{\pm}}\R \oplus T_x\Lambda_{\pm} $ as oriented vector spaces. Here, a positive vector in $T_{a_{\pm}}\R$ is $\partial_a$.
Moreover, $\mathfrak{s}_0$ defines a spin structure on  $\rest{T L}{\{a_{\pm}\}\times \Lambda_{\pm}} \cong \underline{ \R} \oplus T\Lambda_{\pm}$ as a double cover of the oriented orthonormal frame bundle $\SO(\underline{ \R} \oplus T\Lambda_{\pm})$ of $\underline{ \R} \oplus T\Lambda_{\pm}$. Here, we fix some metric on $T\Lambda_{\pm}$.
The pullback of this covering space on the oriented orthonormal frame bundle $\SO(T\Lambda_{\pm})$ of $T\Lambda_{\pm}$, which can be seen as a subbundle of $\SO(\underline{ \R} \oplus T\Lambda_{\pm})$, 
defines a spin structure $\mathfrak{s}_{\pm}$ on $T\Lambda_{\pm}$. 
\end{rem}

Let us consider a base change of $\mathcal{A}_*(\Lambda_-)$ by the ring map $(i_-)_*\colon \Z[H_1(\Lambda_-)]\to \Z[H_1(L)]$. Then, we obtain a DGA over $\Z[H_1(L)]$
\begin{align}\label{base-change}
(\mathcal{A}_*(\Lambda_-)\otimes_{\Z[H_1(\Lambda_-)]} \Z[H_1(L)], \partial_{\mathfrak{s}_-}\otimes \id_{ \Z[H_1(L)]} ).
\end{align}
Furthermore, by the ring map $(i_+)_*\colon \Z[H_1(\Lambda_+)]\to \Z[H_1(L)]$, we can think of it as a DGA over $\Z[H_1(\Lambda_+)]$.

%

We consider the moduli space (\ref{moduli-cobordism}).
By \cite[Lemma 3.8]{EHK}, this moduli space 
is transversely cut out for generic $J$, and its dimension as a smooth manifold is 
\[|a|- (|b_1|+ \dots + |b_m|).\]
By the Gromov compactness, the union  $\coprod_{B\in H_1(L)} \mathcal{M}_{L, B}(a;b_1,\dots ,b_m)$ is a compact $0$-dimensional manifold if $|a|=|b_1|+\dots +|b_m|$.

The orientation of this moduli space is fixed as in Appendix \ref{subsec-ori-Lag}. 
Then, the following result holds. 
\begin{thm}\label{thm-DGA-map}
We define a homomorphism of graded $\Z[H_1(\Lambda_+)]$-algebras
\[\Phi_L \colon \mathcal{A}_*(\Lambda_+)  \to \mathcal{A}_*(\Lambda_-) \otimes_{\Z[H_1(\Lambda_-)]} \Z[H_1(L)]\]
so that for every $a\in \mathcal{R}(\Lambda_+)$,
\[\Phi_L (a) \coloneqq \sum_{|a|=|b_1|+\dots +|b_m|} \sum_{B\in H_1(L)} (-1)^{(n-1)(|a|+1) +m} \left( \#_{\sign} \mathcal{M}_{L,B}(a; b_1,\dots ,b_m) \right) \cdot (b_1\cdots b_m) \otimes e^B. \]
Then, $\Phi_L$ is a DGA map over $\Z[H_1(\Lambda_+)] $, that is, $ \Phi_L\circ \partial_{\mathfrak{s}_+} = (\partial_{\mathfrak{s}_-} \otimes \id_{\Z[H_1(L)]} )\circ \Phi_L $.
\end{thm}
An analogous result in $\Z$-coefficients is proved in \cite[Theorem 2.5]{K}.
We give a proof Theorem \ref{thm-DGA-map} in Appendix \ref{subsec-proof-DGA} by using the conventions of signs in \cite{EES-ori}, which we summarize in Appendix \ref{subsec-convention}. Note that these conventions are different from those of \cite{K}. We use the conventions of \cite{EES-ori} because we will later apply results about the DGAs in \cite{EENS}, which refers to \cite{EES-ori} to fix orientations of moduli spaces in \cite[Section 6]{EENS}.


\subsection{Application to clean Lagrangian intersection in $T^*\R^{n}$}\label{subsec-app}

We equip $\R^{n}$ with the standard metric $\la\cdot, \cdot \ra \coloneqq \sum_{i=1}^n dx_i\otimes dx_i$ and consider the contact manifold with a contact form $(U^*\R^{n},  \alpha_{\R^{n}})$.
Note that for the unit sphere $S^{n-1}$ in $\R^{n}$, there exists a diffeomorphism
\begin{align}\label{diffeo-identify}
 U^*\R^{n} \to T^*S^{n-1}\times \R \colon (q,p) \mapsto ( (p, q-\la q,p \ra p), \la q,p \ra)
\end{align}
such that the pullback of  $dz-\lambda_{S^{n-1}}$ coincides with $\alpha_{\R^{n}}$.
For $n\geq 3$, the Liouville manifold $(P,\lambda) = (T^*S^{n-1}, \lambda_{S^{n-1}})$ satisfies the condition that $2c_1(TP)=0$ and $H_1(P)=0$, which we have assumed in Subsection \ref{subsec-setup}. We define the Chekanov-Eliashberg DGAs of Legendrian submanifolds of $U^*\R^n$ by using the identification (\ref{diffeo-identify}).

We return to the situation of Subsection \ref{subsec-clean} for $Q=\R^{n}$.
Let $L_1$ be an exact Lagrangian submanifold of $T^*\R^{n}$ which bounds a compact connected Legendrian submanifold $\Lambda$ of $U^*\R^{n}$.
We assume that $L_1$ is spin and its Maslov class vanishes.
If the zero section $\R^n$ and $L_1$ have a clean intersection $K=\R^n \cap L_1$,
we define an exact Lagrangian submanifold $L$ in $(\R\times U^*\R^{n}, e^a\alpha_{\R^n})$ as (\ref{induced-cobordism}) which is diffeomorphic to $L_1\setminus K$ by $\psi^{-1} \circ F$. Here, $\Lambda$ and $\Lambda_K$ are endowed with spin structures as the boundaries of $L\cap ([a_-,a_+]\times U^*\R^n)$. We also define
\begin{align}\label{boundary-embedding}
\begin{split}
& i'_+\colon \Lambda \to L_1\setminus K\colon x\mapsto \psi^{-1}\circ F(a_+,x), \\
& i'_-\colon \Lambda_K \to L_1\setminus K \colon x \mapsto  \psi^{-1}\circ F(a_-,x).
\end{split}
\end{align}

If $K$ is connected, then $L$ satisfies the conditions in Definition \ref{def-cobordism}, especially the third one, to be an exact Lagrangian cobordism form $\Lambda_K$ to $\Lambda$.
Therefore, Theorem \ref{thm-DGA-map} is applied to show that a DGA map is induced by $L$.

\begin{thm}\label{thm-clean-intersection}
Suppose that $L_1$ satisfies the above conditions. If the clean intersection $K=\R^n \cap L_1$ is connected, then there exists a DGA map over $\Z[H_1(\Lambda)]$
\[ \Phi_{L_1} \colon \mathcal{A}_*(\Lambda) \to \mathcal{A}_*(\Lambda_{K}) \otimes_{\Z[H_1(\Lambda_K)]} \Z[H_1(L_1\setminus K)] .\]
Here, the target object is considered as a DGA over $\Z [H_1(\Lambda)]$ in the same way as (\ref{base-change}) by
\[ (i'_+)_*\colon \Z[H_1(\Lambda)] \to \Z[H_1(L_1\setminus K)],\ (i'_-)_*\colon \Z[H_1(\Lambda_K)] \to \Z[H_1(L_1\setminus K)]. \]
\end{thm}

\section{$1$-dimensional clean intersection in $T^*\R^3$}\label{sec-knot}

In this section, we deduce some topological constraints on the clean Lagrangian  intersections from Theorem \ref{thm-clean-intersection} when $n=3$ and $L_1=\varphi(L_{K_0})$ for a knot $K_0$ in $\R^3$ and $\varphi\in \Ham_c(T^*\R^3)$.
We focus on the knot invariants from \cite{N} and their properties.

\subsection{Smooth isotopy of the conormal bundle of a knot}\label{subsec-smooth}
To emphasize on the difference between Hamiltonian isotopies and $C^{\infty}$ isotopies, we first discuss $C^{\infty}$ isotopies of $L_{K_0}$.

\begin{prop}\label{prop-smooth-isotopy}
For an arbitrary pair of knots $(K_0, K)$ in $\R^3$,
there exists a compactly supported $C^{\infty}$ isotopy $(\varphi^t)_{t\in [0,1]}$ of $T^*\R^3$ such that $\R^3$ and $\varphi^1(L_{K_0})$ have a clean intersection along $K$.
\end{prop}
\begin{proof}
It suffices to show this claim for the following pair of knots $(K_0,K)$:
$K_0$ is an arbitrary knot given by a knot diagram. We choose one crossing in the diagram and change one over crossing to an under crossing or vice versa in the diagram of $K_0$. Let $K$ be the knot determined by the diagram after changing the crossing.

For such $K_0$ and $K$, we can find an isotopy of immersions $(\iota_t\colon S^1\looparrowright \R^3)_{t\in [0,1]}$
with $\iota_0(S^1)=K_0$ and $\iota_1(S^1)=K$ such that
$\iota_t$ is an embedding for every $t\neq 1/2$ and
$\iota_{1/2} \colon S^1\to \R^3$ is an immersion with a single double point.
Let $x\in \iota_{1/2}(S^1)$ be the double point with $\iota_{1/2}^{-1}(x)= \{\theta_1,\theta_2\}$.

We choose a smooth family of frames $(\{u_t,v_t\})_{t\in [0,1]}$ of the conormal bundle of $\iota_t$ and consider a map
\[ [0,1]\times ( S^1\times \R^2) \to T^*\R^3 \colon  (t,(s,(a,b)))\mapsto (\iota_t(s), au_t(s)+bv_t(s))\]
We smoothly deform this map near $\{t=1/2,\ s=\theta_1\}$, while not changing it near $\{t=0,1\} \cup \{s=\theta_2\}$, to a map
\[ [0,1]\times (S^1\times \R^2)\to  T^*\R^3 \colon (t, (s,w)) \mapsto I_t(s,w)\] 
so that for every $t\in [0,1]$ and $s\in S^1$, $I_t(\{s\}\times \R^2)$ is a hypersurface in $T^*_{\iota_t(s)}\R^3$
embedded as a closed subset, and $I_{1/2} (\{\theta_1\}\times \R^2) \subset T_{x}\R^3\setminus I_{1/2}(\{\theta_2\}\times \R^2)$.
Then, $(I_t)_{t\in [0,1]}$ is an isotopy of embeddings of $S^1\times \R^2$ into $T^*\R^3$ whose images are closed.
Finally, we choose an ambient isotopy $(\varphi^t)_{t\in [0,1]}$ of $T^*\R^3$ with a compact support which coincides with an extension of $(I_t)_{t\in [0,1]}$ in a neighborhood of the zero section, namely, $\varphi^t(L_{K_0})\cap D^*_r\R^3= I_t(L_{K_0})\cap D^*_r\R^3$ for some $r>0$ for every $t\in [0,1]$.
\end{proof}

\subsection{The Chekanov-Eliashberg DGA of the conormal torus of a knot and the framed knot DGA.}

In this subsection, any knot $K$ in $\R^3$ is equipped with an orientation. We often identify $L_K$ with a tubular neighborhood of $K$ in $\R^3$ in the same way as Example \ref{ex-gr} and $\Lambda_K$ is orientated as the boundary of $L_K \cap D^*_1\R^3$
(i.e. if $\nu$ is an outward vector at $q\in \Lambda_K$, $\R\nu \oplus T_q\Lambda_K \cong T_q L_K$ as oriented vector spaces, where $\R \nu$ has $\nu$ as a positive vector).

For an oriented knot $K$, we define its \textit{zero framing} by a section $l\colon K\to \Lambda_K$ such that the linking number of $K$ and $l(K)$ is $0\in \Z$.
The meridian of $K$ is represented by a fiber $(\Lambda_K)_q$ of any $q\in K$ and this circle is oriented so that $((dl)_q(v_1), v_2)$ is a positive basis of $T_{l(q)}\Lambda_K$
if we choose positive vectors $v_1$ in $T_qK$ and $v_2$ in $T_{l(q)}(\Lambda_K)_q$. Then, we fix an orientation of $S^1$ and take a diffeomorphism $m\colon S^1 \to (\Lambda_K)_q$ which preserves orientations.

We denote $\lambda= e^{[l]}$ and $\mu=e^{[m]}$, then $\Z[H_1(\Lambda_K)] = \Z[\lambda^{\pm},\mu^{\pm}]$.

\subsubsection{Framed knot DGA}
For any framed knot in $\R^3$, a DGA over $\Z[\lambda^{\pm},\mu^{\pm}]$, called the \textit{framed knot DGA}, is defined by \cite[Definition 2.8]{N}. (Before \cite{N}, a DGA over $\Z$ was defined in \cite{N1, N2}.) Note that there are several conventions from \cite[Definition 2.6]{N}, \cite[Section 1.2]{EENS} and \cite[Definition 3.11]{N-intro}, and they are related by replacing $\lambda$ and $\mu$ to their minus or inverse as summarized in \cite[Appendix]{N-intro}.
In this paper, we mainly refer to the last one by \cite{N-intro}, which \cite{c1, c2} refer to.

For an oriented knot $K$, let $(\mathcal{A}^{\mathrm{knot}}_*(K),\partial_K)$ denote the \textit{knot DGA} of an oriented knot $K$. See \cite[Definition 3.11]{N-intro} for the definition of a DGA over $\Z[\lambda^{\pm},\mu^{\pm},U^{\pm}]$ defined combinatorially from a braid representation of a knot $K$ (of a single component). If $U$ is substituted by $1$, then we get $(\mathcal{A}^{\mathrm{knot}}_*(K),\partial_K)$.

We consider a version of the Chekanov-Eliashberg DGA of $\Lambda_K$ over $\Z[H_1(\Lambda_K)]$ defined in the manner of \cite[Subsection 2.3]{EENS}
for a specific spin structure on $\Lambda_K$, which we denote by $\mathfrak{s}_{K}$.
For the specific choice of $\mathfrak{s}_K$, see Remark \ref{rem-spin} below. 
In fact, this DGA is fully non-commutative, that is, generated by elements of $H_1(\Lambda_K)$ and Reeb chords of $\Lambda_K$, and $A\in H_1(\Lambda_K)$ and Reeb chords do not commute.
We take the quotient of this DGA such that they commute.
As a DGA over $\Z[\lambda^{\pm},\mu^{\pm}]$, it is isomorphic to the DGA $(\mathcal{A}_*(\Lambda_K),\partial_{\mathfrak{s}_K})$ defined in Subsection \ref{subsec-app} via (\ref{diffeo-identify}). For more details, see Remark \ref{rem-isom-EENS}.

\begin{rem}\label{rem-spin}
We refer to \cite[Section 6.2]{EENS}.
Let $U$ be the unknot in $\R^3$. The zero framing gives a trivialization of $T\Lambda_U$ and $\mathfrak{s}_U$ is defined as the trivial spin structure. 
When $K$ is in a tubular neighborhood of $U$ and given as a closure of a braid, $\Lambda_K$ is contained in a $1$-jet neighborhood of $\Lambda_U$. Then $\mathfrak{s}_K$ is defined to be the pullback of $\mathfrak{s}_U$ by the projection in $J^1\Lambda_U$ of $\Lambda_K$ to the zero section $\Lambda_U$.
\end{rem}

\begin{rem}\label{rem-isom-EENS}
For a generic knot $K$ in $\R^3$, every Reeb chord $a\in \mathcal{R}(\Lambda_K)$ is non-degenerate and has degree $|a|\in \{0,1,2\}$, as discussed in \cite[Section 6.1]{CELN}. Let $\psi \colon \mathcal{A}_*(\Lambda_K) \to \mathcal{A}_*(\Lambda_K)$ be an isomorphism of graded $\Z[\lambda^{\pm},\mu^{\pm}]$-algebras determined by $\psi(a) \coloneqq (-1)^{\tau(a)}a$ for every $a\in \mathcal{R}(\Lambda_K)$, where $\tau(a)=0$ if $|a|\in \{0,1\}$ and $\tau(a)=1$ if $|a|=2$. Then, $(\psi^{-1}\circ \partial_{\mathfrak{s}_K}\circ \psi)(a)$ is given by the left-hand-side of (\ref{diff-DGA}) with the sign $(-1)^{(n-1)(|a|+1)}$ omitted (here, $n=2$). This coincides with the definition of the differential of the DGA in \cite[Subsection 2.3.4]{EENS}.
\end{rem}


\subsubsection{Constraints via Floer cohomology}

We return to consider a Hamiltonian diffeomorphism $\varphi = \varphi_H \in \Ham_c(T^*\R^3)$ for a Hamiltonian $H$ on $T^*\R^3$ with a compact support.
Let $K_0$ be an oriented knot in $\R^3$. Suppose that $\R^3$ and $\varphi(L_{K_0})$ have a clean intersection along $K=\R^3 \cap \varphi(L_{K_0})$. 
As explained in Example \ref{ex-gr}, the orientations of $L_{K_0}$ and $L_K$ are determined by the fixed orientation of $\R^3$. $\varphi(L_{K_0})$ is also oriented so that $\varphi\colon L_{K_0} \to \varphi(L_{K_0})$ preserves orientations.
\begin{lem}\label{lem-knot-ori}
Suppose that $K$ is connected. Then, $K$ is a knot in $\R^3$. Moreover, the embedding map (\ref{inclusion-clean}) for $L_1=\varphi (L_{K_0})$ preserves the orientations.
\end{lem}
\begin{proof}
By Lemma \ref{lem-neighborhood}, we may assume that $\varphi_H(L_{K_0}) \cap D^*_{\epsilon}\R^3= L_K \cap D^*_{\epsilon}\R^3$ for some compactly supported Hamiltonian $H$ on $T^*\R^3$.  Then, (\ref{inclusion-clean}) is the inclusion map
\[L_K \cap D^*_{\epsilon}\R^3 \to \varphi_H(L_{K_0}) .\]
We apply Proposition \ref{prop-Floer} and the invariance of Floer cohomology by $\varphi=\varphi_H$. If $K$ is connected, then for the graded Lagrangian submanifolds $\tilde{\R^3}$ and $\tilde{L}_{K_0}$ of Example \ref{ex-gr}, we have
\begin{align}\label{isom-K}
H^{*+\sigma_{K_0}}(K_0;\Z/2)\cong HF^*(\tilde{\R^3},\tilde{L}_{K_0}) \cong HF^*(\tilde{\R^3},\varphi_H(\tilde{L}_{K_0})) \cong H^{*+\sigma + \sigma_K}(K;\Z/2).
\end{align}
Here, $\sigma_{K},\sigma_{K_0}\in \Z$ are the integers from Remark \ref{rem-shift-degree} and $\sigma\in \Z$ is the difference of gradings of $\varphi_H(\tilde{L}_{K_0})$ and $\gr_{L_K}$ on $L_K \cap D^*_{\epsilon}\R^3$.
The isomorphism (\ref{isom-K}) shows that $\dim K=1$, so $K$ is a knot as well.

From the property of $\sigma_K$ and $\sigma_{K_0}$ discussed in Remark \ref{rem-shift-degree}, $\sigma_K=\sigma_{K_0}$. Indeed, for any tubular neighborhoods $N_0$ and $N$ of $K_0$ and $K$ in $\R^3$ respectively, we have a diffeomorphism of pairs $(N_0,K_0) \cong (N,K)$, by which $(T^*N_0, \tilde{N_0},\tilde{L}_{K_0})$ is identified with $(T^*N, \tilde{N}, \tilde{L}_K)$.
Therefore, (\ref{isom-K}) implies that $H^*(K_0;\Z/2) \cong H^{*+\sigma}(K;\Z/2)$, and thus  $\sigma=0$. This means that the two orientations $\mathfrak{o}_{\tilde{L}_{K}}$ and $\mathfrak{o}_{\varphi_H(\tilde{L}_{K_0})}$ coincide on $L_K \cap D^*_{\epsilon}\R^3$.
By Example \ref{ex-gr}, these two orientations coincide respectively with the orientations of $L_{K_0}$ and $\varphi_H(L_{K_0})$ fixed just before this lemma.
\end{proof}

Let $i_K\colon K \to \varphi(L_{K_0})$ denote the inclusion map as Proposition \ref{prop-Floer}.
\begin{lem}\label{lem-inc-isom}
If $K$ is connected, then
\[(i_K)^*\colon H^*(\varphi(L_{K_0});\Z/2) \to  H^*(K;\Z/2) \]
is an isomorphism.
\end{lem}
\begin{proof}
By Proposition \ref{prop-Floer} and the invariance of Floer cohomology of $\varphi =\varphi_H$,
\begin{align}\label{isom-L_K}
H^*(L_{K_0};\Z/2)\cong HF^*(\tilde{L}_{K_0},\tilde{L}_{K_0}) \cong HF^*(\varphi_H(\tilde{L}_{K_0}), \varphi_H(\tilde{L}_{K_0})) \cong H^*(\varphi(L_{K_0});\Z/2).
\end{align}
Since the diagram (\ref{diagram-m2}) is commutative, the isomorphisms (\ref{isom-K}) and (\ref{isom-L_K}) intertwine the cup products of Proposition \ref{prop-Floer}.
In particular,
\[(i_K)^*\colon H^*(\varphi(L_{K_0});\Z/2) \to  H^*(K;\Z/2) \colon \xi \mapsto (i_K)^*(\xi) \cup 1\]
is an isomorphism as well as the map $(i_{K_0})^*\colon H^*(L_{K_0};\Z/2) \to H^*(K_0;\Z/2)$.
\end{proof}

Via the two embeddings of (\ref{boundary-embedding})
\[ \begin{array}{cc} i'_+\colon \Lambda_{K_0}\to \varphi(L_{K_0})\setminus K,&  i'_- \colon \Lambda_{K}\to \varphi(L_{K_0})\setminus K, \end{array}\]
we think of $\Lambda_{K_0}$ and $\Lambda_{K}$ as embedded tori in $ \varphi(L_{K_0})\setminus K$.
We fix an orientation of $K$.
Then, $H_1(\varphi(L_{K_0})\setminus K)$ is a free $\Z$-module of rank $2$, whose basis is given by $\{[l_0],[m_0]\}$, where
\[l_0\coloneqq i'_+\circ l,\ m_0\coloneqq  i'_-\circ m .\]
Moreover, if we define $a,b\in \Z$ by
\[(i'_-)_*([l])=  a [l_0] + b [m_0] ,\]
 then
\[ (i'_+)_*([m]) =  (-1)^{\sigma} a[m_0] ,\]
where $(-1)^{\sigma}$ is the orientation sign of (\ref{inclusion-clean}). By Lemma \ref{lem-knot-ori}, it follows that $(-1)^{\sigma}=1$.
In addition, since $(i_K)^*$ is an isomorphism over $\Z/2$ by Lemma \ref{lem-inc-isom} and $(i_K)_*([K])= a[K_0]\in H_1(\varphi(L_{K_0});\Z/2)$, it follows that $a\equiv 1 \mod 2$.

In summary, $i'_{\pm}$ induce ring maps between the group rings described as follows: For $\lambda_0\coloneqq e^{[l_0]}$ and $\mu_0\coloneqq e^{[m_0]}$,
\begin{align*}
&R_+ \coloneqq \Z[H_1(\Lambda_+)] =\Z[\lambda^{\pm}, \mu^{\pm}], \ R_-  \coloneqq \Z[H_1(\Lambda_-)] =\Z[\lambda^{\pm}, \mu^{\pm}], \\
&R_0  \coloneqq \Z[H_1(\varphi(L_{K_0}) \setminus K)] = \Z[(\lambda_0)^{\pm}, (\mu_0)^{\pm}], \\
&(i'_+)_*\colon R_+\to R_0 \colon  \lambda \mapsto \lambda_0 ,\ \mu \mapsto  (\mu_0)^{ a}, \\
&(i'_-)_*\colon R_- \to R_0 \colon  \lambda \mapsto (\lambda_0)^a (\mu_0)^b ,\ \mu \mapsto  \mu_0 ,
\end{align*}
where $a,b\in \Z$ and $a \equiv 1 \mod 2$.

\subsubsection{Spin structures}
Let us observe the difference of spin structures. 
Let $\mathfrak{s}'_K$ be the spin structure on $\Lambda_K$ as the positive end of the filling $L_K\cap D^*_1\R^3 \cong S^1\times D^2$ whose spin structure is given by the trivial $\Spin(3)$-bundle. See Remark 3.6 by settng $\Lambda_+=\Lambda_K$ and $\Lambda_-=\emptyset$.
This differs from $\mathfrak{s}_K$ by $d(\mathfrak{s}_K,\mathfrak{s}'_K) \in H^1(\Lambda_K;\Z/2)$ which maps $[l]$ to $0$ and $[m]$ to $1$. 
Likewise, we define the spin structure $s'_{K_0}$ on $\Lambda_{K_0}$.
Then, $\mathfrak{s}'_{K_0}$ extends to a trivial spin structure on $\varphi(L_{K_0})$, and the spin structure on $\Lambda_K$ as the negative end of $\varphi(L_{K_0})\setminus K$ coincides with $\mathfrak{s}'_K$.
Here, note that the underlying orientation of $\mathfrak{s}'_K$ coincides with the orientation induced on $\Lambda_{K}$ as the negative end of $\varphi(L_{K_0})\setminus K$ as described in Remark \ref{rem-spin-boundary}.
This follows from the assertion in Lemma \ref{lem-knot-ori} about orientations.
In addition, the ring isomorphisms defined by (\ref{spin-change}) are
\begin{align*}
\psi (\mathfrak{s}_{K_0},\mathfrak{s}'_{K_0}) \colon R_+\to R_+\colon \lambda \mapsto \lambda,\ \mu\mapsto -\mu, \\
\psi (\mathfrak{s}_{K},\mathfrak{s}'_{K}) \colon R_-\to R_-\colon \lambda \mapsto \lambda,\ \mu\mapsto -\mu.
\end{align*}

By \cite[Appendix]{N-intro}, $r_{\pm}\colon R_{\pm}\to R_{\pm} \colon \lambda \mapsto -\lambda,\ \mu\mapsto -\mu$ relates the two conventions of \cite[Definition 3.11]{N-intro} and \cite[Section 1.2]{EENS}.
Let us define ring maps
\[\begin{array}{cc} I_+\coloneqq (i'_+)_*\circ \psi (\mathfrak{s}_{K_0},\mathfrak{s}'_{K_0}) \circ r_+, &  I_-\coloneqq (i'_-)_*\circ\psi (\mathfrak{s}_{K},\mathfrak{s}'_{K})\circ r_-. \end{array}\]

\subsubsection{Proof of main theorem}

Let $K_0$ and $K$ be oriented knots in $\R^3$. Let $\varphi \in \Ham_c(T^*\R^3)$. Suppose that $\varphi(L_{K_0})$ and the zero section $\R^3$ have a clean intersection along $K$.
In this setting, we get the following application of Theorem \ref{thm-clean-intersection}.

\begin{prop}\label{prop-kcDGA-map}
There exist integers
\[
a,b\in \Z \text{ with } a\equiv 1 \mod 2,
\]
two ring maps
\begin{align*}
&I_+\colon R_+ \to  R_0  \colon  \lambda \mapsto -\lambda_0 ,\ \mu \mapsto  (\mu_0)^{ a}, \\
&I_-\colon R_- \to R_0 \colon  \lambda \mapsto -(\lambda_0)^a (\mu_0)^b ,\ \mu \mapsto  \mu_0, 
\end{align*}
and a DGA map over $R_+$
\[ \Phi_{\varphi} \colon \mathcal{A}_*^{\mathrm{knot}}(K_0) \to \mathcal{A}_*^{\mathrm{knot}}(K)\otimes_{R_-} R_0. \] 
Here, the target object becomes a DGA over $R_+$ by $I_+$ and $I_-$ in the same way as (\ref{base-change}).
\end{prop}
\begin{proof}
Let $(\mathcal{A}_*(\Lambda_{K_0}),\partial_{\mathfrak{s}'_{K_0}})$ and $(\mathcal{A}_*(\Lambda_K), \partial_{\mathfrak{s}'_{K}})$ be the Chekanov-Eliashberg DGAs  defined with the choice of spin structures $\mathfrak{s}'_{K_0}$ and $\mathfrak{s}'_K$ respectively.
Theorem \ref{thm-clean-intersection} shows that there exists a DGA map over $R_+$
\begin{align}\label{DGA-map-pre}
 (\mathcal{A}_*(\Lambda_{K_0}),\partial_{\mathfrak{s}'_{K_0}}) \to (\mathcal{A}_*(\Lambda_K)\otimes_{R_-} R_0,\partial_{\mathfrak{s}'_K}\otimes \id_{R_0}),
 \end{align}
where the target object becomes a DGA over $R_+$ by $(i'_+)_*$ and $(i'_-)_*$.

For a general oriented knot $K$ in $\R^3$, the knot DGA $(\mathcal{A}_*^{\mathrm{knot}}(K),\partial_K)$ is related to $(\mathcal{A}_*(\Lambda_K), \partial_{\mathfrak{s}'_K})$ as follows: Consider the combinatorial framed knot DGA defined in \cite[Subsection 1.2]{EENS} over $\Z[\lambda^{\pm},\mu^{\pm}]$. We denote by $(\hat{\mathcal{A}}^{\mathrm{knot}}_*(K), \hat{\partial}_K)$ its quotient so that $\lambda$ and $\mu$ commute with all elements. Then, we have the following diagram:
\begin{align}\label{isom-of-DGAs}
\begin{split}
 \xymatrix@C=60pt{
(\mathcal{A}^{\mathrm{knot}}_*(K),\partial_K) \ar[r]^-{(\lambda,\mu)\mapsto (-\lambda,-\mu)}  &  (\hat{\mathcal{A}}^{\mathrm{knot}}_*(K), \hat{\partial}_K) \ar@<0.5ex>[d] \\
 (\mathcal{A}_*(\Lambda_{K}),\partial_{\mathfrak{s}'_K})  & (\mathcal{A}_*(\Lambda_{K}),\partial_{\mathfrak{s}_K})  \ar@<0.5ex>[u] \ar[l]_-{\psi(\mathfrak{s}_K,\mathfrak{s}'_K)} 
}
\end{split}
\end{align}
Here, the vertical two maps are DGA maps over $\Z[\lambda^{\pm},\mu^{\pm}]$ between DGAs which are stable tame isomorphic to each other by \cite[Corollary 1.2]{EENS}. The horizontal maps are isomorphisms of DGAs over $\Z$ which commute with the isomorphisms between the coefficient rings. For the upper one, we refer to \cite[Appendix]{N-intro}. For the lower one, see Proposition \ref{prop-change-spin} about changing the spin structures.

Return to the knot DGAs of $K_0$ and $K$.
The proposition follows from (\ref{DGA-map-pre}) and the diagram (\ref{isom-of-DGAs}) for $K_0$ and $K$.
\end{proof}

\begin{rem}\label{rem-a-pm} $\R^3$ and $\varphi(L_{K_0})$ admit brane structures with trivial Pin structures (i.e. those given by trivial $\operatorname{Pin}(3)$-bundles). 
Therefore, as mentioned in Remark \ref{rem-Z-coefficient}, we expect that Proposition \ref{prop-Floer} is extended to $\Z$-coefficients. If it is confirmed, we can prove that $(i_K)^*$ induces an isomorphism over $\Z$, and thus $a=\pm 1$. In particular, $I_{\pm}$ will be proved to be an isomorphism.
\end{rem}

We introduce other knot invariants related to the knot DGA.

The first invariant is a $\Z[\lambda^{\pm},\mu^{\pm}]$-algebra called the \textit{cord algebra} of an oriented knot $K$. This algebra is defined topologically and generated by the set $\pi_1(\R^3\setminus K)$. See \cite[Definition 4.4]{N-intro}.  Let $\mathrm{Cord}(K)$ denote this algebra. One may also refer to \cite[Section 2.2]{CELN} for a fully non-commutative version. By \cite[Theorem 4.7]{N-intro},
$\mathrm{Cord}(K)$ is isomorphic to $\mathcal{A}_0^{\mathrm{knot}}(K)/ \partial_K(\mathcal{A}_1^{\mathrm{knot}}(K))$, that is, the $0$-th degree part of the homology of the knot DGA. Here, we remark that $\mathcal{A}^{\mathrm{knot}}_p(K)=0$ for $p<0$. 
The next corollary is obvious from Proposition \ref{prop-kcDGA-map}.
\begin{cor}
There exists an $R_+$-algebra map from $\mathrm{Cord}(K_0)$ to $\mathrm{Cord}(K) \otimes_{R_-} R_0$ which becomes an $R_+$-algebra by $I_+$ and $I_-$ in the same way as (\ref{base-change}).
\end{cor}

The second invariant is related to the augmentations of DGAs.

\begin{defi}\label{defi-aug}Let $\bold{k}$ be a field.
For an oriented knot $K$, let $\aug(K; \bold{k})$ be the set of ring maps $\epsilon \colon \mathcal{A}^{\mathrm{knot}}_0(K)\to \bold{k}$ such that $\epsilon(1)=1$ and $\epsilon \circ \partial_K(a) =0$ for all $a\in \mathcal{A}^{\mathrm{knot}}_1(K)$. Such a ring map is called an \textit{augmentation} of $\mathcal{A}^{\mathrm{knot}}_*(K)$ over $\bold{k}$.
We also define a subset of $(\C^*)^2$
\begin{align*}
V_K\coloneqq \{(x ,y) \in (\C^*)^2 \mid \text{There exists }\epsilon \in \aug(K;\C)  \text{ such that }(\epsilon(\lambda), \epsilon(\mu))=(x,y) \}.
\end{align*}
It is called the \textit{augmentation variety} of $K$.
\end{defi}
\begin{rem}
We may refer to \cite[Definition 5.4]{N} using different conventions.
The above definition of $V_K$ is just equal to the intersection of the subspace in $(\C^*)^3$ defined by \cite[Definition 5.1]{N-intro} with $\{(x,y,U)\mid U=1\}$.
In addition, $V_K$ is independent of the orientation of $K$. Indeed, by \cite[Proposition 4.1]{N}, there is a bijection $V_K \to V_{-K}\colon (x,y) \mapsto (x^{-1},y^{-1})$, where  $-K$ denote the knot $K$ equipped with the opposite orientation.
\end{rem}

The main theorem (Theorem \ref{thm-intro-aug}) is deduced from Proposition \ref{prop-kcDGA-map}.

\begin{thm}\label{cor-aug}
For any knots $K_0$ and $K$ in $\R^3$ and
 $\varphi \in \Ham_c(T^*\R^3)$, suppose that $\R^3$ and $\varphi(L_{K_0})$ have a clean intersection along $K$.
Then, there exist $a,b\in \Z$ with $a\equiv 1 \mod 2$ such that the set
\[\{( x' ,y') \in (\C^*)^2 \mid \text{ there exists } (x,y)\in V_{K} \text{ such that } (x')^a = x y^{-b} \text{ and }y'= y^{ a} \}. \]
is contained in $V_{K_0}$.
\end{thm}
\begin{proof}
Suppose that $(x,y) = (\epsilon(\lambda), \epsilon (\mu))$ for $\epsilon \in \aug (K;\C)$ and $x'\in \C^*$ satisfies $(x')^a= xy^{-b}$. Then we define a ring map
\[ \delta\colon R_0\to \C \colon  \lambda_0 \mapsto -x',\ \mu_0 \mapsto y. \]
Since $a$ is an odd integer,
\begin{align*}
\delta\circ I_-(\lambda) &= \delta( -(\lambda_0)^a(\mu_0)^b) = -(-x')^ay^b=x, \\
\delta\circ I_-(\mu) &= \delta(\mu_0) =y ,
\end{align*}
so we have $\delta( I_-(k)) = \epsilon (k)$ for every $k\in R_-$. Therefore, we have a well-defined ring map
\[\tilde{\epsilon}\colon \mathcal{A}^{\mathrm{knot}}_0(K)\otimes_{R_-} R_0 \to \C \colon \xi \otimes \eta \mapsto \epsilon(\xi)\cdot \delta(\eta),\]
for which $\tilde{\epsilon} (1\otimes 1)=1$ and $\tilde{\epsilon}\circ (\partial_K\otimes \id_{R_0})=0$ hold.
Then, $\tilde{\epsilon}\circ \Phi_{\varphi} \in \aug(K_0;\C)$ for the DGA map $\Phi_{\varphi}$ of Proposition \ref{prop-kcDGA-map} and it satisfies
\begin{align*}
&\tilde{\epsilon}\circ \Phi_{\varphi} (\lambda)= \epsilon(1)\cdot \delta(I_+(\lambda)) = \delta(-\lambda_0)=x', \\
&\tilde{\epsilon}\circ \Phi_{\varphi} (\mu)= \epsilon(1)\cdot \delta(I_+(\mu)) = \delta((\mu_0)^{a})=y^{ a}.
\end{align*}
Therefore, $(x',y^{ a})\in V_{K_0}$.
\end{proof}

\subsection{Constraints on knot types by Theorem \ref{cor-aug}}
Throughout this subsection, $K_0,K$ are knots in $\R^3$ and $ \varphi \in \Ham_c(T^*\R^3)$, and we suppose that $\R^3$ and $\varphi(L_{K_0})$ have a clean intersection along $K$.
Theorem \ref{cor-aug} can be used to give strong constraints on the knot types of $K_0$ and $K$.

\begin{prop}\label{cor-unknot}
Suppose that $K_0$ is the unknot. Then, $K$ is also the unknot.
\end{prop}
\begin{proof}
We show that if $K$ is not the unknot, then $V_{K}$ contains $\{x=1\}\cup \{y=1\}$ as a proper subset. (This part is not new. See Remark \ref{rem-unknot} below). This assertion follows from the next four results about a general knot $K$: 
\begin{itemize}
\item $\{x=1\}\cup \{y=1\} \subset V_{K}$ (see \cite[Proposition 5.6]{N} and \cite[Exercise 3.17]{N-intro}).
\item The Zariski closure of $V_{K}$ in $(\C^*)^2$ contains
\[\{(x,y^2) \mid (1-y^2) A_K(x,y)=0\},\]
where $A_K(x,y)$ is the $A$-polynomial of $K$ (see \cite[Proposition 5.9]{N} and \cite[Theorem 5.9]{N-intro}).
\item $A_K(x,y)$ is not divisible by $(y-1)$ and $(y+1)$ (see \cite[Section 2.8]{CCGLS} and \cite[Proposition 5.9]{N}).
\item If $K$ is not the unknot, then $A_K(x,y)$ is not a power of $(x-1)$. (see \cite[Theorem 1.1]{DG}).
\end{itemize}
In addition, the augmentation variety of the unknot $K_0$ is $\{x=1\}\cup \{y=1\}$. See \cite[Example 3.13]{N-intro}.

By Theorem \ref{cor-aug},
$V_{K_0}=\{x=1\}\cup \{y=1\}$ contains $\{(x',-1)\mid (x')^a=(-1)^{-b}\}$ since $(1,-1)\in V_{K}$ and $a\equiv 1 \mod 2$. This implies that the following holds for every $x'\in \C^*$:
\[ (x')^a=(-1)^{-b} \Rightarrow (x',-1) \in V_{K_0} \cap \{y=-1\} =\{(1,-1)\} \Leftrightarrow x'=1.\]
This means that $a=\pm 1$.
In addition, $V_{K_0}=\{x=1\}\cup \{y=1\}$ contains $(2^{-b},2)$ since $(1,2^a)\in V_{K}$ and $a=\pm 1$.
This implies that $2^{-b}=1$ and thus $b=0$.

Applying Theorem \ref{cor-aug} again, since $a=\pm 1$ and $b=0$, we have
\[\{(x^a,y^{a})\mid (x,y) \in V_{K}\} \subset \{x=1\}\cup \{y=1\} . \]
Therefore, $V_{K} \subset \{x=1\}\cup \{y=1\}$, and thus $V_{K} = \{x=1\}\cup \{y=1\}$. This means that $K$ is the unknot.
\end{proof}

\begin{rem}\label{rem-unknot}
The proof that $\{x=1\}\cup \{y=1\} \neq V_K$ if $K$ is not the unknot is essentially contained in the proof of \cite[Proposition 5.10]{N}, which states that the cord algebra distinguishes the unknot from any other knot. We reviewed the proof for clarity.
\end{rem}

Next, let $3_1$ denote the right-hand trefoil knot and $\overline{3_1}$ denote its mirror. By \cite[Section 5.2]{N} and \cite[Exercise 4.6]{N-intro},
\[V_{3_1} = \{x=1\} \cup \{y=1\} \cup \{xy^3+1=0\},\ V_{\overline{3_1}} = \{x=1\} \cup \{y=1\} \cup \{x+y^{3}=0\}.\]
For any $m,m'\in \Z_{\geq 0}$, we denote by $K_{m,m'}\coloneqq (3_1)^{\# m} \# (\overline{3_1})^{\# m'}$ the connected sum of $m$ $3_1$ knots and $m'$ $\overline{3_1}$ knots. By the formula for connected sums of knots \cite[Proposition 5.8]{N},
\[V_{K_{m,m'}} = \{y=1\}\cup \bigcup_{k=-m'}^m \{x (-y)^{3k}=1\}.\]
\begin{prop}\label{prop-trefoil}
Suppose that $K_0= K_{m_0,m'_0}$ and $K=K_{m,m'}$. Then, $m_0+m'_0\geq m+m'$. Moreover, if $m_0+m'_0=m+m'$, then $m\equiv m_0 \mod 2$.
\end{prop}
\begin{proof}
We apply Theorem \ref{cor-aug}. First, $V_{K_{m_0,m'_0}}$ contains
$\{(x',-1) \mid (x')^a=(-1)^{-b}\}$ since $(1,-1)\in V_{K_{m,m'}}$and $a\equiv 1 \mod 2$.
This implies that the following holds for every $x'\in \C^*$:
\[ (x')^a=(-1)^{-b} \Rightarrow (x',-1)\in V_{K_{m_0,m'_0}} \cap \{y=-1\} = \{(1,-1)\} \Leftrightarrow x'=1.\]
This means that $a=\pm 1$ and $b \equiv 0 \mod 2$. Next, we have the following inclusions for subsets of $(\C^*)^2$: For any $k\in \Z$ with $ -m'\leq k\leq m$, 
\begin{align*}
& \{(x',y')\mid x'(-y')^{3k+b}=1\}  \\
\subset & \{(x',y')\mid x(-y)^{3k}=1 \text{ for } x\coloneqq (x')^ay^b\text{ and } y\coloneqq (y')^a \} \\ 
 \subset & \{(x',y')\mid \text{There exists }(x,y)\in V_{K_{m,m'}}\text{ such that }  (x')^a=x y^{-b}\text{ and } y'= y^a \} \\
\subset & V_{K_{m_0,m'_0}}.
\end{align*}
For the first inclusion, we use $a\equiv 1$ and $b\equiv 0 \mod 2$. For the second inclusion, we use $a=\pm 1$. For the third inclusion, we apply Theorem \ref{cor-aug}. Therefore, we have
\[  \bigcup_{k=-m'}^m \{x (-y)^{3k+b}= 1\} 
\subset  V_{K_{m_0,m'_0}} = \{y=1\}\cup \bigcup_{k=-m'_0}^{m_0} \{x (-y)^{3k}=1\} \]
This implies that $
3m+b\leq 3m_0$ and $ -3m'+b\geq -3m'_0$. In particular, $m_0+m'_0\geq m+m'$. 

Next, suppose that  $m_0+m'_0=m+m'$. Then the equalities $m+\frac{b}{3}=m_0$ and $-m'+\frac{b}{3}=-m'_0$ hold, which means that $\frac{b}{3}$ is an integer and $m\equiv m_0 \mod 2$ since $b\equiv 0 \mod 2$.
\end{proof}

\begin{rem}
By Proposition \ref{prop-trefoil}, if $K_0=3_1$, then $K\neq \overline{3_1}$.
This shows the necessity of the assumption that $\varphi$ has a compact support. Indeed, by the identification via the map $T^*\R^3 \to \C^3\colon  ((q_1,q_2,q_3), (p_1,p_2,p_3)) \mapsto (q_i+\sqrt{-1}p_i)_{i=1,2,3}$, we can define a Hamiltonian isotopy $(\varphi^t)_{t\in [0,1]}$ to be
$\varphi^t \colon T^*\R^3 \to T^*\R^3\colon (z_1,z_2,z_3)\mapsto (e^{\sqrt{-1}\pi t}z_1,z_2,z_3) $. 
Then, for any knot $K$ in $\R^3$, $\varphi^1(L_{K})$ and $\R^3 $ have a clean intersection along the mirror of $K$. 
Clearly, the support of this Hamiltonian isotopy is $T^*\R^3$, which is non-compact.
\end{rem}

Next, given $p,q\in \Z$ which are coprime and satisfy $1\leq p <|q|$, let $T_{(p,q)}$ denote the $(p,q)$-torus knot. Then, we claim that
\[V_{T_{(p,q)}} = \{y=1\}\cup \{xy^{pq-q}=(-1)^{p-1}\} \cup \bigcup_{k=1}^{p-1} \{x^ky^{(k-1)pq}=1\}.\]
We refer to \cite{c1, c2}. The claim follows from results below:
\begin{itemize}
\item 
For a general oriented knot $K$ in $\R^3$, its augmentation variety is equal to $\{y=1\}\cup U_K$, where $U_K$ is a subset of $\C^*\times (\C^*\setminus \{1\})$ determined by the set of all irreducible \textit{knot contact homology (KCH) representations} of $\pi_1(\R\setminus K)$. See \cite[Theorem 1.2]{c2}.
\item By \cite[Theorem 1.3]{c1} when $1\leq p <q$, $U_{T_{(p,q)}}$ is equal to the zero locus of the polynomial
\[(xy^{(p-1)q}+(-1)^p)\prod_{k=1}^{p-1} (x^ky^{(k-1)pq}-1).\]
To be precise, this polynomial from \cite[Theorem 1.3]{c1} is defined so that its zero locus is equal to the Zariski closure of $U_K$ for $K=T_{(p,q)}$. However, its proof in \cite[Subsection 5.1]{c1} determines $U_{T_{(p,q)}}$ itself, and it is a closed subset.
\item When $1\leq p < -q$, $T_{(p,q)}$ is the mirror of $T_{(p,-q)}$, so $V_{T_{(p,q)}}= \{(x,y^{-1})\mid (x,y)\in V_{T_{(p,-q)}}\}$ by \cite[Section 4.1]{N}.
\end{itemize}
\begin{prop}\label{prop-torus}
Suppose that $K_0=T_{(p_0,q_0)}$ and $K=T_{(p,q)}$. Then, $p_0\geq p$. Moreover, if $p_0=p\geq 2$, then $q_0= q$.
\end{prop}
\begin{proof} To prove that $p_0\geq p$, we may assume that $p\geq 3$ since Proposition \ref{cor-unknot} shows that $p=1$ if $p_0=1$.
For any $k\in \Z\setminus \{0\}$, let $\zeta_k\coloneqq e^{\frac{2\pi\sqrt{-1}}{k}}$, $\tau_{a,k}\coloneqq a^{-2} \left( \frac{(k-1)pq}{k}+b\right)$ and
\[ P_{a,k}\coloneqq (2^{-\tau_{a,k}}\zeta_{ak}  ,2) \in (\C^*)^2.\]
\begin{rem} 
The following argument works if we define $P_{a,k}=(R^{-\tau_{a,k}}\zeta_{ak},R)$ for any $R\in \R_{>0}$. Here, we take $R=2$ for simplicity.
\end{rem}
Then, for $k=1,\dots ,p-1$, we have
\begin{align}\label{P_ak-inclusion}
\begin{split}
P_{a,k}\in & \{(x',2)\mid (x')^{a}= x (2^{\frac{1}{a}})^{-b} \text{ for } x\in \C^* \text{ satisfying }  x^{k} (2^{\frac{1}{a}})^{(k-1)pq}=1 \} \\
\subset & \{(x',y^{ a})\mid (x')^a=xy^{-b} \text{ for }(x,y) \in V_{T_{(p,q)}}\}\\
\subset & V_{T_{(p_0,q_0)}} .
\end{split}
\end{align}
The last inclusion comes from Theorem \ref{cor-aug}.

If $|a(p-1)|\geq 3$, then $\zeta_{a(p-1)} \notin \R$ and thus
\[P_{a,p-1} \notin \{y=1\} \cup \{xy^{p_0q_0-q_0}=(-1)^{p_0-1}\}.\]
Therefore, $P_{a,p-1}$ must be contained in $\bigcup_{k=1}^{p_0-1}  \{x^k y^{(k-1)p_0q_0}=1\} $. In particular, there exists $k\in \{1,\dots ,p_0-1\}$ such that $(\zeta_{a(p-1)})^k\in \R_{>0}$, which means that $|a(p-1)|\leq p_0-1$. Then, it follows that $p\leq p_0$.

If $|a(p-1)|=2$, we have $a=\pm 1$ and $p=3$. Assume that $p_0=2$. Then $\tau_{a,2}=  \frac{3q}{2} + b$ and
\[P_{a,2}= (-2^{-(\frac{3q}{2}+b)},2) \in V_{T_{(2,q_0)}}= \{y=1\} \cup \{xy^{q_0}=-1\} \cup \{x=1\}.\]
Thus, $P_{a,2}$ must be contained in $\{xy^{q_0}=-1\}$, which means that $q_0=\frac{3q}{2}+b$. On the other hand, for a similar reason as $P_{a,2}$, we have
\begin{align*}
(2^{-(\frac{3q}{2}+b)},2) \in & \{(x',2)\mid (x')^{a}= x(2^{\frac{1}{a}})^{-b} \text{ for } x\in \C^* \text{ satisfying }  x^{2} (2^{\frac{1}{a}})^{3q}=1 \} \\
\subset & \{(x',y^a)\mid (x')^a=xy^{-b} \text{ for }(x,y) \in V_{T_{(3,q)}}\}\\
\subset & V_{T_{(2,q_0)}} .
\end{align*}
It follows that $(2^{-q_0},2)\in V_{T_{(2,q_0)}}$, but this is a contradiction. Therefore, $p_0\geq 3=p$.

Next, suppose that $p_0=p$. Then, we have $a =\pm 1$ from $|a(p-1)|\leq p_0-1$ and $p_0=p$.
If $p\geq 3$, then $(-1)^{p-1} \zeta_{a(p-1)}\notin \R_{>0}$ and $(\zeta_{a(p-1)})^k \notin \R_{>0}$ for $k=1,\dots ,p-2$. This means that
\[P_{a,p-1} \in \{(x,2) \mid (-1)^{p-1}x\notin \R_{>0}\} \cup \bigcup_{k=1}^{p-2} \{(x,2) \mid x^k\notin \R_{>0}\} ,\]
and thus
\[ P_{a,p-1} \notin \{y=1\}\cup \{ x y^{pq-q} = (-1)^{p-1}\} \cup \bigcup^{p-2}_{k=1} \{ x^k y^{(k-1)pq=1}=1\}.\]
Therefore, $P_{a,p-1}\in V_{T_{(p,q_0)}}$ from (\ref{P_ak-inclusion}) implies that $P_{a,p-1} \in \{x^{p-1}y^{(p-2)pq_0}=1\}$. From the definition of $P_{a,p-1}$, we have $-(p-1)\tau_{a,p-1} + (p-2)pq_0=0$, which is equivalent to
\begin{align}\label{pqb}
 (p-1)b=(p-2)p(q_0- q).
\end{align}
\begin{itemize}
\item 
If $p\geq 5$, then $\zeta_{a(p-2)} \notin \R$ and thus
\begin{align}\label{P-inc}
P_{a,p-2} \notin \{y=1\} \cup \{xy^{pq_0-q_0}=(-1)^{p-1}\}.
\end{align}
Therefore, $P_{a,p-2}\in V_{T_{(p,q_0)}}$ from (\ref{P_ak-inclusion}) implies that $P_{a,p-2} \in \bigcup_{k=1}^{p-1} \{x^ky^{(k-1)pq_0}=1\}$. Moreover, we have $P_{a,p-2}\in \{x^{p-2}y^{(p-3)pq_0}=1\}$ since $(\zeta_{a(p-2)})^{k}\notin \R_{>0}$ for $k=1,\dots ,p-3,p-1$. Therefore, $(p-2)b= (p-3)p(q_0- q)$ holds.
Combining with (\ref{pqb}), we have $q_0=q$.
\item If $p=4$, then $P_{a,p-2}=P_{a,2}=(-2^{-(2q+b)},2)$. Assume that (\ref{P-inc}) does not holds for $p=4$. Then, $P_{a,2}\in \{xy^{3q_0}=-1\}$, which means that $3q_0=2q+b$. Combining with (\ref{pqb}), we have $q_0=-2 q$. Since $p=4$ and $q_0$ are coprime, this is a contradiction.  Therefore, (\ref{P-inc}) holds for $p=4$, and we can prove that $q_0=q$ in the same way as the case of $p\geq 5$.
\item If $p=3$, then $(2^{-(2q+b)},2)\in V_{T_{(3,q_0)}}$ from Theorem \ref{cor-aug} since $(2^{-2aq},2^{ a})\in V_{T_{(3.q)}}$. Therefore,
\[ (2^{-(2q+b)},2) \in  \{xy^{2q_0}=1\}\cup \{x=1\} \cup \{ x^2y^{3q_0}=1\} .\]
Thus, either $ b=2(q_0- q)$, $b=-2q$ or $b=  \frac{3 q_0}{2}-2 q$ holds. Assuming the first case, it is obvious that (\ref{pqb}) means $q_0= q$. Assuming the second or the third case, (\ref{pqb}) means that $-4 q=3(q_0- q)$ or $3q_0-4q=3(q_0-q)$. This is a contradiction since $p=3$ and $q$ are coprime.
\end{itemize}
If $p_0=p=2$, we have 
\[ \{ (-2^{-(q+b)},2), (2^{-b},2)\} \subset V_{(T_{(2,q_0)})} = \{y=1\} \cup  \{xy^{q_0}=-1\} \cup \{x=1\}\]
from Theorem \ref{cor-aug} since $\{(-2^{-a q},2^{ a}) , (1,2^{ a}) \}\subset V_{T_{(2,q)}}$.
This means that $q+b=q_0$ and $ b=0$, and thus $q_0= q$. This completes the proof.
\end{proof}

\begin{rem}
Theorem \ref{cor-aug} would be improved by showing $a=\pm 1$ as mentioned in Remark \ref{rem-a-pm}.
However, one can see that $a=\pm 1$ is obtained in the proof of Proposition \ref{cor-unknot} and \ref{prop-trefoil}. Thus, the outcome is not changed even if Theorem \ref{cor-aug} is improved by $a=\pm 1$.
About Proposition \ref{prop-torus}, it seems that the result is not improved significantly by $a=\pm 1$.
Indeed, we cannot deduce the equality $(p,q)=(p_0,q_0)$ (without assuming $p=p_0$ in advance) even when we assume $a=\pm 1$.

For example, consider the condition that $p\equiv p_0 \mod 2$, $q=rq_0$ for some $r\in \Z_{\geq 1}$ and
$r(p-k)+k$ divides $p_0$ for every $k=1,\dots ,p-1$.
When $q=q_0$ (i.e. $r=1$), it is equivalent to that $p\equiv p_0 \mod 2$ and $p$ divides $p_0$. When $p$ is small and $q\neq q_0$, $((p,q),(p_0,q_0))=((2,3m),(4,m))$ and $((3,7n),(45,n))$, where $m>4$ is an odd integer and $n>45$ is coprime to $45$, satisfiy this condition.
Then, under this condition on $(p,q)$ and $(p_0,q_0)$,
\begin{align*}
&\{xy^{pq-q+b}=(-1)^{p-1}\}\subset \{xy^{p_0q_0-q_0}=(-1)^{p_0-1}\},\\
&\{x^ky^{(k-1)pq +kb}=1\}\subset \{x^{l_k} y^{(l_k-1)p_0q_0}=1\}, \text{ where } l_k\coloneqq \frac{kp_0}{r(p-k)+k}\text{ for }k=1,\dots ,p-1,
\end{align*}
hold for $b=(p_0-1-rp+r)q_0$, and
it is straightforward to check that $V_{K_0}$ for $K_0=T_{(p_0,q_0)}$ and $V_{K}$ for $K=T_{(p,q)}$ satisfy the relation of Theorem \ref{cor-aug} for $a=\pm 1$, $b=(p_0-1-rp+r)q$.  
\end{rem}

\appendix

\section{Appendix: Orientations of moduli spaces and proof of Theorem \ref{thm-DGA-map}}\label{sec-proof}
In this appendix, we fix orientations of moduli spaces and prove  Theorem \ref{thm-DGA-map} about the existence of a DGA map between Chekanov-Eliashberg DGAs over the integral group rings.
It is proved in a parallel way as the result of \cite{K} over $\Z$, but we use different conventions from \cite{EES-ori}.

\subsection{Notation and convention of signs}\label{subsec-convention}
For any finite dimensional $\R$-vector space $V$, let $\topwedge V$ denote the top exterior power $\bigwedge^{\dim V} V$.
We use the same conventions as in \cite[Subsection 3.2.1]{EES-ori} about the following isomorphisms:
\begin{itemize}
\item Given a vector space $V$, we define $\topwedge V \to \topwedge V^*$ up to multiplications by positive scalars by
\begin{align}\label{isom-dual}
v_1\wedge \dots \wedge v_n \mapsto v_1^*\wedge \dots \wedge v_n^*.
\end{align}
Here, $(v_1,\dots ,v_n)$ is a basis of $V$ and $(v_1^*,\dots ,v_n^*)$ is its dual basis.
\item Given an ordered set of vector spaces $\{V_1,\dots ,V_l\}$, we define
\[\topwedge (V_1\oplus \dots \oplus V_l) \to \topwedge V_1 \otimes \dots \otimes \topwedge V_l\]
by
\[v^1_1\wedge \dots \wedge v^1_{n_1}\wedge \dots \wedge  v^l_1\wedge \dots \wedge v^l_{n_l} \mapsto (v^1_1\wedge \dots \wedge v^1_{n_1})\otimes \dots \otimes ( v^l_1\wedge \dots \wedge  v^l_{n_l} ).\]
Here, $(v^k_1,\dots ,v^k_{n_k})$ is a basis of $V_k$ for $k=1,\dots ,l$.
\item Given an exact sequence
\begin{align}\label{original-exact-seq}
\xymatrix{
0\ar[r] & V_1 \ar[r]^-{\alpha} & W_1 \ar[r]^-{\beta} & W_2 \ar[r]^-{\gamma} & V_2 \ar[r] &0, }\end{align}
we define an isomorphism
\begin{align}\label{isom-exact}
\topwedge V_1\otimes \topwedge V_2^* \to \topwedge W_1 \otimes \topwedge W_2^*
\end{align}
by
\begin{align*}
&(v^1_1\wedge \dots \wedge v^m_1) \otimes ((\gamma w^1_2)^*\wedge \dots \wedge (\gamma w^s_2)^*) \\
\mapsto & ((\alpha v^1_1)\wedge \dots \wedge (\alpha v^m_1) \wedge w^{m+1}_1\wedge \dots w^n_1) \\
& \otimes ((\beta w^n_1)^*\wedge \dots \wedge (\beta w^{m+1}_1)^*\wedge (w^1_2)^*\wedge \dots \wedge (w^s_2)^*).
\end{align*}
About the choice of these vectors and the well-definedness of (\ref{isom-exact}), see \cite[Subsection 3.2.1]{EES-ori}.
\end{itemize}
For an ordered set of vector spaces $\{V_1,\dots ,V_l\}$, we use a notation by a column
\[ \begin{bmatrix} V_1 \\ \vdots \\ V_l \end{bmatrix} \coloneqq V_1\oplus \cdots \oplus V_l . \]
\begin{rem}\label{rem-cancel}
Compare (\ref{original-exact-seq}) with an exact sequence
\begin{align*}
\xymatrix{ 0\ar[r] &  {\begin{bmatrix} U \\  V_1 \end{bmatrix}} \ar[r]^-{\alpha'} & {\begin{bmatrix} U \\  W_1 \end{bmatrix}} \ar[r]^-{  \beta'} & W_2 \ar[r]^-{\gamma} & V_2 \ar[r] &0 }
\end{align*}
such that $\alpha'(0,v)=(0,\alpha(v))$, $\beta'(0,w)=\beta(w)$, and the restriction of $\alpha'$ to the $U$-components is $\id_U$.
The isomorphism induced by this exact sequence
is the same as the tensor product of (\ref{isom-exact}) and the identity map on $\topwedge U$. Thus, when $\topwedge V_1\otimes \topwedge V_2^*$ and $\topwedge W_1\otimes \topwedge W_2^*$ are oriented, the sign of  (\ref{isom-exact}) is the same as the sign of the isomorphism induced by the above exact sequence.
\end{rem}

In the latter part of this appendix, we frequently use exact sequences defined by linear gluing arguments. We refer to \cite[Subsection 3.3.2]{EES-ori}. Note that the orders of vector spaces in the columns are different from the exact sequences used in \cite[Theorem 4.2]{K}.

Lastly, for every $n\in \Z_{\geq 0}$, we fix an orientation of $\R^n$ by $dx_1\wedge \dots \wedge dx_n$, where $(x_1,\dots ,x_n)$ is the coordinate of $\R^n$. Let  $\mathcal{L}(U(n))$ be the free loop space of $U(n)$.
Then, by \cite[Definition 3.10]{EES-ori}, we define a canonical orientation on the determinant line bundle over $\mathcal{L}(U(n))$. More precisely, the fiber of $A\in \mathcal{L}(U(n))$ is the determinant line of 
a $\Dvar$-operator on the disk $D$ (with no puncture) whose Lagrangian boundary condition is given by $A(z)(\R^n)$ for every $z\in \partial D$. In other words, $A\colon \partial D\to U(n)$ gives a trivialized boundary condition for the $\Dbar$-operator on $D$.
For the terminology of a \textit{trivialized boundary condition} for a $\Dbar$-operator, see \cite[Subsection 3.1]{EES-ori}.

\subsection{Capping operators}\label{subsec-cap}

Let $\Lambda$ be a compact Legendrian submanifold of $(P\times \R, dz -\lambda)$. As in Subsection \ref{subsec-setup}, it is a connected spin manifold and the Maslov class of $\pi_P(\Lambda)$ vanishes. Moreover $\pi_P(\Lambda)\subset P$ consists only of transverse double points. We fix an orientation of $\Lambda$ and its spin structure $\mathfrak{s}$.

In addition, let $J$ be a almost complex structure on $P$ compatible with $d\lambda$.
We assume that $J$ is adapted to $\Lambda$ and that $\Lambda$ is admissible in the sense of the definitions in \cite[Subsection 2.3]{EES}. Then, for each Reeb chord $c \in\mathcal{R}(\Lambda)$, there exists a neighborhood $B_c$ of $\pi_P(c(0))$ on which $J$ is integrable. We take a unitary trivialization $\Psi_c\colon \rest{TP}{B_c}\to \underline{\C}^n$. We also fix a positively oriented framing $X_c$ of $(\gamma_c)^*T\Lambda$.

For every $(c\colon [0,T]\to P \times \R)\in \mathcal{R}(\Lambda)$,
identifying $B_c$ with an open subset of $\C^n$,
we define a pair of \textit{capping operators} in the same way as in \cite[Subsection 3.3]{EES-ori}. Following the notations of \cite[Subsection 4.5.4]{K}, we denote them by $(\Dbar_{l, c,+}, \Dbar_{l, c,-})$.
These are $\Dbar$-operators on $D_1$ with trivialized boundary conditions $R_{c,\pm}\colon \partial D_1\to U(n)\times U(2)\subset U(n+2)$ which are stabilized, and $R_{c,\pm}$ is determined by the pair $((\pi_P)_*(T_{c(0)}\Lambda) ,(\pi_P)_*(T_{c(T)}\Lambda) )$ of Lagrangian subspaces of $T_{\pi_P(c(0))} P \cong \C^n$ and the choice of the frames of $T_{c(0)}\Lambda$ and $T_{c(T)}\Lambda$ fixed by $X_c$.  The correspondence with the notations of \cite[Subsection 3.3.4]{EES-ori} is as follows: $R_{c,+}= \begin{cases} R_{p,o} &  \text{ if }|c| \text{ is odd},\\ R_{p,e}& \text{ if }|c| \text{ is even,}\end{cases} $  $R_{c,-}= \begin{cases} R_{n,o} &  \text{ if }|c| \text{ is odd},\\ R_{n,e}& \text{ if }|c| \text{ is even.}\end{cases} $
From the computations of \cite[Subsection 3.3.6]{EES-ori},
\begin{align*}\begin{array}{ll}
\dim \ker \Dbar_{l,c,+} \equiv 0,  & \dim \coker \Dbar_{l,c,+} \equiv |c|+n-1, \\
\dim \ker \Dbar_{l,c,-} \equiv 1, & \dim \coker \Dbar_{l,c,-} \equiv |c|,
\end{array}
\end{align*}
modulo $2$. By \cite[Lemma 3.4]{EES-ori}, there exists an exact sequence by a linear gluing argument
\begin{align}\label{ori-cap-l}
\xymatrix{
0\ar[r] &\ker \Dbar_{l,c} \ar[r] & {\begin{bmatrix} \ker \Dbar_{l,c,+} \\  \ker \Dbar_{l,c,-} \end{bmatrix}} \ar[r] 
& {\begin{bmatrix} \coker \Dbar_{l,c,+} \\ \coker \Dbar_{l,c,-} \end{bmatrix}} \ar[r] & \coker \Dbar_{l,c} \ar[r] &0.
}
\end{align}
Here, $\Dbar_{l,c}$ is a $\Dbar$-operator on the disk $D$ with a trivialized boundary condition obtained by gluing $R_{c,+}$ and $R_{c,-}$.

We also define $\Dbar_{s,c,+}$ and $\Dbar_{s,c,-}$ as in \cite[Subsection 4.5.4]{K}, which we call the \textit{capping operators} as well. They are $\Dbar$-operators on $D_1$ with the trivialized boundary conditions
\begin{align}\label{stab-R} \tilde{R}_{c,\pm} \colon \partial D_1 \to U(1)\times U(n+2)\subset U(n+3) \colon z\mapsto \begin{pmatrix} 1 & 0 \\ 0 & R_{c,\pm}(z)
\end{pmatrix},
\end{align}
and defined on Sobolev spaces with exponential weight $(\epsilon, \dots , \epsilon)\in \R^{n+3}$ for a small $\epsilon>0$.
By \cite[Proposition 4.11]{K}, there are isomorphisms
\begin{align}\label{isom-ker-coker}
\begin{array}{cc}
\ker \Dbar_{s,c,\pm} \to \ker \Dbar_{l,c,\pm},& \coker \Dbar_{s,c,\pm} \to \coker \Dbar_{l, c,\pm} 
\end{array}
\end{align}
defined by composing with the projection $\C\oplus \C^{n+2}\to \C^{n+2}$.
Thus, we have natural isomorphisms
\begin{align}\label{isom-proj}
\begin{array}{cc}
\pi_+ \colon \det \Dbar_{s,c,+}\to \det \Dbar_{l,c,+},&  \pi_- \colon \det \Dbar_{s,c,-}\to \det \Dbar_{l,c,-}.
\end{array}
\end{align}
In addition,
by gluing the boundary conditions $\tilde{R}_{c,+}$ and $\tilde{R}_{c,-}$ of (\ref{stab-R}), we obtain a $\Dbar$-operator $\Dbar_{s,c} $ on the disk $D$. Then, there exists an exact sequence by a linear gluing argument
\begin{align}\label{ori-cap-s}
\xymatrix{
0\ar[r] &\ker \Dbar_{s,c} \ar[r] & {\begin{bmatrix} \ker \Dbar_{s,c,+} \\ \R \\ \ker \Dbar_{s,c,-} \end{bmatrix}} \ar[r] 
& {\begin{bmatrix} \coker \Dbar_{s,c,+} \\ \coker \Dbar_{s,c,-} \end{bmatrix}} \ar[r] & \coker \Dbar_{s,c} \ar[r] &0.
}
\end{align}

Let us define the orientations of the determinant lines of capping operators by the following steps:
We first fix an orientation of $\det \Dbar_{s,c,-}$ for every $c\in \mathcal{R}(\Lambda)$. 
Then, the isomorphism (\ref{isom-exact}) induced by (\ref{ori-cap-s}) and the canonical orientation of $\det \Dbar_{s,c}$ determine the orientation of $\det \Dbar_{s,c,+}$.
We assign an orientation to $\det \Dbar_{l,c,-}$ so that the isomorphism $\pi_-$ \textit{reverses} orientations. Finally, the isomorphism (\ref{isom-exact}) induced by (\ref{ori-cap-l}) and the canonical orientation of $\det \Dbar_{l,c}$ determine the orientation of $\det \Dbar_{l,c,+}$.

\begin{lem}\label{lem-cap}
The isomorphism $\pi_+$ reverses the orientations.
\end{lem}
\begin{proof}While the conventions are different,
the proof goes in a similar way as in \cite[Lemma 4.13]{K}. There exist natural isomorphisms
$\ker \Dbar_{s, c}\to \R\oplus \ker \Dbar_{l,c}$ and $\coker \Dbar_{s,c}\to \coker \Dbar_{l,c}$, where $\R$ corresponds to the subspace in $\ker \Dbar_{s,c}$ spanned by the constant function to $(1,0)\in \C\oplus \C^{n+2}$.
They define an isomorphism $\det \Dbar_{s,c} \to \R\otimes \det \Dbar_{l,c}$ and this preserves the canonical orientations.
Together with the isomorphisms (\ref{isom-ker-coker}), we obtain from (\ref{ori-cap-s}) an exact sequence
\[
\xymatrix{
0\ar[r] & {\begin{bmatrix}\R \\ \ker \Dbar_{l,c} \end{bmatrix}} \ar[r] & {\begin{bmatrix} \ker \Dbar_{l,c,+} \\ \R \\  \ker \Dbar_{l,c,-} \end{bmatrix}} \ar[r] 
& {\begin{bmatrix} \coker \Dbar_{l,c,+} \\ \coker \Dbar_{l,c,-} \end{bmatrix}} \ar[r] & \coker \Dbar_{l,c} \ar[r] &0
}
\]
Since $\dim \ker \Dbar_{l,c,+} \equiv 0$ modulo $2$, it does not cost a sign to change the order in the second column by moving $\R$ to the top. In addition, the restriction of the second map to the $\R$-components is $\R\to \R \colon v\mapsto v$. As Remark \ref{rem-cancel}, the exact sequence after changing the order induces an isomorphism which has the same sign as the one induced by (\ref{ori-cap-l}), which preserves the orientations by definition. 
Since $\pi_-$ reverses orientations, $\pi_+$ also reverses orientations.
\end{proof}

\subsection{Orientation of $\mathcal{M}_{L,B}(a; b_1,\dots ,b_m)$}\label{subsec-ori-Lag}

Let $L$ be an exact Lagrangian cobordism from $\Lambda_-$ to $\Lambda_+$ considered in Subsection \ref{subsec-DGA-map}.
Let $\mathfrak{s}_0$ be a spin structure on $L$.
For short, we use the notation $\hat{\mathcal{M}}$ for the space (\ref{moduli-hat})
and $\mathcal{M}$ for the moduli space (\ref{moduli-cobordism}).

For every $(u,\kappa) \in \hat{\mathcal{M}}$, let $\bar{u}\colon D_{m+1}\to P$ be the $P$-component of $u$. We take a unitary trivialization of $\bar{u}^* TP$ which coincides with the pullback of $\Psi_a$ (resp. $\Psi_{b_k}$) near $p_0$ (resp. $p_k$). This defines a unitary trivialization $\Psi_u \colon u^* T(\R\times (P\times \R))  \to \C\oplus \C^n$ so that $\partial_a$ is mapped to $(1,0)$. Note that from the definition (\ref{alm-cpx-str}) of $\tilde{J}$, $\partial_z$ is mapped to $(-\sqrt{-1},0)$.


Except the differences of the conventions, the process to define the orientation of $\hat{\mathcal{M}}$ is the same as in \cite[Subsection 4.6]{K}.
$\hat{\mathcal{M}}$ is the zero set of a section $\Gamma$ of a Banach bundle over a Banach manifold.
The linearization of $\Gamma$ at $(u,\kappa)\in \hat{\mathcal{M}}$ has the form $(d\Gamma)_{(u,\kappa)} = d_u\Gamma  \oplus d_{\kappa}\Gamma $, where
\begin{align*}
 d_u\Gamma & \colon H_{2,\boldsymbol{\epsilon}} [\lambda_u] (D_{m+1}; \C^{n+1}) \oplus \R^{m+1} \to H_{1,\boldsymbol{\epsilon}} [0] (D_{m+1}; T^{*0,1}D_{m+1} \otimes \C^{n+1}) , \\
d_{\kappa}\Gamma & \colon T_{\kappa}\mathcal{C}_{m+1} \to H_{1,\boldsymbol{\epsilon}} [0] (D_{m+1}; T^{*0,1}D_{m+1} \otimes \C^{n+1}),
\end{align*}
and  $d_u\Gamma$ is a linear $\Dbar$-operator with the boundary condition
\[ \lambda_u \coloneqq \Psi_u ( (\rest{u}{\partial D_{m+1}})^*TL),\]
which is a bundle of Lagrangian subspaces of $\C^{n+1}$ over $\partial D_{m+1}$.
Here, for each $i\in \{0,\dots ,m\}$, the vector $(0,\dots ,0,1,0,\dots ,0) \in \R^{m+1}$, whose $(i+1)$-th component is $1$, corresponds to a $C^{\infty}$ function $f_i\colon \colon D_{m+1} \to \C^{n+1}$ which is supported in a neighborhood of $p_i$ and constant to $(1,0) \in \C\oplus \C^n$ near $p_i$.
For the definition of the weighted Sobolev spaces, see \cite[Subsection 3.1.1]{EES-ori}.
Here, we take $\boldsymbol{\epsilon}=(\epsilon, \dots ,\epsilon)\in  \R^{n+1}$ for $\epsilon>0$ and this defines the weight of Sobolev spaces. We choose $\epsilon>0$ to be sufficiently small so that $d_u\Gamma$ is a Fredholm operator.
Any element $(f, (a_0,\dots ,a_m))\in H_{2,\boldsymbol{\epsilon}}[\lambda_u](D_{m+1};\C^{n+1}) \oplus \R^{m+1}$ is identified with a function $f+\sum_{i=0}^{m}a_if_i$ on $D_{m+1}$.

We take $\alpha_{\pm}\colon \Lambda_{\pm}\to \R$ satisfying the conditions as in \cite[Subsection 3.4.2]{EES-ori} near the end points of every Reeb chords. Let $\alpha'\colon L\to \R$ be a smooth function such that $\alpha'(a,p)=\alpha_+(p)$ if $a\geq a_+$ and $\alpha'(a,p)=\alpha_-(p)$ if $a\leq a_-$.
We then stabilize $\lambda_u$ to
\[\tilde{\lambda}_u\coloneqq \lambda_u \oplus (\rest{u}{\partial D_{m+1}})^*l,\]
where $l_q$ for every $q\in L$ is  a Lagrangian subspace of $\C^2=\C e_1\oplus \C e_2$ with a framing $(f_1,f_2)$ defined by
\[ f_1\coloneqq  e^{\sqrt{-1}\alpha'(q) /2} e_1,\ f_2\coloneqq e^{\sqrt{-1}\alpha'(q)}e_2 . \]

By closing up $r_{\bar{L}}\circ \rest{u}{\partial D_{m+1}}$ by the capping paths, we obtain a loop $\gamma_u \colon S^1\to \bar{L}$. Along any capping path, $(\gamma_u)^* (TL \oplus l)$ has a trivialization determined by the $X_c$ and $(f_1,f_2)$.
We extend it to an isometric trivialization on $S^1$ which lifts to a trivialization along $\gamma_u$ of the principal spin bundle for $\mathfrak{s}_0$.
We restrict this trivialization to the domain of $r_{\bar{L}}\circ \rest{u}{\partial D_{m+1}}$ and pull it back by $r_{\bar{L}}$, then we obtain a trivialization $\eta_u \colon  (\rest{u}{\partial D_{m+1}})^*(TL\oplus l) \to \R^{n+3}$ preserving the fiber metrics.
Its $\C$-linear extension gives
\[\eta_u^{\C}\colon (\rest{u}{\partial D_{m+1}})^*(T(\R\times (P\times \R))\oplus \C^2 \to \C^{n+3}.\]
%
%
This defines $A_u \colon \partial D_{m+1} \to U(n+3)$ by
\[A_u (z)\coloneqq ((\Psi_u)_z \oplus \id_{\C^2})\circ (\eta^{\C}_u)_z^{-1}\]
for every $z\in \partial D_{m+1}$.
Now we define
\[\Dbar_u\colon H_{2,\boldsymbol{\epsilon}} [\tilde{\lambda}_u] (D_{m+1}; \C^{n+3}) \oplus \R^{m+1} \to H_{1,\boldsymbol{\epsilon}} [0] (D_{m+1}; T^{*0,1}D_{m+1} \otimes \C^{n+3}) \]
to be the $\Dbar$-operator with a trivialized boundary condition $A_u$. 
Here the exponential weight is $\boldsymbol{\epsilon}=(\epsilon, \dots , \epsilon)$ for a small $\epsilon>0$. In addition, $\ker d_u\Gamma \cong \ker \Dbar_u$ and $\coker d_u\Gamma \cong \coker \Dbar_u$ canonically (i.e. the kernels (resp. the cokernels) of $d_u\Gamma$ and $\Dbar_u$ after the stabilization are canonically isomorphic) since the $\Dbar$-operator on $D_{m+1}$ with a boundary condition $(\rest{u}{\partial D_{m+1}})^*l$ has a trivial kernel and a trivial cokernel as noted in \cite[Subsection 3.3.5]{EES-ori}.


If $\Gamma$ is transverse to the zero section, then $(d\Gamma)_{(u,\kappa)}$ is surjective for every $(u,\kappa)\in \hat{\mathcal{M}}$, and thus we get an exact sequence
\begin{align}\label{seq-Gamma}
\xymatrix{
0 \ar[r] & \ker \Dbar_u \ar[r] & \ker (d\Gamma)_{(u,\kappa)} \ar[r] & T_{\kappa}\mathcal{C}_{m+1} \ar[r] & \coker \Dbar_u \ar[r] &0.
}
\end{align}
This induces an isomorphism (\ref{isom-exact}) to be
\begin{align}\label{isom-moduli}
\begin{split}
\det \Dbar_u \to & \  \topwedge\ker (d\Gamma)_{(u,\kappa)} \otimes \topwedge (T_{\kappa}\mathcal{C}_{m+1})^* \\
&= \topwedge T_{(u,\kappa)} \hat{\mathcal{M}}  \otimes \topwedge (T_{\kappa}\mathcal{C}_{m+1})^* .
\end{split}
\end{align}

By a linear gluing argument, we obtain an exact sequence
\begin{align}\label{ori-u}\xymatrix{
0 \ar[r] & \ker \Dbar_{\hat{u}} \ar[r] &
{\begin{bmatrix}\ker \Dbar_{s,b_1,-} \\ \vdots \\  \ker \Dbar_{s,b_m,-} \\ \ker \Dbar_{s,a,+} \\ \ker \Dbar_u \end{bmatrix}} \ar[r]
&{\begin{bmatrix}\coker \Dbar_{s,b_1,-} \\ \vdots \\  \coker \Dbar_{s,b_m,-} \\ \coker \Dbar_{s,a,+} \\ \coker \Dbar_u \end{bmatrix}} \ar[r]& \coker  \Dbar_{\hat{u}}\ar[r] &0.
}\end{align}
Here, $\Dbar_{\hat{u}}$ is a $\Dbar$-operator on the disk $D$ whose trivialized boundary condition $A_{\hat{u}}$ is given by gluing $\tilde{R}_{b_k,-}$ ($k=1,\dots ,m$), $\tilde{R}_{a,+}$ of (\ref{stab-R}) and $A_{u}$. (If we reverse the order of the vector spaces in the middle two columns, we obtain the same sequence as in \cite[(19)]{K}.)
By the isomorphism (\ref{isom-exact}) induced by (\ref{ori-u}), the canonical orientation of $\det \Dbar_{\hat{u}}$ and the orientations of capping operators determine an orientation of $\det \Dbar_u$.

We use the orientation of $T_{\kappa}\mathcal{C}_m$ defined in \cite[Subsection 3.4.1]{EES-ori}.
Then, the orientation of $T_{(u,\kappa)}\hat{\mathcal{M}}$ is defined so that the isomorphism (\ref{isom-moduli}) preserves the orientations.

When $m\geq 2$, $\hat{\mathcal{M}}=\mathcal{M}$. When $m=0,1$, we define the orientation of $\mathcal{M}=\hat{\mathcal{M}}/\mathcal{A}_{m+1}$ as follows: For any $(u,\kappa)\in \hat{\mathcal{M}}$, we take a local slice at $(u,\kappa)$ of the action of $\mathcal{A}_{m+1}$ and identify it with a neighborhood of $[(u,\kappa)]$ in $\mathcal{M}$. Then, we have an isomorphism 
\[ T_{[(u,\kappa)]} \mathcal{M} \times T_1 \mathcal{A}_{m+1} \to T_{(u,\kappa)} \hat{\mathcal{M}}\]
We use the orientation of $T_1\mathcal{A}_{m+1}$ defined in \cite[Subsection 3.4.1]{EES-ori}.
The orientation of $T_{[(u,\kappa)]} \mathcal{M}$ is defined so that this isomorphism preserves the orientations.


\subsection{Moduli space $\mathcal{M}_{\Lambda,A}^P(a;b_1,\dots ,b_m)$}\label{subsec-moduli-P}

Referring to \cite[Subsection 2.3]{EES-R}, let us introduce the moduli space
\[\mathcal{M}^P_{\Lambda,A}(a; b_1,\dots ,b_m)\]
of $J$-holomorphic curves in $P$ up to conformal equivalence.
(See also \cite[Subsection 2.3]{EES} and \cite[Subsection 2.3.4]{EENS}.)
This consists of triples $(v,f,\kappa)$ such that
$\kappa\in \mathcal{C}_{m+1}$ and $(v,f)$ is a pair of a pseudo-holomorphic map $v \colon D_{m+1}\to P$ with respect to $J$ and $j_{\kappa}$, and a smooth function $f \colon \partial D_{m+1}\to \R$ satisfying:
\begin{itemize} 
\item $(v(z), f(z))\in \Lambda$ for all $z\in \partial D_{m+1}$. In particular, $v(\partial D_{m+1})\subset \pi_P(\Lambda)$.
\item 
$\lim_{s \to \infty}v \circ \psi_0(s,t) = \pi_P(a(0))$ and $\lim_{s\to -\infty} v\circ \psi_k(s,t) = \pi_P(b_k(0))$ for $k\in \{1,\dots ,m\}$ $C^{\infty}$-uniformly on $t \in [0,1]$. Moreover,
for $t\in \{0,1\}$, $\lim_{s\to \infty }(v,f)\circ \psi_0(s,t) = a(T(1-t))$ and $\lim_{s\to -\infty} (v,f)\circ \psi_k(s,t)=b_k(T_k(1-t))$.
\item We extend $(\rest{v}{\partial D_{m+1}},f)$ continuously on $\overline{\partial D_{m+1}}$. Let $\partial (v,f)$ denote this $1$-chain in $\Lambda$. Then
\[ [\partial (v,f) +(\gamma_{b_1} + \dots +\gamma_{b_m})-\gamma_a] = A \in H_1(\Lambda) . \]
\end{itemize}
As (\ref{moduli-cobordism}), when $m=0,1$, we need to take the quotient by the action of $\mathcal{A}_{m+1}$.

From the definition of $\tilde{J}$, there exists a map
\[\Pi_P\colon \bar{\mathcal{M}}_{\Lambda, A}(a; b_1,\dots ,b_m) \to \mathcal{M}^P_{\Lambda, A}(a; b_1,\dots ,b_m) , \]
which maps $[(u,\kappa)]$ to $(v,\rest{f}{\partial D_{m+1}},\kappa)$ determined by $u(z)=(a(z),(v(z),f(z)))\in \R\times (P\times \R)$ for all $z\in D_{m+1}$.

By \cite[Lemma 4.5]{EES}, $\mathcal{M}^P_{\Lambda,A}(a;b_1,\dots ,b_m)$ is transversely cut out for generic $J$. By \cite[Proposition 2.3]{EES}, its dimension as a smooth manifold is
\[|a|-(|b_1|+\dots +|b_m|)-1,\]
and the union  $\coprod_{A\in H_1(\Lambda)} \mathcal{M}^P_{\Lambda, A}(a;b_1,\dots ,b_m)$ is a compact $0$-dimensional manifold if $|a|-1=|b_1|+\dots +|b_m|$.
Moreover, for $\tilde{J}$ of (\ref{alm-cpx-str}) associated to such $J$,
$\bar{\mathcal{M}}_{\Lambda,A}(a;b_1,\dots ,b_m)$ is also cut out transversely and $\Pi_P$ is a diffeomorphism. See \cite[Theorem 2.1, Lemma 8.2]{DR}.

\subsection{Orientation of $\mathcal{M}^P_{\Lambda,A}(a;b_1,\dots ,b_m)$}\label{subsec-compare}

We consider the case where $\Lambda_+=\Lambda_-=\Lambda$ and $L=\R\times \Lambda$. 
We fix a spin structure $\mathfrak{s}$ on $\Lambda$.
Let us use the notation $\bar{\mathcal{M}}$ for $\bar{\mathcal{M}}_{\Lambda,A}(a;b_1,\dots ,b_m)$ and $\mathcal{M}^P$ for $\mathcal{M}^P_{\Lambda,A}(a;b_1,\dots ,b_m)$.
Referring to \cite{EES-ori}, we review how to define the orientation of $\mathcal{M}^P$. We omit the case of $m=0,1$ (see \cite[Subsection 3.4.3 (a)]{EES-ori}).

Similar to $d_u\Gamma$ in the previous section, we have a linearized $\Dbar$-operator at $(v,f,\kappa)\in \mathcal{M}^P$. Its boundary condition is stabilized to define a $\Dbar$-operator $\Dbar_v$. See \cite[Subsection 3.4.2]{EES-ori}.
By a linear gluing argument, there exists an exact sequence
\begin{align}\label{ori-pi-u}\xymatrix{
0 \ar[r] & \ker \Dbar_{\hat{v}} \ar[r] &
{\begin{bmatrix}\ker \Dbar_{l,b_1,-} \\ \vdots \\  \ker \Dbar_{l,b_m,-} \\ \ker \Dbar_{l,a,+} \\ \ker \Dbar_{v} \end{bmatrix}} \ar[r]
&{\begin{bmatrix}\coker \Dbar_{l,b_1,-} \\ \vdots \\  \coker \Dbar_{l,b_m,-} \\ \coker \Dbar_{l,a,+} \\ \coker \Dbar_{v} \end{bmatrix}} \ar[r]& \coker  \Dbar_{\hat{v}}\ar[r] &0.
}\end{align}
Here, $\Dbar_{\hat{v}}$ is an $\Dbar$-operator on the disk $D$ obtained by gluing the boundary conditions for $\Dbar_v$, $\Dbar_{l,b_k,-}$ ($k=1,\dots ,m$) and $\Dbar_{l,a,+}$.
By the isomorphism (\ref{isom-exact}) induced by (\ref{ori-pi-u}),
the canonical orientation of $\det \Dbar_{\hat{v}}$ and
the orientations of the capping operators determine an orientation of $\det \Dbar_v$.

Similar to (\ref{isom-moduli}), we have an isomorphism
\begin{align}\label{isom-moduli-P}
\det \Dbar_v \to \topwedge T_{(v,f,\kappa)} \mathcal{M}^P \otimes \topwedge (T_{\kappa}\mathcal{C}_{m+1})^* .
\end{align}
The orientation of $T_{(v,f,\kappa)} \mathcal{M}^P$ is defined so that this isomorphism preserves orientations.


\begin{defi}\label{def-ori-barM}
We define the orientation of $\bar{\mathcal{M}}_{\Lambda, A}(a; b_1,\dots ,b_m)$ so that the diffeomorphism $\Pi_P\colon \bar{\mathcal{M}} \to \mathcal{M}^P$  preserves the orientations.
\end{defi}
Then, Proposition \ref{prop-DGA} which states that $\partial_{\mathfrak{s}}\circ \partial_{\mathfrak{s}}=0$ follows directly from \cite[Theorem 4.1]{EES-ori}.

On the other hand, by taking a slice at $(u,\kappa)\in \mathcal{M}$ of the action of $\R$ and identifying it locally with $\bar{\mathcal{M}}$, we have an isomorphism 
\begin{align}\label{isom-slice}
 T_0\R \times T_{[(u,\kappa)]} \bar{\mathcal{M}} \to T_{(u,\kappa)} \mathcal{M}.
 \end{align}
We need to compare the product orientation of $T_0\R \times T_{[(u,\kappa)]} \bar{\mathcal{M}}$ with the orientation of 
$T_{(u,\kappa)} \mathcal{M}$ given in Appendix \ref{subsec-ori-Lag}.
\begin{prop}\label{prop-rev-ori} The isomorphism (\ref{isom-slice}) reverses orientations.
\end{prop}
\begin{proof}
We prove the case of $m\geq 2$.
Take $(u,\kappa)\in \mathcal{M}$ and $(v,f,\kappa)= \Pi_P(u,\kappa)\in \mathcal{M}^P$.
There exist natural isomorphisms
$\ker \Dbar_u \to \R\oplus \ker \Dbar_v$ and $\coker \Dbar_u \to \coker \Dbar_v$, where $\R$ corresponds to the subspace of $\ker \Dbar_u$ spanned by the constant function to $(1,0)\in \C\oplus \C^{n+2}$.
They define an isomorphism
\begin{align}\label{isom-decomp}
\det \Dbar_u \to \R \otimes \det \Dbar_v,
\end{align}
and the following diagram commutes:
\[\xymatrix@C=30pt{
\det \Dbar_u \ar[r]^-{(\ref{isom-moduli})} \ar[d]_-{(\ref{isom-decomp})}& \topwedge T_{(u,\kappa)} \mathcal{M}\otimes \topwedge (T_{\kappa}\mathcal{C}_{m+1})^* \ar[d]^-{(\ref{isom-slice})} \\
\R\otimes \det \Dbar_v \ar[r] & \R\otimes \topwedge T_{(v,f,\kappa)} \mathcal{M}^P\otimes \topwedge (T_{\kappa}\mathcal{C}_{m+1})^*,
}\]
where the lower horizontal map is the tensor product of $\id_{\R}$ and (\ref{isom-moduli-P}).
It suffices to show that the orientation sign of (\ref{isom-decomp}) is $-1$.

There exists natural isomorphisms
$\ker \Dbar_{\hat{u}}\to \R\oplus \ker \Dbar_{\hat{v}}$ and $\coker \Dbar_{\hat{u}}\to \coker \Dbar_{\hat{v}}$, where $\R$ corresponds to the subspace in $\ker \Dbar_{\hat{u}}$ spanned by the constant function to $(1,0)\in \C\oplus \C^{n+2}$.
They define an isomorphism $\det \Dbar_{\hat{u}} \to \R\otimes \det \Dbar_{\hat{v}}$ and this preserves the canonical orientations.
Together with the isomorphisms (\ref{isom-ker-coker}), the exact sequence (\ref{ori-u}) becomes
\begin{align*}
\xymatrix{
0 \ar[r] & {\begin{bmatrix} \R \\ \ker \Dbar_{\hat{v}} \end{bmatrix}} \ar[r] &
{\begin{bmatrix}\ker \Dbar_{l,b_1,-} \\ \vdots \\  \ker \Dbar_{l,b_m,-} \\ \ker \Dbar_{l,a,+} \\ \R \\ \ker \Dbar_{v} \end{bmatrix}} \ar[r]
&{\begin{bmatrix}\coker \Dbar_{l,b_1,-} \\ \vdots \\  \coker \Dbar_{l,b_m,-} \\ \coker \Dbar_{l,a,+} \\ \coker \Dbar_{v} \end{bmatrix}} \ar[r]& \coker  \Dbar_{\hat{v}}\ar[r] &0.
}\end{align*}
We change the order of the second column by moving the vector space $\R$ to the top. This costs the sign $(-1)^m$, since $\dim \ker \Dbar_{l,a,+} \equiv 0$ and $\dim \ker \Dbar_{l,b_k,-}\equiv 1$ modulo $2$ for $k=1,\dots ,m$.
In addition, the restriction of the second map to the $\R$-components is $\R\to \R \colon v\mapsto v$. As Remark \ref{rem-cancel}, the exact sequence after changing the order induces an isomorphism which has the same sign as the one induced by (\ref{ori-pi-u}).
Since $\pi_{b_k,-}$ for $k=1,\dots ,m$ and $\pi_{b,+}$ reverses the orientations by Lemma \ref{lem-cap}, we conclude that the orientation sign of (\ref{isom-decomp}) is equal to $(-1)^m\cdot (-1)^{m+1}=(-1)$.

In the case where $m=0,1$, we reduce to the case where the automorphism group is trivial by adding even number of boundary marked points. For details, see the discussion of \cite[Subsection 4.2.3]{EES-ori}.
\end{proof}


\subsection{Proof of Theorem \ref{thm-DGA-map}}\label{subsec-proof-DGA}
We prove  Theorem \ref{thm-DGA-map} assuming Proposition \ref{prop-DGA-ori} presented in this subsection.

By \cite[Lemma B.6]{E}, 
when $|a|-\sum_{k=1}^m |b_k|=1$, the moduli space  (\ref{moduli-cobordism}) has a compactification by adding boundaries of the following two types:
\begin{enumerate}
\item The first type is the product
\begin{align}\label{boundary-1}
\bar{\mathcal{M}}_{\Lambda_+,A_+}(a; c_1,\dots ,c_l) \times \prod_{i=1}^l \mathcal{M}_{L,B_i}(c_i; b^i_1,\dots ,b^i_{m_i})
\end{align}
for $A_+\in H_1(\Lambda_+)$, $B_i\in H_1(L)$, $c_i\in \mathcal{R}(\Lambda_+)$ and $b^i_1,\dots ,b^i_{m_i}\in \mathcal{R}(\Lambda_-)$ ($i=1,\dots ,l$) such that:
\begin{itemize}
\item $(i_+)_*(A_+)+ \sum_{i=1}^l B_i=B$.
\item $(b^1_1,\dots ,b^1_{m_1},\dots , b^l_1,\dots b^l_{m_l}) =(b_1,\dots ,b_m)$.
\item $|c_i| = |b^i_1|+ \dots + |b^i_{m_i}|$ for $i=1,\dots ,l$.
\end{itemize}
\item The second type is the product
\begin{align}\label{boundary-2}
\mathcal{M}_{L,C}(a; b_1,\dots ,b_{j-1}, c, b_{k+1},\dots ,b_m)\times  \bar{\mathcal{M}}_{\Lambda_-,A_-} (c; b_j,\dots ,b_k)
\end{align}
for $C\in H_1(L)$, $A_-\in H_1(\Lambda_-)$, $j \leq k+1$ and $c\in \mathcal{R}(\Lambda_-)$ such that:
\begin{itemize}
\item $C+ (i_-)_*(A_-) =B$.
\item $|c|-1=\sum_{i=j}^k |b_i|$.
\end{itemize}
\end{enumerate}
Moreover, the compactification becomes a $1$-dimensional manifold whose boundary consists of the above two types.
In general, the boundary orientation of a boundary point of a compact oriented $1$-dimensional manifold
is given by the positive sign if the outward vector is positive, and otherwise by the negative sign.

With respect to the orientations of moduli spaces we gave in Appendix \ref{subsec-ori-Lag} and \ref{subsec-compare}, we have the following computations.
\begin{prop}\label{prop-DGA-ori}
The boundary orientations on (\ref{boundary-1}) and (\ref{boundary-2}) differ from the product orientations by the sign $(-1)^{\tau_1}$ and $(-1)^{\tau_2}$ respectively, where
\begin{align*}
\tau_1&= m  + \sum_{i=1}^l  (n-1)(|c_i|+1)  ,\\
\tau_2&=  (n-1)(|c|+1)  + (m-k+j -1) +\sum_{i=1}^{j-1} |b_i|.
\end{align*}
\end{prop}
The proof is left to Appendix \ref{subsec-prop-sign}.

\begin{proof}[Proof of Theorem \ref{thm-DGA-map} assuming Proposition \ref{prop-DGA-ori}]
Take an arbitrary Reeb chord $a\in \mathcal{R}(\Lambda_+)$.
By a direct computation,
\[\Phi_L\circ \partial_{\mathfrak{s}_+}(a) - (\partial_{\mathfrak{s}_-} \otimes \id) \circ \Phi_L(a) = \sum_{B\in H_1(L)} \sum_{|a|-1=|b_1|+\dots +|b_m|} M_B(b_1,\dots ,b_m)\cdot b_1\cdots b_m \otimes e^B.\]
Here, $M_B(b_1,\dots ,b_m)\in \Z$ is the sum of all integers of the two types
\begin{align*}\begin{array}{cc}
(-1)^{\mu_1} N_\bold{B}(\bold{c}, \bold{b}^1,\dots , \bold{b}^l) , &
(-1)^{\mu_2}N_C(c, j,k),
\end{array}
\end{align*}
where $N_{\bold{B}}(\bold{b}, \bold{c}^1,\dots , \bold{c}^l)$ (resp. $N_C(c,j,k)$) is the count of (\ref{boundary-1}) (resp. (\ref{boundary-2})) by the sign of the product orientation, and
\begin{align*}
\mu_1& \coloneqq (n-1)(|a|+1)+ \sum_{i=1}^l ((n-1)(|c_i|+1) + m_i), \\
\mu_2 &\coloneqq 1+ ( (n-1)(|a|+1) + (m-k+j)) + \sum_{i=1}^{j-1}|b_i| + (n-1)(|c|+1).
\end{align*}
Let $N'_{\bold{B}}(\bold{b}, \bold{c}^1,\dots , \bold{c}^l)$ (resp. $N'_C(c,j,k)$) denote the counts of (\ref{boundary-1}) (resp. (\ref{boundary-2})) by the sign of the boundary orientation. Then, Proposition \ref{prop-DGA-ori} shows that
\begin{align*}
&(-1)^{\mu_1} N_\bold{B}(\bold{c}, \bold{b}^1,\dots , \bold{b}^l) = (-1)^{\nu_1} N'_\bold{B}(\bold{c}, \bold{b}^1,\dots , \bold{b}^l) , \\
&(-1)^{\mu_2}N_C(c, j,k)  = (-1)^{\nu_2}N'_C(c, j,k), 
\end{align*}
hold for $\nu_i\coloneqq \mu_i + \tau_i$ ($i=1,2$).
It can be checked that $\nu_1\equiv \nu_2\equiv (n-1)(|a|+1) \mod 2$. Thus, $(-1)^{(n-1)(|a|+1)}M_B(b_1,\dots ,b_m)$ is equal to
the number of all boundary points of the two types (\ref{boundary-1}) and (\ref{boundary-2}) counted by the sign of the boundary orientation, and clearly it is equal to $0$.
\end{proof}

\subsection{Proof of Proposition \ref{prop-DGA-ori}}\label{subsec-prop-sign}
Recall the setup for Proposition \ref{prop-DGA-ori}.
Suppose that $a\in \mathcal{R}(\Lambda_+)$ and $b_1,\dots ,b_m\in \mathcal{R}(\Lambda_-)$ satisfy $|a|-\sum_{k=1}^m|b_k|=1$. 
For short, we denote $\mathcal{M}_{L,B}(a;b_1,\dots ,b_m)$ by $\mathcal{M}^1$ since its dimension is $1$.
Then, $\mathcal{M}^1$ has a compactification by adding 0-dimensional moduli spaces of pseudo-holomorphic buildings of the two types (\ref{boundary-1}) and (\ref{boundary-2}).
We compute the signs $(-1)^{\tau_1}$ for (\ref{boundary-1}) and $(-1)^{\tau_2}$ for (\ref{boundary-2}) which are the differences of the two orientations on the boundaries: the product orientation and the boundary orientation. 

\begin{proof}[Proof of Proposition \ref{prop-DGA-ori}]
From now on, we omit writing conformal structures for elements of moduli spaces.
We only consider the case where all $\tilde{J}$-holomorphic curves have at least two negative punctures. Otherwise, we reduce to the case where the automorphism group is trivial by adding even number of boundary marked points. See the discussion in \cite[Subsection 4.2.3]{EES-ori}.

\textit{The case of} (\ref{boundary-1}).
For short, we denote $\mathcal{M}_{\Lambda_+,A_+}(a;c_1,\dots ,c_l)$ by $\mathcal{M}_+$, $\bar{\mathcal{M}}_{\Lambda_+,A_+}(a;c_1,\dots ,c_l)$ by $\bar{\mathcal{M}}_+$ and $\mathcal{M}_{L,B_i}(c_i;b^i_1,\dots ,b^i_{m_i})$ by $\mathcal{M}_i$ for $i=1,\dots ,l$. Let $T_i\coloneqq \int c_i^*(dz-\lambda)$ denote the action of the Reeb chord $c_i$.
Let $w\in \mathcal{M}^1$ be a $\tilde{J}$-holomorphic curve obtained by gluing $[u]\in \bar{\mathcal{M}}_+$ and $v_i\in \mathcal{M}_i$ for $i=1,\dots ,l$.
By a linear gluing argument, there exists an exact sequence
\begin{align}\label{isom-glue-10}
\xymatrix{
0\ar[r] & \ker \Dbar_w \ar[r] & {\begin{bmatrix} \bigoplus_{i=1}^l\ker \Dbar_{v_i} \\\ker \Dbar_u \end{bmatrix}} \ar[r] & {\begin{bmatrix} \bigoplus_{i=1}^l {\begin{bmatrix}\coker \Dbar_{v_i} \\ \R_i \end{bmatrix}} \\ \coker \Dbar_u \end{bmatrix}} \ar[r] & \coker \Dbar_w \ar[r] & 0,
}\end{align}
where $\R_i\coloneqq \R$ for $i=1,\dots ,l$.
Its construction is similar to \cite[Remark 3.3]{EES-ori} and here let us identify any $r \in \R_i$ with a $\C\oplus \C^{n+2}$-valued $(0,1)$-form $r\cdot \zeta_i$ on the $i$-th gluing strip, where
\begin{align}\label{anti-1-form}
(\zeta_i)_{(s,t)} \colon \partial_s\mapsto (-1,0),\ \partial_t\mapsto (\sqrt{-1},0)
\end{align}
for any $(s,t)\in \R\times [0,T_i]$.
Note that $\ker \Dbar_{v_i}=0$ for $i=1,\dots ,l$ since $\mathcal{M}_i$ is cut out transversely and has dimension $0$. 
We have an isomorphism (\ref{isom-exact}) induced by the exact sequence (\ref{isom-glue-10})
\begin{align}\label{isom-sigma1}
\begin{split}
\det \Dbar_w  \to  \topwedge \ker\Dbar_u \otimes \left( \bigotimes_{i=1}^l \left( \topwedge  (\coker \Dbar_{v_i})^* \otimes \R_i^* \right) \right) \otimes \topwedge (\coker \Dbar_u)^*  .
\end{split}
\end{align}
The orientations of $\det\Dbar_w$, $\det \Dbar_u =\topwedge \ker \Dbar_u\otimes \topwedge (\coker \Dbar_u)^*$ and $\det \Dbar_{v_i}= \topwedge (\coker \Dbar_{v_i})^*$ are determined as explained in Appendix \ref{subsec-ori-Lag}.
Let $(-1)^{\sigma_1}$ be the sign of the isomorphism (\ref{isom-sigma1}). Then
\[\sigma_1 \equiv l + \sum_{i=1}^l ( i \cdot ( m_i+1) + (n-1)(|c_i|+1))  \mod 2 .\]
This is computed as in \cite[(41)]{K} by comparing the two exact sequences below which come from gluing the capping operators: 
\begin{align}\label{seq-1-1}
\xymatrix@C=20pt@R=10pt{
 \ker \Dbar_T \ar[r] & {\begin{bmatrix} \bigoplus_{i=1}^l { \begin{bmatrix} \bigoplus_{\alpha=1}^{m_i} \ker \Dbar_{b^i_{\alpha},-} \\ \ker \Dbar_{c_i,+} \\ \ker \Dbar_{v_i} \end{bmatrix}} \\  { \begin{bmatrix}  \bigoplus_{i=1}^l\ker \Dbar_{c_i,-}  \\ \ker \Dbar_{a,+} \\ \ker \Dbar_{u} \end{bmatrix}} \end{bmatrix}} \ar[r] & 
 {\begin{bmatrix} \bigoplus_{i=1}^l {\begin{bmatrix}  \bigoplus_{\alpha=1}^{m_i}\coker \Dbar_{b^i_{\alpha},-} \\ \coker \Dbar_{c_i,+} \\ \coker \Dbar_{v_i} \\ \R_i \\ \R^{n+2} \end{bmatrix}} \\  { \begin{bmatrix} \bigoplus_{i=1}^l \coker \Dbar_{c_i,-} \\ \coker \Dbar_{a,+} \\ \coker \Dbar_{u} \end{bmatrix}} \end{bmatrix}}
  \ar[r] & \coker \Dbar_T, }
\end{align}
\begin{align}\label{seq-1-2}
\xymatrix@C=20pt@R=10pt{
 \ker \Dbar_T \ar[r] & {\begin{bmatrix} \bigoplus_{i=1}^l {\begin{bmatrix} \ker \Dbar_{c_i,+} \\ \R'_i \\ \ker \Dbar_{c_i,-} \end{bmatrix}} \\ \bigoplus_{i=1}^l \left( \bigoplus_{\alpha=1}^{m_i} \ker \Dbar_{b^i_{\alpha},-} \right) \\ \ker \Dbar_{a,+} \\ \bigoplus_{i=1}^l \ \ker \Dbar_{v_i} \\ \ker\Dbar_u \end{bmatrix}} \ar[r] & 
 {\begin{bmatrix} \bigoplus_{i=1}^l {\begin{bmatrix} \coker \Dbar_{c_i,+} \\ \coker \Dbar_{c_i,-}  \\ \R'_i \\ \R^{n+2} \end{bmatrix} } \\ \bigoplus_{i=1}^l \left( \bigoplus_{\alpha=1}^{m_i} \coker \Dbar_{b^i_{\alpha},-} \right) \\ \coker \Dbar_{a,+} \\\bigoplus_{i=1}^l {\begin{bmatrix} \coker \Dbar_{v_i} \\ \R_i \end{bmatrix}} \\ \coker\Dbar_u \end{bmatrix}} 
  \ar[r] & \coker \Dbar_T.
}
\end{align}
Here, the first and the last zero maps are omitted and $\Dbar_T$ is a $\Dbar$-operator on the disk $D$ defined in the proof of \cite[Theorem 2.5]{K}. The capping operators $\Dbar_{s,c',\pm}$ for any $c'\in \mathcal{R}(\Lambda)$ are written by $\Dbar_{c',\pm}$.
For every $i\in \{1,\dots ,l\}$, $\R_i \coloneqq \R \eqqcolon \R'_i$. In (\ref{seq-1-2}), $\R'_i$ in the second column is mapped by the identity to $\R'_i$ in the third column.
$\sigma_1$ can be calculated by the following steps:
\begin{itemize}
\item We arrange the exact sequence (\ref{seq-1-2}). In both the second and the third columns, we first move $\R'_1$ to the top and then,
by induction on $i=2,\dots ,l$, we move $\R'_i$ to the place just below $\R'_{i-1}$. This costs the sign $(-1)^{\sigma_{1,1}}$, where $\sigma_{1,1} = (n-1)l$.
Moreover, we switch the places of $\ker \Dbar_{c_i,+}$ and $\ker \Dbar_{c_i,-}$ in the second column and also switch the places of $\coker \Dbar_{c_i,+}$ and $\coker \Dbar_{c_i,-}$ in the third column. This costs the sign $(-1)^{\sigma_{1,2}}$, where
$\sigma_{1,2} = n\cdot \sum_{i=1}^l |c_i|$.
Then, for $\R'^l \coloneqq \bigoplus_{i=1}^l \R'_i$, (\ref{seq-1-2}) is changed to the following form:
\begin{align}\label{seq-1-3}
 \xymatrix@C=20pt@R=10pt{ \ker \Dbar_T \ar[r] & {\begin{bmatrix} \R'^l  \\ \bigoplus_{i=1}^l {\begin{bmatrix} \ker \Dbar_{c_i,-}  \\ \ker \Dbar_{c_i,+} \end{bmatrix}} \\ \bigoplus_{i=1}^l \left( \bigoplus_{\alpha=1}^{m_i} \ker \Dbar_{b^i_{\alpha},-} \right) \\ \ker \Dbar_{a,+} \\ \bigoplus_{i=1}^l \ \ker \Dbar_{v_i} \\ \ker\Dbar_u \end{bmatrix}} \ar[r] & 
 {\begin{bmatrix} \R'^l \\ \bigoplus_{i=1}^l {\begin{bmatrix} \coker \Dbar_{c_i,-} \\ \coker \Dbar_{c_i,+}  \\ \R^{n+2} \end{bmatrix} } \\ \bigoplus_{i=1}^l \left( \bigoplus_{\alpha=1}^{m_i} \coker \Dbar_{b^i_{\alpha},-} \right) \\ \coker \Dbar_{a,+} \\\bigoplus_{i=1}^l {\begin{bmatrix} \coker \Dbar_{v_i} \\ \R_i \end{bmatrix}} \\ \coker\Dbar_u \end{bmatrix}} 
  \ar[r] & \coker \Dbar_T .}
  \end{align}
By the argument in Remark \ref{rem-cancel}, we may remove $\R'^l$ from the second and the third columns.
\item Next, we arrange the exact sequence (\ref{seq-1-1}). We switch the places of $\bigoplus_{\alpha=1}^{m_i} \ker \Dbar_{b^i_{\alpha},-}$ and $\ker \Dbar_{c_i,-}$ in the second column for every $i=1,\dots ,l$. This costs the sign $(-1)^{\sigma_{1,3}}$, where $\sigma_{1,3}= l\cdot \sum_{i=1}^l m_i$. 
We also switch the places of $\bigoplus_{\alpha=1}^{m_i} \coker \Dbar_{b^i_{\alpha},-}$ and $\coker \Dbar_{c_i,-}$ in the third column for every $i=1,\dots ,l$.
This costs the sign $(-1)^{\sigma_{1,4}}$, where $\sigma_{1,4}= \sum_{i=1}^l |c_i|$. 

\item After the above arrangement of (\ref{seq-1-1}), we move $\ker \Dbar_{v_1}$ in the second column to the place just below $\ker \Dbar_{a,+}$ and then, by induction on $i=2,\dots ,l$, we move $\ker \Dbar_{v_i}$ to the place just below $\ker \Dbar_{v_{i-1}}$. This does not cost the sign since $\dim \ker \Dbar_{v_i}=0$ for every $i=1,\dots ,l$. Moreover, we move $\begin{bmatrix}\coker \Dbar_{v_1} \\ \R_1 \end{bmatrix}$ in the third column to the place just below $\coker \Dbar_{a,+}$ and then, by induction on $i=2,\dots ,l$, we move $\begin{bmatrix}\coker \Dbar_{v_i} \\ \R_i \end{bmatrix}$ just below $\begin{bmatrix}\coker \Dbar_{v_{i-1}} \\ \R_{i-1} \end{bmatrix}$.
This costs the sign $(-1)^{\sigma_{1,5}}$, where
$\sigma_{1,5} = \sum_{i=1}^l(m_i+1)(l-i)$.
Here, note that $\dim \begin{bmatrix}\coker \Dbar_{v_i} \\ \R_i \end{bmatrix} = m_i-1$ for every $i=1,\dots ,l$.
After these arrangements, the exact sequence agrees with (\ref{seq-1-3}) with $\R'^l$ removed.
\item Finally, $\sigma_1$ is given by
\[\sigma_{1,1}+\sigma_{1,2}+\sigma_{1,3}+\sigma_{1,4} + \sigma_{1,5} \equiv  l + \sum_{i=1}^l ( i \cdot ( m_i+1) + (n-1)(|c_i|+1)) \mod 2.\]
\end{itemize}

%

We want to relate (\ref{isom-glue-10}) to the exact sequence (\ref{seq-Gamma})  for $w \in \mathcal{M}^1$.
Let us consider the following isomorphisms:
\[\xymatrix@R=10pt{
 {\begin{bmatrix} \bigoplus_{i=1}^l\ker \Dbar_{v_i} \\\ker \Dbar_u \end{bmatrix}} \ar[r]^-{\gamma_1} &   {\begin{bmatrix} \bigoplus_{i=1}^l T_v\mathcal{M}_i \\ T_u\mathcal{M}_+ \end{bmatrix}} & {\begin{bmatrix} \bigoplus_{i=1}^l T_v\mathcal{M}_i \\ \R \oplus T_{[u]}\bar{\mathcal{M}}_+ \end{bmatrix}}  \ar[l]_-{(\ref{isom-slice})} \ar[r]^-{\gamma_3} & T_w \mathcal{M}^1 , \\
{\begin{bmatrix} \bigoplus_{i=1}^l {\begin{bmatrix}\coker \Dbar_{v_i} \\ \R_i \end{bmatrix}} \\ \coker \Dbar_u \end{bmatrix}} & {\begin{bmatrix} \bigoplus_{i=1}^l {\begin{bmatrix} T\mathcal{C}_{m_i+1} \\ \R_i \end{bmatrix}} \\ T\mathcal{C}_{l+1} \end{bmatrix}}  \ar[l]_-{\gamma_2} \ar[r]^-{\gamma_4} & T\mathcal{C}_{m+1}.
}
\]
These maps are defined as follow:
\begin{itemize}
\item $\gamma_1$ and $\gamma_2$ are combinations of $\id_{\R_i}\colon \R_i \to \R_i$ and 
those maps from (\ref{seq-Gamma}) for $u\in \mathcal{M}_+$ and $v_i\in \mathcal{M}_i$ ($i=1,\dots ,l$). 
To see that they are isomorphisms, note that $\bar{\mathcal{M}}_+$ and $\mathcal{M}_i$ are $0$-dimensional manifolds.
We recall that the isomorphism (\ref{isom-exact}) induced by (\ref{seq-Gamma}) preserves the orientations and (\ref{isom-slice}) changes the orientations by the sign $(-1)^1$ by Proposition \ref{prop-rev-ori}.
\item $\gamma_3$ is the linearization of a map defined by gluing $[u]$ and $(v_1,\dots ,v_l)$. (Here, $1\in \R$ corresponds to $\frac{d}{d\rho}\in T_{\rho}\R$, where $\rho \gg 0$ is a gluing parameter.)
It has the sign $(-1)^{\tau_1}$.
Indeed, if $w_{\rho}\in \mathcal{M}^1$ is obtained from a pre-gluing of $(v_1,\dots ,v_l)$ and $\tau_{\rho}\circ u$ for $\rho \gg 0$, then $w_{\rho}$ converges to the boundary point as $\rho \to \infty$. Therefore, a positive vector in $\R$ is mapped to an outward vector in $T_{w}\mathcal{M}^1$.
\item $\gamma_4$ is the linearization of a map defined by gluing punctured disks with conformal structures. See \cite[Subsection 4.2.2]{EES-ori}. It has the sign $(-1)^{\sigma_2}$, where
\[\sigma_2 \equiv \sum_{i=1}^l((m_i-1)\cdot i +1) + \sum_{i=1}^l (m_i-2) \mod 2.\]
Indeed, $(-1)^{\sum_{i=1}^l((m_i-1)\cdot i +1)}$ is the sign of the isomorphism ${\begin{bmatrix} \bigoplus_{i=1}^l {\begin{bmatrix} \R_i \\ T\mathcal{C}_{m_i+1}  \end{bmatrix}} \\ T\mathcal{C}_{l+1} \end{bmatrix}} \to T\mathcal{C}_{m+1}$, which follows the convention of gluing in \cite[Subsection 4.2.2]{EES-ori}, and can be computed by applying \cite[Lemma 4.7]{EES-ori} inductively on $j=l,\dots ,1$. $(-1)^{\sum_{i=1}^l (m_i-2)}$ is the cost of sign for switching the places of $\R_i$ and $T\mathcal{C}_{m_i+1}$ for every $i=1,\dots ,l$.
\end{itemize}
Then, $\gamma_1,\dots ,\gamma_4$ and (\ref{isom-slice}) induce an isomorphism
\begin{align}\label{isom-gamma1-4}
\begin{split}
& \topwedge \ker \Dbar_u \otimes \left( \bigotimes_{i=1}^l  \left( \topwedge ( \coker \Dbar_{v_i})^* \otimes \R_i^* \right) \right) \otimes \topwedge (\coker \Dbar_u)^* \\
\to \ & \topwedge T_w\mathcal{M}^1 \otimes \topwedge (T\mathcal{C}_{m+1})^*
 \end{split}
\end{align}
and the sign of this isomorphism is $(-1)^{1 +\tau_1 +\sigma_2}$.
The composite map
\begin{align*}
(\ref{isom-gamma1-4}) \circ (\ref{isom-sigma1}) \colon \det \Dbar_w \to \topwedge T_w\mathcal{M}^1 \otimes \topwedge(T\mathcal{C}_{m+1})^*
\end{align*}
has the sign $(-1)^{\sigma_1 + (1+\tau_1+\sigma_2)}$.
\begin{lem}\label{lem-sigma-tau-1}
If $w\in \mathcal{M}^1$ is sufficiently close to the boundary point $([u],(v_1,\dots ,v_l))$,
then $(\ref{isom-gamma1-4}) \circ (\ref{isom-sigma1})$ coincides with the isomorphism (\ref{isom-exact}) induced by the exact sequence (\ref{seq-Gamma}) for $w$ up to multiplication by a negative scalar. In particular, $(\ref{isom-gamma1-4}) \circ (\ref{isom-sigma1})$ reverses the orientations.
\end{lem}
\begin{proof}
We prepare several notations.
We define
\[\begin{array}{rl}
  D_{m_i+1}(\rho_i) & \coloneqq D_{m_i+1}\setminus \psi_0((T_i^{-1}\rho_i ,\infty)\times [0,1]), \\
D_{l+1}(\boldsymbol{\rho}) & \coloneqq D_{l+1}\setminus \bigcup_{i=1}^l \psi_i((-\infty, -T_i^{-1}\rho_i )\times [0,1])  \end{array}\]
for $\boldsymbol{\rho}=(\rho_1,\dots ,\rho_l) \in (\R_{>0})^l$.
By identifying $\psi_0(T_i^{-1}\rho_i , t)\in D_{m_i+1}(\rho_i)$ with $\psi_i(-T_i^{-1}\rho_i  ,t)\in D_{l+1}(\boldsymbol{\rho})$ for every $t\in [0,1]$ and $i\in \{1,\dots ,l\}$, we obtain a disk $D_{m+1}(\boldsymbol{\rho})$ with $m+1$ boundary punctures.
Let $Q_{i,\rho_i}$ denote the image of the embedding
\[   [-\rho_i ,\rho_i] \times [0,T_i] \to D_{m+1}(\boldsymbol{\rho}) \colon (s,t) \mapsto \begin{cases} \psi_i(T_i^{-1} (s-\rho_i)  ,T_i^{-1}t) \in D_{l+1}(\boldsymbol{\rho}) & \text{ if }s\in [0,\rho_i] , \\ \psi_0(T_i^{-1} (s+\rho_i) , T_i^{-1}t) \in D_{m_i+1}(\rho_i) & \text{ if }s\in [-\rho_i,0]. \end{cases}\]

Recall that $\mathcal{M}^1$ is the zero set of a section $\Gamma$ of a Banach bundle over a Banach manifold.
Let $\tilde{w}_{\rho}$ be a pre-gluing of $(v_1,\dots ,v_l)$ and $\tau_{\rho}\circ u$ defined on $D_{m+1}(\rho,\dots ,\rho)$.
For its construction, we refer to \cite[Section 10.4]{CELN} in a setting similar to the present case.
$\tilde{w}_{\rho}$ is an element of the Banach manifold and if $\rho>0$ is sufficiently large, there exists an open ball $B_{\rho}$ centered at $\tilde{w}_{\rho}$ 
in which we can find $w_{\rho} \in B_{\rho} \cap \mathcal{M}^1$ by applying Floer's Picard Lemma \cite[Lemma 10.10]{CELN}
(here we omit writing conformal structures).
For more details, see \cite[Section 10.4]{CELN}.
$(w_{\rho})_{\rho \gg 0}$ converges to the boundary point $([u],(v_1,\dots ,v_l))$ as $\rho \to \infty$.

We rewrite $\gamma_3$ and $\gamma_4$ for $w=w_{\rho}$ by $\gamma_{3,\rho}$ and $\gamma_{4,\rho}$. The isomorphisms $\gamma_1,\gamma_2, \gamma_{3,\rho},\gamma_{4,\rho}$ and the exact sequences (\ref{isom-glue-10}) and (\ref{seq-Gamma}) for $w_{\rho}\in \mathcal{M}^1$ are combined to the following diagram:
\begin{align*}
\xymatrix@C=20pt@R=15pt{
0 \ar[r] & \ker \Dbar_{w_{\rho}} \ar[r] & \ker \Dbar_u \ar[d]_-{\gamma_1} \ar[r]^-{a_{\rho}} & {\begin{bmatrix} \bigoplus_{i=1}^l (\coker \Dbar_{v_i}\oplus \R_i) \\ \coker \Dbar_u \end{bmatrix}}  \ar[r]^-{b_{\rho}} &  \coker \Dbar_{w_{\rho}} \ar[r] & 0  \colon (\ref{isom-glue-10}) \\
& & \R \ar[d]_-{\gamma_{3,\rho}} & {\begin{bmatrix} \bigoplus_{i=1}^l ( T\mathcal{C}_{m_i+1} \oplus \R_i) \\ T\mathcal{C}_{l+1} \end{bmatrix}} \ar[d]_-{\gamma_{4,\rho}}  \ar[u]^-{\gamma_2} & & \\
0 \ar[r] & \ker \Dbar_{w_{\rho}} \ar@{=}[uu] \ar[r] & T_{w_{\rho}} \mathcal{M}^1 \ar[r]^-{c_{\rho}} & T \mathcal{C}_{m+1} \ar[r]^-{d_{\rho}} & \coker \Dbar_{w_{\rho}} \ar[r] \ar@{=}[uu] & 0 \colon  (\ref{seq-Gamma}).
}
\end{align*}
Here, for $\gamma_1$ and $\gamma_{3,\rho}$, we omit writing $T_{[u]} \bar{\mathcal{M}}_+$, $\ker \Dbar_{v_i}$ and $T_{v_i}\mathcal{M}_i$ for $i=1,\dots ,l$  since they are $0$-dimensional vector spaces. We also identify $\R$ with $T_u \mathcal{M}_+$ via (\ref{isom-slice}). 
Let us describe those maps in the diagram concretely:
\begin{itemize}
\item  $\gamma_1^{-1}(1)\in \ker \Dbar_u$ is the constant function to $(1,0)\in \C\oplus \C^{n+2}$. Let $f_0\colon D_{m+1}\to \C\oplus \C^{n+2}$ denote this constant function.
\item $c_{\rho}$ agrees with the differential of the map $\mathcal{M}^1\to \mathcal{C}_{m+1}\colon (w,\kappa)\mapsto \kappa$.
\item
$d_{\rho}(V) \in \coker \Dbar_{w_{\rho}}$ for $V\in T\mathcal{C}_{m+1}$ can be described as follows: 
Choose a variation $(\kappa(r))_{r\in [-r_0,r_0]}$ of conformal structures on $D_{m+1}$ such that $V= \rest{\frac{d}{dr}}{r=0}\kappa(r) \in T\mathcal{C}_{m+1}$. Then, $d_{\rho}(V)$ is the projection to $\coker \Dbar_{w_{\rho}}$ of
\[ \textstyle{ \rest{\frac{d}{dr}}{r=0} (dw_{\rho} + \tilde{J} \circ d w_{\rho} \circ j_{\kappa(r)}) = \rest{\frac{d}{dr}}{r=0}  \tilde{J}\circ d w_{\rho} \circ j_{\kappa(r)} = \rest{\frac{d}{dr}}{r=0}  d w_{\rho} \circ j_{\kappa(0)} \circ j_{\kappa(r)} . }\]
Likewise, $\gamma_2(V_i) \in \coker \Dbar_{v_i}$  for $V_i= \rest{\frac{d}{dr}}{r=0}\kappa_i (r)\in T\mathcal{C}_{m_i+1}$ is the projection to $\coker \Dbar_{v_i}$ of
\[ \textstyle{\rest{\frac{d}{dr}}{r=0}  dv_i\circ j_{\kappa(0)} \circ j_{\kappa_i(r)} . }\]
In the same way, $\gamma_2(V') \in \coker \Dbar_u$ for $V' \in T\mathcal{C}_{l+1}$ can be described.
\item
Those maps in (\ref{isom-glue-10}) are defined in a similar way as in \cite[Remark 3.3]{EES-ori}. In particular, $b_{\rho}\circ \gamma_2(e_i)=b_{\rho}(e_i)$ for $e_i \coloneqq 1 \in \R_i$ is described as follows:
We choose a $C^{\infty}$ function $\mu_{\rho} \colon \R \to [0,1]$ such that $\mu_{\rho} (s) = \begin{cases} 1 & \text{ if }s\in [-\frac{3}{4}\rho,\frac{3}{4}\rho] , \\ 0 & \text{ if }s \notin [-\frac{7}{8}\rho, \frac{7}{8}\rho], \end{cases}$ and 
 $|d^k \mu_{\rho}| = \mathcal{O}(\rho^{-1})$ for $k=1,2$. 
Referring to the proof of \cite[Lemma 3.1]{EES-ori}, 
$b_{\rho}(e_i)$ is defined to be the projection to $\coker \Dbar_{w_{\rho}}$ of a $(0,1)$-form $\zeta_{i,\rho}$ supported in $Q_{i,\rho}$ and defined by
\[ (\zeta_{i,\rho})_{(s,t)} \coloneqq  \mu_{\rho}(s)\cdot (\zeta_i)_{(s,t)} .\]
for every $(s,t)\in [-\rho,\rho]\times [0,T_i] \cong Q_{i,\rho}$. Here, $\zeta_i$ is the $(0,1)$-form defined by (\ref{anti-1-form}).
\end{itemize}
As in the proof of  \cite[Lemma 3.1]{EES-ori}, we use the weighted Sobolev norm $||\cdot ||_{k,\rho}$ ($k=1,2$) for functions (or $(0,1)$-forms) on $D_{m+1}(\rho,\dots ,\rho)$ with a weight function $\lambda_{\rho}\colon D_{m+1}(\rho,\dots ,\rho) \to \R_{>0}$ analogous to the one in \cite[Section 10.4]{CELN}.
In addition, we identify $\coker \Dbar_{w_{\rho}}$ with the $L^2$-orthogonal complement of $\Im \Dbar_{w_{\rho}}$.
The radius with respect to $||\cdot ||_{2,\rho}$ of the ball $B_{\rho}$ centered at the pre-gluing $\tilde{w}_{\rho}$ converges to $0$ as $\rho \to \infty$.
Lastly, for any compact domain $D'\subset D_{m+1}(\rho,\dots,\rho)$, let us introduce a notation: By using $\rest{\lambda_{\rho}}{D'}\colon D'\to \R_{>0}$ as a weight function, similar to the norm $|| \cdot ||_{k,\rho}$, we define $||\cdot ||_{k,\rho,D'}$ to be the Sobolev norm for functions (or $(0,1)$-forms) on $D'$ whose weight function is $\rest{\lambda_{\rho}}{D'}$.

By using these descriptions and notations, we can prove the following three results about the diagram when $\rho$ is sufficiently large:
\begin{enumerate}
\item We choose a $C^{\infty}$ function $\mu'_{\rho}\colon D_{m+1} (\rho,\dots ,\rho )\to [0,1]$ which is constant to $1$ on $D_{l+1}(\frac{\rho}{2},\dots ,\frac{\rho}{2})$ and constant to $0$ on $\coprod_{i=1}^l D_{m_i+1}(\rho)$, and satisfies $|d^k \mu'_{\rho}| = \mathcal{O}(\rho^{-1})$ for $k=1,2$. In addition,  for every $i\in \{1,\dots ,l\}$, we assume that $\mu'_{\rho}(s,t)$ is increasing in $s$ and independent of $t$ for $(s,t) \in [-\rho,\rho]\times [0,T_i]\cong Q_{i,\rho}$. 
Then, $\rest{\mu'_{\rho} \cdot f_0}{D_{l+1}(\rho,\dots ,\rho)}$ is extended to a function on $D_{m+1}(\rho,\dots ,\rho)$ so that it is constant to $0$ on $\coprod_{i=1}^l D_{m_i+1}(\rho)$. This satisfies
\[ \textstyle{ \Dbar (\mu'_{\rho} \cdot f_0) = - \frac{\partial \mu'_{\rho}}{\partial s} \zeta_i } \]
on $Q_{i,\rho}$. $\zeta_i$ defined by (\ref{anti-1-form}) corresponds to $1\in \R_i$ and the $L^2$-inner product $\la - \frac{\partial \mu'_{\rho}}{\partial s} \zeta_i ,\zeta_i \ra_{L^2}$ is negative.
This shows that the $\R_i$-component of $a_{\rho}(f_0)$ is negative for every $i\in \{1,\dots ,l\}$.
On the other hand, $\gamma_{4,\rho}^{-1}\circ c_{\rho} \circ \gamma_{3,\rho}(1)$ converges to $ ((0,1)_{i=1,\dots ,l} , 0 )\in  {\begin{bmatrix} \bigoplus_{i=1}^l ( T\mathcal{C}_{m_i+1} \oplus \R_i) \\ T\mathcal{C}_{l+1} \end{bmatrix}} $ as $\rho \to \infty$.
Therefore, the following hold when $\rho$ is sufficiently large:
First, since $a_{\rho}$ is not the zero map and $\dim \ker \Dbar_u=1$, it follows that $\ker \Dbar_{w_{\rho}} =0$.
Second, the $\R_1$-component of $\gamma_2^{-1}\circ a_{\rho}\circ \gamma_1^{-1}(1)$ is negative, while the $\R_1$-component of $\gamma_{4,\rho}^{-1} \circ c_{\rho}\circ \gamma_{3,\rho}(1)$ is positive.
\item Fix any $i\in \{1,\dots, l\}$ and consider any $V_i\in T\mathcal{C}_{m_i+1}$. Then, there exists a compact domain $D'$ in $D_{m_i+1}$ and a variation $(\kappa_i(r))_{r\in [-r_0,r_0]}$ of conformal structures such that $V_i=\rest{\frac{d}{dr}}{r=0} \kappa_i(r)$ and the restriction $\rest{\kappa_i(r)}{D_{m_i+1}\setminus D'}$ is independent of $r$.
When $\rho$ is sufficiently large, $D'$ is contained in $D_{m_i+1}(\rho)$ and we obtain a variation $(\kappa(r)_{\rho})_{r\in [-r_0,r_0]}$ of conformal structures on $D_{m+1}(\rho,\dots ,\rho)$ by gluing $\kappa_i(r)$ with a conformal structure on $D_{l+1}(\rho ,\dots ,\rho) \sqcup \coprod_{i'\neq i}D_{m_{i'}+1}(\rho)$ independent of $r$. Then, $\gamma_{4,\rho}(V_i) = \rest{\frac{d}{dr}}{r=0} \kappa(r)_{\rho}$.
Since $\rest{w_{\rho}}{D'}$ converges to $\rest{v_i}{D'}$ with respect to the Sobolev norm $||\cdot ||_{2,\rho,D'}$ as $\rho \to \infty$, 
\[ \textstyle{ \lim_{\rho \to \infty} || \rest{\frac{d}{dr}}{r=0}  \rest{ dw_{\rho}}{D'} \circ j_{\kappa(0)_{\rho}}\circ j_{\kappa(r)_{\rho}} - \rest{\frac{d}{dr}}{r=0}   \rest{dv_i}{D'} \circ j_{\kappa_i(0)}\circ j_{\kappa_i(r)} ||_{1,\rho,D'} =0 . } \]
In addition, we have the following fact about the map $b_{\rho}$ which can be checked, referring to \cite[Remark 3.3]{EES-ori} and the proof of \cite[Lemma 3.1]{EES-ori}, from the definitions of $b_{\rho}$ and the projections to $\coker \Dbar_{w_{\rho}}$ and $\coker \Dbar_{v_i}$:
For any element 
\[g\in H_{1,\epsilon}[0](D_{m+1}(\rho,\dots ,\rho);T^{*0,1}D_{m+1}(\rho,\dots ,\rho)\otimes \C^{n+3})\]
supported in $D'$, which can also be viewed as an element of $H_{1,\epsilon}[0](D_{m_i+1};T^{*0,1}D_{m_i+1}\otimes \C^{n+3})$, the difference between the projection of $g$ to $\coker \Dbar_{w_{\rho}}$ and  the image for $b_{\rho}$ of the projection of $g$ to $\coker \Dbar_{v_i}$ converges to $0$ as $\rho \to \infty$.
These computations when $\rho \to \infty$ show that
\begin{align}\label{rho-2}
\textstyle{ \lim_{\rho \to \infty} || d_{\rho}\circ \gamma_{4,\rho}(V_i) - b_{\rho} \circ \gamma_2(V_i) ||_{1,\rho} =0  .}
 \end{align}
In the same way, we can show that
\begin{align}\label{rho-2'}
\textstyle{\lim_{\rho \to \infty} || d_{\rho}\circ \gamma_{4,\rho}(V') - b_{\rho} \circ \gamma_2(V') ||_{1,\rho} =0}
\end{align}
for any $V'\in T\mathcal{C}_{l+1}$.
\item 
Let $I_{\rho}\coloneqq \frac{1}{2}\int_{-\infty}^{\infty} \mu_{\rho}(s)ds$. Then, $d_{\rho}\circ \gamma_{4,\rho}(I_{\rho}e_i)$ for $e_i = 1\in \R_i$ is computed as follows:
Let $j_r$ be the complex structure on $Q_{i,\rho +rI_{\rho}} \cong [-\rho-rI_{\rho}, \rho + r I_{\rho}]\times [0,T_i]$ determined by $j_r(\partial_s)=\partial_t$. For any $r\in [-r_0,r_0]$, where $r_0>0$ is sufficiently small, let $\varphi_r$ be a diffeomorphism
\[\textstyle{ \varphi_r\colon Q_{i,\rho} \to Q_{i,\rho +rI_{\rho}}  \colon (s,t) \mapsto (s + r \cdot (\int_{-\infty}^s\mu_{\rho}(s)ds -I_{\rho}) ,t).
}\]
(It is naturally extended to a diffeomorphism from $D_{m+1}(\rho,\dots ,\rho)$ to $D_{m+1}(\rho_1,\dots,\rho_l)$, where $\rho_i=\rho+r I_{\rho}$ and $\rho_{i'}=\rho$ if $i'\neq i$.)
We obtain a variation $(\varphi_r^*j_r)_{r\in [-r_0,r_0]}$ of complex structures on $Q_{i,\rho}$.
Then, $d\circ \gamma_4(I_{\rho}e_i)$ is the projection to $\coker \Dbar_{w_{\rho}}$ of a $(0,1)$-form which is supported in $Q_{i,\rho}$ and defined on $Q_{i,\rho}$ by
\[ \textstyle{ \zeta'_{i,\rho}  \coloneqq \rest{\frac{d}{dr}}{r=0} d w_{\rho} \circ j_0 \circ (\varphi_r^* j_r) . } \]
We can explicitly compute that for every $(s,t) \in  [-\rho,\rho]\times [0,T_i] \cong Q_{i,\rho}$,
\[\textstyle{ (\zeta'_{i,\rho})_{(s,t)} \colon \partial_s \mapsto -\mu_{\rho}(s)\frac{\partial w_{\rho}}{\partial s}(s,t) ,\ \partial_t \mapsto \mu_{\rho}(s) \frac{\partial w_{\rho}}{\partial t}(s,t). }\]
Note that its support is contained in $Q'_{i,\rho} \coloneqq [-\frac{7}{8}\rho,\frac{7}{8}\rho]\times [0,T_i]$.
If we replace $w_{\rho}$ in the above formula by the trivial strip $u_{c_i}\colon [-\rho,\rho]\times [0,T_i] \to \R\times Y\colon (s,t)\mapsto (s, c_i(T_i-t))$, it coincides with $(\zeta_{i,\rho})_{(s,t)}$.
(As we fixed in Appendix \ref{subsec-ori-Lag}, $(1,0)$ (resp. $(\sqrt{-1},0)$) in $\C\oplus \C^{n+2}$ corresponds to $\partial_a = \frac{\partial u_{c_i}}{\partial s}$ (resp. $-\partial_z = \frac{\partial u_{c_i}}{\partial t}$).)
Since the $||\cdot ||_{2,\rho,Q'_{i,\rho}}$-norm of the difference between $\rest{w_{\rho}}{Q'_{i,\rho}}$ and the trivial strip converges to $0$ as $\rho \to \infty$,
this implies that
\begin{align}\label{rho-3}
\textstyle{ \lim_{\rho \to \infty} || d_{\rho}\circ \gamma_{4,\rho}(I_{\rho} e_i) - b_{\rho} \circ \gamma_2(e_i) ||_{1,\rho} = \lim_{\rho\to \infty} || \zeta'_{i, \rho} -\zeta_{i,\rho} ||_{1,\rho, Q'_{i,\rho}} =0  . }
\end{align}
\end{enumerate}

Hereafter, we assume $\rho>0$ to be sufficiently large.
We choose a basis $f_1,\dots , f_{m-l-2}$ of $\begin{bmatrix} \bigoplus_{i=1}^l T\mathcal{C}_{m_i+1}  \\ T\mathcal{C}_{l+1} \end{bmatrix}$.
Since the $\R_1$-component of $\gamma_2^{-1} \circ a_{\rho}\circ \gamma_1^{-1}(1)$ is negative (in particular, non-zero),
\[\gamma_2^{-1} \circ a_{\rho}\circ \gamma_1^{-1}(1),e_2,\dots ,e_l,f_1,\dots ,f_{m-l-2}\]
is a basis of $\begin{bmatrix} \bigoplus_{i=1}^l ( T\mathcal{C}_{m_i+1} \oplus \R_i) \\ T\mathcal{C}_{l+1} \end{bmatrix}$.
From the definition of (\ref{isom-exact}), the isomorphism (\ref{isom-sigma1}) is given by
\[\begin{array}{rl}
 &b_{\rho}\circ \gamma_2 (e_2)\wedge \dots \wedge b_{\rho}\circ \gamma_2( e_l) \wedge b_{\rho}\circ \gamma_2(f_1)\wedge \dots \wedge b_{\rho}\circ \gamma_2(f_{m-l-2}) \\
 \mapsto & \gamma_1^{-1}(1) \otimes \left(  a_{\rho} \circ \gamma^{-1}_1(1)\wedge \gamma_2(e_2) \wedge \dots \wedge \gamma_2(e_l)\wedge \gamma_2(f_1) \wedge \dots \wedge \gamma_2(f_{m-l-2}) \right) . \end{array} \]
and this agrees with the map determined by
\[\begin{array}{rl}
 &b_{\rho} \circ \gamma_2 (e_2)\wedge\dots \wedge b_{\rho}\circ \gamma_2( e_l) \wedge b_{\rho} \circ \gamma_2(f_1)\wedge \dots \wedge b_{\rho} \circ \gamma_2(f_{m-l-2}) \\
 \mapsto & -A_{\rho}\cdot \gamma_1^{-1}(1) \otimes \left( \gamma_2(e_1)\wedge \gamma_2(e_2) \wedge \dots \wedge \gamma_2(e_l)\wedge \gamma_2(f_1) \wedge \dots \wedge \gamma_2(f_{m-l-2}) \right)  \end{array} \]
for some $A_{\rho}>0$. (Here, we identify $V_2^*$ and $W_2^*$ in (\ref{isom-exact}) with $V_2$ and $W_2$ via the isomorphism (\ref{isom-dual}), which is defined up to multiplication by a positive scalar.)
On the other hand, since the $\R_1$-component of $\gamma_{4,\rho}^{-1} \circ c_{\rho}\circ \gamma_{3,\rho}(1)$ is positive, the isomorphism (\ref{isom-exact}) induced by the exact sequence (\ref{seq-Gamma}) is given by
\begin{align}\label{isom-Gamma'}
\begin{array}{rl}
& d_{\rho} \circ \gamma_{4,\rho} (e_2)\wedge \dots \wedge d_{\rho} \circ \gamma_{4,\rho}(e_l) \wedge d_{\rho} \circ \gamma_{4,\rho}(f_1)\wedge \dots \wedge d_{\rho} \circ \gamma_{4,\rho} (f_{m-l-2}) \\
\mapsto & A'_{\rho} \cdot \gamma_{3,\rho} (1) \otimes  \left( \gamma_{4,\rho} (e_1) \wedge \gamma_{4,\rho}(e_2)\wedge \dots \wedge \gamma_{4,\rho}(e_l) \wedge \gamma_{4,\rho}(f_1)\wedge \dots \wedge \gamma_{4,\rho}(f_{m-l-2}) \right) 
\end{array}
\end{align}
for some $A'_{\rho} >0$.

Since (\ref{isom-gamma1-4}) is the isomorphism induced by $\gamma_1,\gamma_2 ,\gamma_{3,\rho}$ and $\gamma_{4,\rho}$, the composite map $ (\ref{isom-gamma1-4}) \circ (\ref{isom-sigma1})$ is given by
\[\begin{array}{rl}
 &b_{\rho} \circ \gamma_2 (e_2)\wedge\dots \wedge b_{\rho}\circ \gamma_2( e_l) \wedge b_{\rho} \circ \gamma_2(f_1)\wedge \dots \wedge b_{\rho} \circ \gamma_2(f_{m-l-2}) \\
 \mapsto & -A_{\rho} \cdot \gamma_{3,\rho} (1) \otimes \left( \gamma_{4,\rho} (e_1)\wedge \gamma_{4,\rho} (e_2) \wedge \dots \wedge \gamma_{4,\rho} (e_l)\wedge \gamma_{4,\rho} (f_1) \wedge \dots \wedge \gamma_{4,\rho} (f_{m-l-2}) \right) . \end{array} \]
Consider the subspace of $\begin{bmatrix} \bigoplus_{i=1}^l ( T\mathcal{C}_{m_i+1} \oplus \R_i) \\ T\mathcal{C}_{l+1} \end{bmatrix}$ spanned by $e_2,\dots ,e_l,f_1,\dots ,f_{m+l-2}$. We note that the operator norm of the inverse of $b_{\rho}\circ \gamma_2$ restricted on this subspace is bounded uniformly with respect to $\rho$.
On this subspace, we define an automorphism $\Phi_{\rho}\coloneqq(b_{\rho}\circ \gamma_2)^{-1}\circ (d_{\rho} \circ \gamma_{4,\rho})$, then
 (\ref{rho-2}), (\ref{rho-2'}) and (\ref{rho-3}) show that
\[ \Phi_{\rho}(I_{\rho} e_2)\wedge \dots \Phi_{\rho}(I_{\rho} e_l) \wedge \Phi_{\rho} (f_1)\wedge \dots \wedge \Phi_{\rho}(f_{m-l-2}) = B_{\rho}\cdot e_2\wedge \dots \wedge e_l \wedge f_1\wedge \dots \wedge f_{m-l-2}\]
for $B_{\rho}>0$ such that $\lim_{\rho \to \infty} B_{\rho} =1$. Therefore, $ (\ref{isom-gamma1-4}) \circ (\ref{isom-sigma1})$ is equal to the map given by
\[
\begin{array}{rl}
& d_{\rho} \circ \gamma_{4,\rho} (e_2)\wedge \dots \wedge d_{\rho} \circ \gamma_{4,\rho}(e_l) \wedge d_{\rho} \circ \gamma_{4,\rho}(f_1)\wedge \dots \wedge d_{\rho} \circ \gamma_{4,\rho} (f_{m-l-2}) \\
\mapsto & -C_{\rho}\cdot \gamma_{3,\rho} (1) \otimes  \left( \gamma_{4,\rho} (e_1) \wedge \gamma_{4,\rho}(e_2)\wedge \dots \wedge \gamma_{4,\rho}(e_l) \wedge \gamma_{4,\rho}(f_1)\wedge \dots \wedge \gamma_{4,\rho}(f_{m-l-2}) \right) ,
\end{array}
\]
where $C_{\rho}\coloneqq A_{\rho}B_{\rho}(I_{\rho})^{-(l-1)}>0$.
The assertion of the lemma follows by comparing this map with (\ref{isom-Gamma'}).
%
%
\end{proof}
By Lemma \ref{lem-sigma-tau-1}, the sign of  $ (\ref{isom-gamma1-4}) \circ (\ref{isom-sigma1})$ is $(-1)^1$. Therefore, we can compute that
\begin{align*}
\tau_1 \equiv (\sigma_1 + \sigma_2+ 1) +1 \equiv m + \sum_{i=1}^l  (n-1)(|c_i|+1) \mod 2.
\end{align*}

\textit{ The case of}  (\ref{boundary-2}).
For short, we denote $\mathcal{M}_{\Lambda_-,A_-}(c;b_j,\dots ,b_k)$ by $\mathcal{M}_-$, $\bar{\mathcal{M}}_{\Lambda_-,A_-}(c;b_j,\dots ,b_k)$ by $\bar{\mathcal{M}}_-$ and $\mathcal{M}_{L,C}(a; b_1,\dots , b_{j-1}, c, b_{k+1},\dots ,b_m)$ by $\mathcal{M}$.
Let $T\coloneqq \int c^* (dz-\lambda)$ denote the action of the Reeb chord $c$.
Let $w\in \mathcal{M}^1$ be a $\tilde{J}$-holomorphic curve obtained by gluing $v\in \mathcal{M}$ and $[u]\in \bar{\mathcal{M}}_-$.
By a linear gluing argument, there exists an exact sequence
\begin{align}\label{seq-glue-2}
\xymatrix{
0\ar[r] & \ker \Dbar_w \ar[r] & {\begin{bmatrix} \ker \Dbar_u \\ \ker \Dbar_v \end{bmatrix}} \ar[r] &  {\begin{bmatrix} \coker \Dbar_u \\ \R \\ \coker \Dbar_v \end{bmatrix}} \ar[r] & \coker \Dbar_w \ar[r] & 0.
}\end{align}
Its construction is similar to \cite[Remark 3.3]{EES-ori} and here let us identify any $r \in \R$ with a $\C\oplus \C^{n+2}$-valued $1$-form $r \cdot \zeta_0$ on the gluing strip, where
\begin{align}\label{anti-1-form'}
(\zeta_0 )_{(s,t)} \colon \partial_s \mapsto (-1,0),\ \partial_t \mapsto (\sqrt{-1},0)
\end{align}
for every $(s,t)\in \R\times [0,T]$.
Note that $\ker \Dbar_v=0$ since $\mathcal{M}$ is cut out transversely and has dimension $0$.
We have an isomorphism (\ref{isom-exact}) induced by the exact sequence (\ref{seq-glue-2})
\begin{align}\label{isom-sigma4}
\det \Dbar_w \to  \ker \Dbar_u \otimes  \topwedge (\coker \Dbar_u)^* \otimes \R^* \otimes  \topwedge (\coker \Dbar_v)^* 
\end{align}
The orientations of $\det\Dbar_w$, $\det \Dbar_u =\topwedge \ker \Dbar_u\otimes \topwedge (\coker \Dbar_u)^*$ and $\det \Dbar_{v}= \topwedge (\coker \Dbar_{v})^*$ are determined as explained in Appendix \ref{subsec-ori-Lag}.
Let $(-1)^{\sigma_3}$ be the sign of the isomorphism (\ref{isom-sigma4}). If we write $m' = k-j+1$ for short, then
\[\sigma_3 \equiv  (n-1)(|c|+1) + (m-m'-1) +(m'+1)(j+1)  + \sum_{i=1}^{j-1} |b_i| \mod 2.\]
This is computed as in \cite[(42)]{K} by comparing the two exact sequences below which come from gluing the capping operators:
%
\begin{align}\label{seq-2-1}
\xymatrix@C=20pt@R=10pt{
\ker\Dbar_{\tilde{T}} \ar[r] & {\begin{bmatrix} {\begin{bmatrix} \bigoplus_{i=j}^k \ker \Dbar_{b_i,-} \\ \ker \Dbar_{c,+} \\ \ker \Dbar_{u} \end{bmatrix}} \\ 
 {\begin{bmatrix} \bigoplus_{i=1}^{j-1} \ker \Dbar_{b_i,-}  \\ \ker \Dbar_{c,-} \\\bigoplus_{i=k+1}^m \ker \Dbar_{b_i,-} \\ \ker \Dbar_{a,+} \\ \ker \Dbar_{v} \end{bmatrix}} \end{bmatrix}} \ar[r] 
 &  {\begin{bmatrix} {\begin{bmatrix} \bigoplus_{i=j}^k \coker \Dbar_{b_i,-} \\ \coker \Dbar_{c,+} \\ \coker \Dbar_{u} \end{bmatrix}} \\ \R \\ \R^{n+2} \\
 {\begin{bmatrix} \bigoplus_{i=1}^{j-1} \coker \Dbar_{b_i,-}  \\ \coker \Dbar_{c,-} \\\bigoplus_{i=k+1}^m \coker \Dbar_{b_i,-} \\ \coker \Dbar_{a,+} \\ \coker \Dbar_{v} \end{bmatrix}} \end{bmatrix}} \ar[r] & \coker \Dbar_{\tilde{T}}, }
\end{align}
\begin{align}\label{seq-2-2}
\xymatrix@C=20pt@R=10pt{
\ker\Dbar_{\tilde{T}} \ar[r] & {\begin{bmatrix} {\begin{bmatrix}\ker \Dbar_{c,+}\\ \R' \\ \ker\Dbar_{c,-}  \end{bmatrix}} \\ 
 {\begin{bmatrix} \bigoplus_{i=1}^{m} \ker \Dbar_{b_i,-}  \\ \ker \Dbar_{a,+} \\ \ker\Dbar_{u} \\ \ker \Dbar_{v} \end{bmatrix}} \end{bmatrix}} \ar[r] 
 &  {\begin{bmatrix} {\begin{bmatrix}\coker \Dbar_{c,+}\\  \coker\Dbar_{c,-}  \end{bmatrix}} \\ 
\R' \\ \R^{n+2} \\ {\begin{bmatrix} \bigoplus_{i=1}^{m} \coker \Dbar_{b_i,-}  \\ \coker \Dbar_{a,+} \\ \coker\Dbar_{u} \\ \R \\ \coker \Dbar_{v} \end{bmatrix}} \end{bmatrix}} 
 \ar[r] & \coker \Dbar_{\tilde{T}}.
} 
\end{align}
Here, the first and the last zero maps are omitted and $\Dbar_{\tilde{T}}$ is a $\Dbar$-operator on the disk $D$ defined in the proof of \cite[Theorem 2.5]{K}.
The capping operators $\Dbar_{s,c',\pm}$ for any $c'\in \mathcal{R}(\Lambda)$ are written by $\Dbar_{c',\pm}$. In (\ref{seq-2-2}), $\R'\coloneqq \R$ in the second column is mapped by the identity to $\R'$ in the third column.
$\sigma_3$ can be calculated by the following steps: 
\begin{itemize}
\item We arrange the sequence (\ref{seq-2-2}). In both the second and the third columns, we move $\R'$ to the top. It costs the sign $(-1)^{\sigma_{3,1}}$, where $\sigma_{3,1}=n+1$. Moreover, we switch the places of $\ker \Dbar_{c,+}$ and $\ker \Dbar_{c,-}$ in the second column and also switch the places of $\coker \Dbar_{c,+}$ and $\coker \Dbar_{c,-}$ in the third column. This costs the sign $(-1)^{\sigma_{3,2}}$, where $\sigma_{3,2}=n\cdot |c|$. Then, (\ref{seq-2-2}) is changed to the following form:
\begin{align}\label{seq-2-3}
\xymatrix@C=20pt@R=10pt{
\ker\Dbar_{\tilde{T}} \ar[r] & {\begin{bmatrix} \R' \\ {\begin{bmatrix}\ker \Dbar_{c,-}\\ \ker\Dbar_{c,+}  \end{bmatrix}} \\ 
 {\begin{bmatrix} \bigoplus_{i=1}^{m} \ker \Dbar_{b_i,-}  \\ \ker \Dbar_{a,+} \\ \ker\Dbar_{u} \\ \ker \Dbar_{v} \end{bmatrix}} \end{bmatrix}} \ar[r] 
 &  {\begin{bmatrix} \R' \\ {\begin{bmatrix}\coker \Dbar_{c,-}\\  \coker\Dbar_{c,+}  \end{bmatrix}} \\ 
 \R^{n+2} \\ {\begin{bmatrix} \bigoplus_{i=1}^{m} \coker \Dbar_{b_i,-}  \\ \coker \Dbar_{a,+} \\ \coker\Dbar_{u} \\ \R \\ \coker \Dbar_{v} \end{bmatrix}} \end{bmatrix}} 
 \ar[r] & \coker \Dbar_{\tilde{T}}.
} 
\end{align}
By the argument in Remark \ref{rem-cancel}, we may remove $\R'$ from the second and the third columns.
\item Next, we arrange the exact sequence (\ref{seq-2-1}). In the second column, we switch the places of $\bigoplus_{i=j}^k \ker \Dbar_{b_i,-}$ and $\ker \Dbar_{c,-}$. This costs the sign $(-1)^{\sigma_{3,3}}$, where $\sigma_{3,3} = k(j+1) +1$. In the third column, we switch the places of $\bigoplus_{i=j}^k \coker \Dbar_{b_i,-}$ and $\coker \Dbar_{c,-}$. This costs the sign $(-1)^{\sigma_{3,4}}$, where $\sigma_{3,4}=|c| + (k-j+1) + \sum_{i=1}^{j-1} |b_i|$.
\item After the above arrangements of (\ref{seq-2-1}), we move $\ker \Dbar_u$ in the second column to the place just below $\ker \Dbar_{a,+}$ and also move $\begin{bmatrix} \coker \Dbar_u \\ \R \end{bmatrix}$ in the third column to the place just below $\coker \Dbar_{a,+}$. This costs the sign $(-1)^{\sigma_{3,5}}$, where $\sigma_{3,5} = m$. After this arrangement, the exact sequence agrees with (\ref{seq-2-3}) with $\R'$ removed. 
\item If we set $m' \coloneqq (k-j+1)$, then $\sigma_3$ is given by
\[\sigma_{3,1} + \sigma_{3,2}+\sigma_{3,3} + \sigma_{3,4} +\sigma_{3,5} \equiv  (n-1)(|c|+1) + (m-m'-1) +(m'+1)(j+1)  + \sum_{i=1}^{j-1} |b_i| \mod 2. \]
\end{itemize}

We want to relate (\ref{seq-glue-2}) to the exact sequence (\ref{seq-Gamma}) for $w \in \mathcal{M}^1$. 
Let us consider the following isomorphisms:
\[\xymatrix@R=10pt{
 {\begin{bmatrix} \ker \Dbar_{u} \\\ker \Dbar_v \end{bmatrix}} \ar[r]^-{\gamma'_1} &   {\begin{bmatrix} T_u \mathcal{M}_- \\ T_v \mathcal{M} \end{bmatrix}} & {\begin{bmatrix} \R\oplus  T_{[u]}\bar{\mathcal{M}}_- \\   T_v \mathcal{M} \end{bmatrix}}  \ar[l]_-{(\ref{isom-slice})} \ar[r]^-{\gamma'_3} & T_w \mathcal{M}^1 , \\
{\begin{bmatrix} \coker \Dbar_{u} \\ \R \\ \coker \Dbar_v \end{bmatrix}} & {\begin{bmatrix}  T\mathcal{C}_{m'+1} \\ \R \\ T\mathcal{C}_{m-m'+2} \end{bmatrix}}  \ar[l]_-{\gamma'_2} \ar[r]^-{\gamma'_4} & T\mathcal{C}_{m+1}.
}
\]
These maps are defined as follows:
\begin{itemize}
\item $\gamma'_1$ and $\gamma'_2$ are combinations of $\id_{\R}\colon \R \to \R$ and 
those maps from (\ref{seq-Gamma}) for $u\in \mathcal{M}_-$ and $v\in \mathcal{M}$. 
To see that they are isomorphisms, note that $\mathcal{M}$ and $\bar{\mathcal{M}}_-$ are $0$-dimensional manifolds.
We recall that the isomorphism (\ref{isom-exact}) induced by (\ref{seq-Gamma}) preserves the orientations and (\ref{isom-slice}) changes the orientations by the sign $(-1)^1$ by Proposition \ref{prop-rev-ori}.
\item $\gamma'_3$ is the linearization of a map defined by gluing $v$ and $[u]$. (Here, $1\in \R$ corresponds to $\frac{d}{d\rho}\in T_{\rho}\R$, where $\rho \gg 0$ is a gluing parameter.)
It has the sign $(-1)^{\tau_2}$. Indeed, if $w_{\rho}\in \mathcal{M}^1$ is obtained from a pre-gluing of $\tau_{-\rho}\circ u$ for $\rho \gg 0$ and $v$, then $w_{\rho}$ converges to the boundary point as $\rho\to \infty$.
Thus, a positive vector in $\R$ is mapped to an outward vector in $T_{w}\mathcal{M}^1$.
\item $\gamma'_4$ is the linearization of a map defined by gluing punctured disks with conformal structures. See \cite[Subsection 4.2.2]{EES-ori}. It has the sign $(-1)^{\sigma_4}$, where
\[\sigma_4 \equiv ((m'-1)j+1) + (m'-2) \mod 2.\]
Indeed, $(-1)^{(m'-1)j+1}$ is the sign of the isomorphism ${\begin{bmatrix} \R \\ T\mathcal{C}_{m'+1} \\ T\mathcal{C}_{m-m'+2} \end{bmatrix}} \to T\mathcal{C}_{m+1}$, which follows the convention of gluing in \cite[Subsection 4.2.2]{EES-ori}, and it is computed by \cite[Lemma 4.7]{EES-ori}. $(-1)^{m'-2}$ is the cost of sign for switching the places of $\R$ and $T\mathcal{C}_{m'+1}$.
\end{itemize}
Then, $\gamma'_1,\dots ,\gamma'_4$ and (\ref{isom-slice}) induce an isomorphism
\begin{align}\label{isom-gamma'-1-4}
\ker \Dbar_u \otimes \left( \topwedge (\coker \Dbar_u)^* \otimes \R^* \otimes  \topwedge (\coker \Dbar_v)^* \right) \to \topwedge T_w\mathcal{M}^1 \otimes \topwedge (T\mathcal{C}_{m+1})^*
\end{align}
and the sign of this isomorphism is $(-1)^{1 + \tau_2 +\sigma_4}$.
The composite map
\begin{align*}
(\ref{isom-gamma'-1-4}) \circ (\ref{isom-sigma4}) \colon \det \Dbar_w \to \topwedge T_w\mathcal{M}^1 \otimes \topwedge (T\mathcal{C}_{m+1})^*
\end{align*}
has the sign $(-1)^{\sigma_3 + (1+\tau_2+\sigma_4)}$.
\begin{lem}\label{lem-sigma-tau-2}
If $w\in \mathcal{M}^1$ is sufficiently close to the boundary point $(v,[u])$,
$(\ref{isom-gamma'-1-4}) \circ (\ref{isom-sigma4})$ coincides with the isomorphism (\ref{isom-exact}) induced by the exact sequence (\ref{seq-Gamma}) for $w$ up to multiplication by a positive scalar. In particular, $(\ref{isom-gamma'-1-4})\circ (\ref{isom-sigma4})$ preserves the orientations.
\end{lem}
\begin{proof}
We prepare several notations.
We define
\[ \begin{array}{rl}
D_{m'+1}(\rho)\coloneqq & D_{m'+1}\setminus \psi_0 ((T^{-1}\rho ,\infty)\times [0,1] ),\\
D_{m-m'+2}(\rho) \coloneqq& D_{m-m'+2}\setminus \psi_j((-\infty,-T^{-1}\rho)\times [0,1]),
\end{array}\]
for $\rho>0$.
By identifying $\psi_0(T^{-1}\rho,t)\in D_{m'+1}(\rho)$ with $\psi_j(-T^{-1}\rho,t) \in D_{m-m'+2}(\rho) $ for every $t\in[0,1]$,
we obtain a disk $D_{m+1}(\rho)$ with $m+1$ boundary punctures.
Let $Q_{\rho}$ denote the image of the embedding
\[   [-\rho ,\rho ] \times [0,T] \to D_{m+1}(\rho) \colon (s,t) \mapsto \begin{cases} \psi_j (T^{-1} (s-\rho)  ,T^{-1} t) \in D_{m-m'+2}(\rho) & \text{ if }s\in [0,\rho] , \\ \psi_0(T^{-1} (s+\rho) , T^{-1} t) \in D_{m'+1}(\rho) & \text{ if }s\in [-\rho,0]. \end{cases}\]
By a similar argument as in \cite[Section 10.4]{CELN} about gluing, 
 $w_{\rho}\in \mathcal{M}^1$ is obtained from a pre-gluing of $\tau_{-\rho}\circ u$ and $v$ when $\rho>0$ is sufficiently large. $(w_{\rho})_{\rho \gg 0}$ converges to the boundary point $(v,[u])$ as $\rho \to \infty$.

We rewrite $\gamma'_3$ and $\gamma'_4$ for $w=w_{\rho}$ by $\gamma'_{3,\rho}$ and $\gamma'_{4,\rho}$. The isomorphisms $\gamma'_1,\gamma'_2, \gamma'_{3,\rho},\gamma'_{4,\rho}$ and the exact sequences (\ref{seq-glue-2}) and (\ref{seq-Gamma}) for $w_{\rho}\in \mathcal{M}^1$ are combined to the following diagram:
\begin{align*}
\xymatrix@C=20pt@R=15pt{
 0 \ar[r] & \ker \Dbar_{w_{\rho}} \ar[r] & \ker \Dbar_u \ar[d]_-{\gamma'_1} \ar[r]^-{a'_{\rho}} & {\begin{bmatrix} \coker \Dbar_u \oplus \R \\ \coker \Dbar_v \end{bmatrix}} \ar[r] ^-{b'_{\rho}}& \coker \Dbar_{w_{\rho}} \ar[r] & 0 \colon (\ref{seq-glue-2}) \\
 &  & \R \ar[d]_-{\gamma'_{3,\rho}} &  {\begin{bmatrix} T\mathcal{C}_{m'+1} \oplus \R \\ T\mathcal{C}_{m-m'+2} \end{bmatrix}} \ar[d]_-{\gamma'_{4,\rho}} \ar[u]^-{\gamma'_2} & & & \\
 0 \ar[r] & \ker \Dbar_{w_{\rho}} \ar@{=}[uu] \ar[r] & T_{w_{\rho}} \mathcal{M}^1 \ar[r]^-{c'_{\rho}} & T\mathcal{C}_{m+1} \ar[r]^-{d'_{\rho}} & \coker \Dbar_{w_{\rho}} \ar@{=}[uu] \ar[r] & 0 \colon (\ref{seq-Gamma}). }
\end{align*}
Here, for $\gamma'_1$ and $\gamma'_{3,\rho}$, we omit writing $\ker \Dbar_v$, $T_{[u]} \bar{\mathcal{M}}_-$ and $T_v\mathcal{M}$ since they are $0$-dimensional vector spaces. We also identify $\R$ with $T_u\mathcal{M}_-$ via (\ref{isom-slice}).

We can prove the following three results about this diagram when $\rho$ is sufficiently large. (They are analogous to the proof of Lemma \ref{lem-sigma-tau-1}, so we omit detailed discussions.)
\begin{enumerate}
\item 
$f_0 \coloneqq (\gamma'_1)^{-1}(1)\in \ker \Dbar_u$ is the constant function to $(1,0)\in \C\oplus \C^{n+2}$.
We choose a $C^{\infty}$ function $\mu''_{\rho}\colon D_{m+1} (\rho)\to [0,1]$ which is constant to $1$ on $D_{m'+1}(\frac{\rho}{2})$ and constant to $0$ on $D_{m -m' +2}(\rho)$, and satisfies $|d^k \mu''_{\rho}| = \mathcal{O}(\rho^{-1})$ for $k=1,2$.
In addition, we assume that $\mu''_{\rho}(s,t)$ is decreasing in $s$ and independent of $t$ for $(s,t) \in [-\rho,\rho]\times [0,T] \cong Q_{\rho}$.
Then, $\rest{\mu''_{\rho} \cdot f_0}{D_{m'+1}(\rho)}$ is extended to a function on $D_{m+1}(\rho)$ so that it is constant to $0$ on $ D_{m-m'+2}(\rho)$. This satisfies
\[ \textstyle{ \Dbar (\mu''_{\rho} \cdot f_0) = - \frac{\partial \mu''_{\rho}}{\partial s} \zeta_0 } \]
on $Q_{\rho}$. $\zeta_0$ defined by (\ref{anti-1-form'}) corresponds to $1\in \R$ and the $L^2$-inner product $\la - \frac{\partial \mu''_{\rho}}{\partial s} \zeta_0 ,\zeta_0 \ra_{L^2}$ is positive.
This shows that the $\R$-component of $a'_{\rho}(f_0)$ is positive.
On the other hand, $(\gamma_{4,\rho}')^{-1}\circ c'_{\rho} \circ \gamma'_{3,\rho}(1)$ converges to $ ((0,1) , 0 )\in  {\begin{bmatrix} T\mathcal{C}_{m'+1} \oplus \R \\ T\mathcal{C}_{m-m'+2} \end{bmatrix}} $ as $\rho \to \infty$.
Therefore, the following hold when $\rho$ is sufficiently large:
First, since $a'_{\rho}$ is not the zero map and $\dim \ker \Dbar_u=1$, it follows that $\ker \Dbar_{w_{\rho}} =0$.
Second, the $\R$-components of both $(\gamma'_2)^{-1}\circ a'_{\rho}\circ (\gamma'_1)^{-1}(1)$ and $(\gamma'_{4,\rho})^{-1} \circ c'_{\rho}\circ \gamma'_{3,\rho}(1)$ are positive.
\item For any $V\in T\mathcal{C}_{m'+1}$ and $V' \in T\mathcal{C}_{m-m'+2}$,
\[ \lim_{\rho \to \infty} || d'_{\rho}\circ \gamma'_{4,\rho}(V) - b'_{\rho}\circ \gamma'_2(V) ||_{1,\rho} =  \lim_{\rho \to \infty} ||  d'_{\rho}\circ \gamma'_{4,\rho}(V') - b'_{\rho}\circ \gamma'_2(V') ||_{1,\rho} =0 . \]
\item For $I_{\rho}\in \R$ defined in the proof of Lemma \ref{lem-sigma-tau-1},
\[ \lim_{\rho \to \infty}  ||  d'_{\rho}\circ \gamma'_{4,\rho}(I_{\rho}) - b'_{\rho}\circ \gamma'_2(1) ||_{1,\rho}  =0.\]
\end{enumerate}
%
By taking a basis of ${\begin{bmatrix} T\mathcal{C}_{m'+1} \oplus \R \\ T\mathcal{C}_{m-m'+2} \end{bmatrix}}$, in a way parallel to the proof of Lemma \ref{lem-sigma-tau-1}, we can show that the isomorphism (\ref{isom-exact}) induced by (\ref{seq-Gamma}) is equal to $(\ref{isom-gamma'-1-4})\circ (\ref{isom-sigma4})$ up to multiplication by some positive real number.
\end{proof}

By Lemma \ref{lem-sigma-tau-2}, the sign of $(\ref{isom-gamma'-1-4})\circ (\ref{isom-sigma4})$ is $(-1)^0$. Therefore, we can compute that
\begin{align*}
\tau_2 \equiv \sigma_3+\sigma_4 +1
\equiv  (n-1)(|c|+1)  +(m-m')+\sum_{i=1}^{j-1} |b_i| \mod 2.
\end{align*}
This finishes the proof of Proposition \ref{prop-DGA-ori}.
\end{proof}

\bibliography{reference.bib}

\end{document}